\pdfoutput =1

\documentclass[journal,romanappendices]{IEEEtran}
\IEEEoverridecommandlockouts
\usepackage{cite}
\usepackage{amsmath,amssymb,amsfonts}
\usepackage{algorithmic}
\usepackage{graphicx}
\usepackage{textcomp}
\usepackage{xcolor}
\usepackage{amsthm}
\usepackage{enumitem}
\usepackage{algorithm}
\usepackage{bbm} 
\usepackage{multicol}
\usepackage{multirow}
\usepackage{array}
\usepackage{cases}
\usepackage{subfig}
\usepackage[outdir=./]{epstopdf}
\allowdisplaybreaks

\def\BibTeX{{\rm B\kern-.05em{\sc i\kern-.025em b}\kern-.08em
    T\kern-.1667em\lower.7ex\hbox{E}\kern-.125emX}}
    
\newtheorem{theorem}{Theorem}
\newtheorem{definition}{Definition}
\newtheorem{lemma}{Lemma}

\newtheorem{corollary}{Corollary}[theorem]
\newtheorem{assumption}{Assumption}   
\DeclareMathOperator*{\argmin}{arg\,min}

\begin{document}

\makeatletter
\newcommand{\thickhline}{%
    \noalign {\ifnum 0=`}\fi \hrule height 1pt
    \futurelet \reserved@a \@xhline
}
\newcolumntype{"}{@{\hskip\tabcolsep\vrule width 1pt\hskip\tabcolsep}}
\makeatother

\title{{Variance reduction for Riemannian non-convex optimization with batch size adaptation}\\
\thanks{Andi Han and Junbin Gao are with the Discipline of Business Analytics, University of Sydney Business School, University of Sydney, Australia. (Email: andi.han@sydney.edu.au, junbin.gao@sydney.edu.au). }
}

\author{Andi Han, Junbin Gao}

\maketitle
\pagestyle{headings}
\pagenumbering{arabic}

\begin{abstract}
    Variance reduction techniques are popular in accelerating gradient descent and stochastic gradient descent for optimization problems defined on both Euclidean space and Riemannian manifold. In this paper, we further improve on existing variance reduction methods for non-convex Riemannian optimization, including R-SVRG and R-SRG/R-SPIDER with batch size adaptation. We show that this strategy can achieve lower total complexities for optimizing both general non-convex and gradient dominated functions under both finite-sum and online settings. As a result, we also provide simpler convergence analysis for R-SVRG and improve complexity bounds for R-SRG under finite-sum setting. Specifically, we prove that R-SRG achieves the same near-optimal complexity as R-SPIDER without requiring a small step size. Empirical experiments on a variety of tasks demonstrate effectiveness of proposed adaptive batch size scheme.  
\end{abstract}

\begin{IEEEkeywords}
Riemannian Optimization; Non-convex Optimization; Online Optimization; Stochastic Variance Reduction; Batch Size Adaptation; 
\end{IEEEkeywords}

\section{Introduction}
Consider the following online and finite-sum optimization problems defined on a smooth Riemannian manifold $\mathcal{M}$.
\begin{equation}
    \min_{x \in \mathcal{M}} f(x) := \begin{cases} \mathbb{E}[f(x; \omega)], & \text{ online}\\
    \frac{1}{n} \sum_{i=1}^n f_i(x), & \text{ finite-sum } \end{cases} \label{problem_formulation_begin}
\end{equation}
where $f: \mathcal{M} \xrightarrow[]{} \mathbb{R}$ is a smooth real-valued non-convex function. The finite-sum formulation of minimizing an empirical average of $n$ component functions is a special type of online optimization where $\omega$ can be finitely sampled. For some cases, $n$ can be large or possibly infinite and only streaming stochastic gradients are available. This corresponds to the online problem with $\omega$ indexed by $i$. Hence, for notation clarity, we only consider the case $f(x) := \frac{1}{n}\sum_{i=1}^n f_i(x)$ and refer to it as finite-sum or online optimization depending on size of $n$. Problem \eqref{problem_formulation_begin} encompasses a great variety of machine learning applications, including principal component analysis \cite{SatoRSVRG2019}, low rank matrix completion \cite{BoumalTR_LRMC2011}, Riemannian centroid computation \cite{YuanLBFGS_mean2016}, independent component analysis \cite{TheisICA2009} and so forth. 

Some traditional solutions treat \eqref{problem_formulation_begin} as a nonlinearly constrained problem and therefore projected gradient based methods are readily applied. However, for some manifold types, particularly positive definite matrices, the projection operator can be expensive to evaluate \cite{ZhangRSVRG2016}. Also, because this class of methods ignores geometry of the search space, they are often outperformed by methods that respect geometric structure of manifolds \cite{YuanLBFGS_mean2016}. 

For these reasons, a growing interest is on solving \eqref{problem_formulation_begin} directly over the manifold space via Riemannian gradient based methods. Two basic solutions are Riemannian steepest descent (R-SD) \cite{UdristeRMGD1994} and Riemannian stochastic gradient descent (R-SGD) \cite{BonnabelRMSGD2013}. Although R-SD enjoys a faster convergence rate $\mathcal{O}(1/T)$ than $\mathcal{O}(1/\sqrt{T})$ of R-SGD for non-convex optimization \cite{BoumalRTR_global2019,HosseiniSGD2019}, R-SD requires a full pass over $n$ component functions per iteration. This computation is extremely costly when $n$ is large, thereby prohibiting its applicability for online optimization. Recent improvement on R-SD is limited to extending Nesterov acceleration to Riemannian manifold for geodesic (strongly) convex functions \cite{AhnNAG2020,ZhangNAG2018,AlimisisPNAG2020}. For general non-convex functions, it is unknown whether faster convergence guarantee is maintained. On the other hand, despite with higher per-iteration efficiency, R-SGD suffers from high gradient variance, similar to its Euclidean version. Therefore it relies on a decaying step size to ensure convergence \cite{HosseiniSGD2019}. A current line of research focuses on adapting gradient and step size of R-SGD, motivated by the success of adaptive methods on Deep Learning applications. In particular, some successful efforts have been made to generalize Adagrad, Adam and Rmsprop to manifold optimization \cite{KumarRMRmsprop2018, KasaiRmsprop2019, KasaiRMadam2020,BecigneulRMadam2018}. These methods can be viewed as preconditioned R-SGD and do not theoretically outperform R-SGD with better complexity. 

To improve on R-SD and R-SGD and achieve lower total complexity, many studies leverage variance reduction techniques from unconstrained optimization in the Euclidean space. Zhang \emph{et al}. \cite{ZhangRSVRG2016} proposed Riemannian stochastic variance reduced gradient (R-SVRG) based on the ideas in \cite{ReddiSVRG2016,JohnsonSVRG2013}, Sato \emph{et al}. \cite{SatoRSVRG2019} further developed a more general R-SVRG with retraction and vector transport. By occasionally evaluating full gradient of a reference point, R-SVRG allows a larger step size and hence converges faster particularly around optimal point. But on manifold space, when the reference point is far from current iterates, the use of vector transport can incur unintended distortion. Therefore, inspired by the work of \cite{NguyenSRG2017}, Kasai \emph{et al}. \cite{KasaiRSRG2018} introduced Riemannian stochastic recursive gradient (R-SRG) that transports gradients between consecutive iterates. More recently, Zhang \emph{et al}. \cite{ZhangRSPIDER2018} and Zhou \emph{et al}. \cite{ZhouRSPIDER2019} independently proposed Riemannian stochastic path integrated differential estimator (R-SPIDER) that hybrids the same recursive gradient estimator with normalized gradient descent as in \cite{FangSPIDER2018}. They showed that R-SPIDER achieves the near-optimal complexity similar to its vector space counterpart. Other related work includes \cite{TripuraneniASGD2018,BabanezhadMASAGA2018} where Polyak iterate averaging \cite{PolyakIA1992} and SAGA \cite{DefazioSAGA2014} are also generalized to Riemannian manifolds. However, their analysis is limited to retraction or geodesic (strongly) convex functions. 

With all these promising results of variance reduction, a natural question to ask is whether their complexities can be further improved. A common feature among these methods is periodic computations of full batch gradient, which potentially limits convergence speed particularly during early stage of training. This is because at early stage, stochastic gradients are pointing to similar directions and therefore it is unnecessary to use exact gradients to correct for deviations \cite{BallesSGDadapt2016}. While approaching optimal point, larger batch gradient becomes increasingly important to reduce variance of stochastic gradients. Furthermore, Keskar \emph{et al}. \cite{KeskarLB_deeplearn2016} showed that gradient noise at the outset of training helps escape sharp minima, leading to higher generalization power. Therefore, a reasonable strategy is to gradually increase the batch size throughout optimization path. 

Increasing batch size of SGD is often viewed as an approach to reduce variance so that step size decay is no longer necessary \cite{BallesSGDadapt2016,SmithSGDAdapt2017}. This is usually achieved by pre-specifying a strategy for batch size increase, such as exponential or linear \cite{FriedlanderHSGD2012, ZhouHSGD2018}. Alternatively, adaptively changing the batch size based on gradient variance or model quality often yields improved convergence rates \cite{DeSGDadapt2017,BallesSGDadapt2016,SievertSGDadapt2019}. For variance reduction methods, Harikandeh \emph{et al}. \cite{HarikandehP_SVRG2015} proved that SVRG is robust to inexact gradient at reference point provided that batch size is increasing. Still, they adopted an exponential increase scheme for practical applications. A recent work \cite{JiAdaSPIDERboost2019} provided a unified batch adaptation strategy for variance reduction methods, including SVRG and SRG \cite{NguyenSRG2017}.\footnote{In \cite{JiAdaSPIDERboost2019}, the authors use the term SPIDERboost, which is an improved version of SPIDER. It is noticed that SPIDERboost shares nearly identical formulation as non-convex SRG \cite{NguyenRSRG_FS2019}.} They showed that the same iteration complexities can be guaranteed with potentially fewer per-iteration gradient evaluations. Their experiment results verify the effectiveness of adaptive batch size strategy, which not only outperforms vanilla variance reduction methods, but also dominates some predetermined batch size increase schemes. Motivated by this work, we aim to examine the potential of batch size adaptation on Riemannian optimization problems and improve on state-of-the-art variance reduction methods, including R-SVRG, R-SRG and R-SPIDER.

Our \textbf{main contribution} lies in the following five aspects:
\begin{itemize}
    \item We propose new variance reduction methods with adaptive batch size for non-convex manifold optimization. We provide convergence analysis in terms of retraction and vector transport for both general non-convex functions and gradient dominated functions (see Definition \ref{gradient_dominate_def}). We focus on general mini-batch versions of R-SVRG and R-SRG, which allows more flexible choices of the step size and inner loop size.
    \item We demonstrate that adaptive batch size R-SVRG and R-SRG can preserve the same iteration complexities as their original methods while per-iteration complexities can be potentially reduced. This suggests that total complexity can be much lower for both general non-convex and gradient dominated functions.
    \item Current analysis for non-convex R-SRG \cite{KasaiRSRG2018} is limited to single loop convergence, which is suboptimal compared with R-SPIDER under finite-sum setting. Our analysis of R-SRG with batch size adaptation can be directly applied to its vanilla version. We prove that the same near-optimal complexity can be achieved by R-SRG without requiring a small step size as in R-SPIDER. 
    \item  For R-SVRG, Sato \emph{et al}. \cite{SatoRSVRG2019} only proved convergence under retraction strongly convex functions and Zhang \emph{et al}. \cite{ZhangRSVRG2016} proved convergence under non-convex functions but restricted to exponential map and parallel transport. Therefore, we first complete convergence analysis of non-convex R-SVRG with retraction and vector transport following the standard analysis of Lyapunov function. Then we show that our proof of adaptive version can be trivially generalized to R-SVRG. This new analysis turns out to be much simpler without the need to construct a proper Lyapunov function and avoids using trigonometric distance bound (see Lemma \ref{trig_distance_bound_lemma}) as a result. 
    \item Finally, our experiment results demonstrate superiority of adaptive batch size scheme over a number of applications.
\end{itemize}

\section{Preliminaries}

Before presenting our algorithms, some preliminary concepts and notations should be defined. Problem \eqref{problem_formulation_begin} requires to find a critical point of objective function with Riemannian manifold constraint. By exploiting intrinsic properties of Riemannian manifold, the problem can be regarded as unconstrained optimization over manifold space. 

A Riemannian manifold $\mathcal{M}$ is endowed with a smooth inner product $\langle \cdot, \cdot \rangle_x$ on tangent space ${T}_x\mathcal{M}$ for every $x \in \mathcal{M}$. It naturally induces norm $\|u\|_x := \sqrt{\langle u, u \rangle_x}$ for $u \in T_x\mathcal{M}$. A geodesic curve $\gamma: [0,1] \xrightarrow[]{} \mathcal{M}$ is a locally shortest path with zero acceleration. The exponential map $\text{Exp}_x: T_x\mathcal{M} \xrightarrow{} \mathcal{M}$ maps a tangent vector $u \in T_x\mathcal{M}$ along the geodesic leading to $y = \text{Exp}_x(u) \in \mathcal{M}$ such that $\gamma(0) = x, \gamma(1) = y, \dot{\gamma}(0) = \frac{d}{dt} \gamma(t) = u$. If for any two points $x,y \in \mathcal{U} \subset \mathcal{M}$, there exists a unique geodesic connecting them, exponential map has well-defined inverse $\text{Exp}_x^{-1} : \mathcal{M} \xrightarrow[]{} T_x\mathcal{M}$. Subset $\mathcal{U}$ is often called normal neighbourhood of $x$ and its size is lower bounded by injectivity radius \cite{AbsilMatrixManifold2009}. Accordingly, Riemannian distance $d(x,y) = \| \text{Exp}_x^{-1}(y) \|_x = \| \text{Exp}_y^{-1}(x) \|_y$ is a minimizing distance between $x, y$. Parallel transport $\Gamma_x^y: T_x\mathcal{M} \xrightarrow[]{} T_y\mathcal{M}$ relates tangent vectors on disjoint tangent spaces by transporting them along the geodesic $\gamma$ while preserving norm and inner products. 

For certain manifolds, exponential map and parallel transport can be expensive to evaluate or even do not exist. Indeed, Stiefel and fixed-rank manifold have no closed form for parallel transport \cite{SatoRSVRG2019}. In this paper, our analysis focuses on more general and efficient retraction and vector transport. Retraction $R_x: T_x\mathcal{M} \xrightarrow{} \mathcal{M}$ is the first-order approximation of exponential map and maps a tangent vector $\xi$ to $z = R_x(\xi)$ such that $R_x(0) = x$ and $\text{D} R_x(0)[\xi] = \xi$. Similarly, we can define retraction curve $c(t) := R_x(t\xi)$ and inverse retraction map $R_x^{-1}: \mathcal{M} \xrightarrow[]{} T_x\mathcal{M}$ if $R$ has smooth bijection. Vector transport $\mathcal{T}_x^z v$ or equivalently $\mathcal{T}_\xi v$ with $z = R_x(\xi)$ transports $v \in T_x\mathcal{M}$ along the retraction curve defined by direction $\xi$. Formally, denote tangent bundle as $T\mathcal{M}$, which is the union of tangent spaces on $\mathcal{M}$. Then $\mathcal{T}: T\mathcal{M} \oplus T\mathcal{M} \xrightarrow{} T\mathcal{M}$ satisfies (1) $\mathcal{T}_\xi v \in T_{R_x(\xi)}\mathcal{M}$, (2) $\mathcal{T}_{0_x} v = v$ and (3) $\mathcal{T}_\xi$ is a linear map. An isometric vector transport preserves norm and inner product similar as parallel transport. That is, $\langle \mathcal{T}_x^z u , \mathcal{T}_x^z v \rangle_z = \langle u , v \rangle_x$ for all $u, v \in T_x\mathcal{M}$ and $x, z \in \mathcal{M}$. Throughout this paper, we implicitly assume vector transport is isometric. It is easy to show that retraction and vector transport includes exponential map and parallel transport as special cases. For notation purposes, we omit the subscripts for norm and inner product. Specific indication should be clear from context. Also, we denote $[n] := \{ 1,...,n\}$ and $\mathbbm{1}_{\{\cdot\}}$ as the indicator function.  $\text{grad}f_{\mathcal{I}}(x) := \frac{1}{|\mathcal{I}|} \sum_{i \in \mathcal{I}} \text{grad}f_{i}(x)$ is a mini-batch Riemannian stochastic gradient on $T_x\mathcal{M}$, where $\mathcal{I} \subset [n]$ is an index set with cardinality $|\mathcal{I}|$. When $\mathcal{I} \equiv [n]$, we obtain the full gradient as $\text{grad}f(x) = \frac{1}{n} \sum_{i=1}^n \text{grad}f_i(x)$.

\begin{table*}[!t]
\setlength{\tabcolsep}{3pt}
\renewcommand{\arraystretch}{1.5}
\caption{Comparison of IFO complexity between existing results and this work on {general non-convex} problems. We present complexities in terms of parameters defined in Assumptions \ref{common_assump}, \ref{assumption_svrg} and \ref{R_svrg_additional_assump}. $\Theta := \max\{L, \sqrt{L_l^2 + \theta^2 G^2}\}$, $\Theta_1 := L + \sqrt{L^2 + \varrho_1 (L_l + \theta G)^2 \mu^2 \nu^2}$, $\Theta_2 := L + \sqrt{L^2 + \varrho_2 (L_l + \theta G)^2}$, where $\varrho_1, \varrho_2 > 0$ are constants that do not depend on any parameters. $\Tilde{B} := \frac{1}{S}\sum_{s=1}^{S} \min\{ {\alpha_1 \sigma^2}/{\beta}_s , n\}$ under finite-sum setting and $\Tilde{B} := \frac{1}{S}\sum_{s=1}^{S} \min\{ {\alpha_1 \sigma^2}/{\beta}_s , {\alpha_2 \sigma^2}/{\epsilon^2}\}$ under online setting.}
\label{complexity_table_summary}
\resizebox{\textwidth}{!}{%
\begin{tabular}{c || c | c c | c c}
\thickhline

                               \multicolumn{2}{c|}{\multirow{2}{*}{General non-convex}}          & \multicolumn{2}{c|}{(Retraction and vector transport)} & \multicolumn{2}{c}{(Exponential map and parallel transport)} \\
                               \multicolumn{2}{c|}{}           & Finite-sum                 & Online                 & Finite-sum                     & Online                    \\
\hline\hline
\multirow{3}{*}{\shortstack{Existing \\ work}} & R-SVRG \cite{ZhangRSVRG2016}    &   ---                            & ---                       & $\mathcal{O}\big(n + \frac{L \zeta^{1/2} n^{2/3}}{\epsilon^2} \big)$                               & ---                          \\
                               & R-SRG \cite{KasaiRSRG2018}    & $\mathcal{O} \big( n + \frac{\Theta^2}{\epsilon^4} \big)$                       &  ---                      &  $\mathcal{O} \big( n + \frac{L^2}{\epsilon^4} \big)$                              & ---                          \\
                               & R-SPIDER \cite{ZhouRSPIDER2019,ZhangRSPIDER2018} & $\mathcal{O} \big( n + \frac{\Theta \sqrt{n}}{\epsilon^2} \big)^{\mathrm{*}}$                            &  $\mathcal{O} \big( \frac{\Theta \sigma}{\epsilon^3} \big)$                       & $\mathcal{O} \big( n + \frac{L \sqrt{n}}{\epsilon^2} \big)^{\mathrm{*}}$                               &  $\mathcal{O} \big( \frac{L \sigma}{\epsilon^3}  \big)$                                \\
\hline
\multirow{4}{*}{\shortstack{This \\ work}}     & R-SVRG    & $\mathcal{O}\big( n + \frac{\Theta_1 n^{2/3} }{\epsilon^2} \big)$                           & $\mathcal{O}\big( \frac{\Theta_1 \sigma^{4/3}}{\epsilon^{10/3}} \big)$                       & $\mathcal{O}\big( n + \frac{L n^{2/3} }{\epsilon^2} \big)$                               &    $\mathcal{O}\big( \frac{L \sigma^{4/3}}{\epsilon^{10/3}} \big)$                       \\
                               & R-SRG     &  $\mathcal{O}\big( n + \frac{\Theta_2 \sqrt{n}}{\epsilon^2}  \big)$                          &   $\mathcal{O}\big( \frac{\Theta_2 \sigma}{\epsilon^3} \big)$                     &  $\mathcal{O}\big( n + \frac{L \sqrt{n}}{\epsilon^2}  \big)$                              &  $\mathcal{O}\big( \frac{L \sigma}{\epsilon^3} \big)$                         \\
                               & R-AbaSVRG & $\mathcal{O}\big( \Tilde{B} + \frac{\Theta_1 \Tilde{B}}{n^{1/3} \epsilon^2} +  \frac{\Theta_1 n^{2/3} }{\epsilon^2} \big)$                           &  $\mathcal{O}\big( \frac{\Theta_1 \Tilde{B}}{\sigma^{2/3} \epsilon^{4/3}} + \frac{\Theta_1 \sigma^{4/3}}{\epsilon^{10/3}} \big)$                      &  $\mathcal{O}\big( \Tilde{B} + \frac{L \Tilde{B}}{n^{1/3} \epsilon^2} +  \frac{L n^{2/3} }{\epsilon^2} \big)$                              &  $\mathcal{O}\big( \frac{L \Tilde{B}}{\sigma^{2/3} \epsilon^{4/3}} + \frac{L \sigma^{4/3}}{\epsilon^{10/3}} \big)$                         \\
                               & R-AbaSRG  &   $\mathcal{O}\big( \Tilde{B} + \frac{\Theta_2 \Tilde{B} }{\sqrt{n} \epsilon^2} + \frac{\Theta_2 \sqrt{n}}{\epsilon^2}  \big)$                         &   $\mathcal{O}\big( \frac{\Theta_2 \Tilde{B}}{\sigma \epsilon} + \frac{\Theta_2 \sigma}{\epsilon^3} \big)$                     &   $\mathcal{O}\big( \Tilde{B} + \frac{L \Tilde{B} }{\sqrt{n} \epsilon^2} + \frac{L \sqrt{n}}{\epsilon^2}  \big)$                              &   $\mathcal{O}\big( \frac{L \Tilde{B}}{\sigma \epsilon} + \frac{L \sigma}{\epsilon^3} \big)$                        \\
\thickhline
\end{tabular}}
\vspace{1pt}

{\raggedright \, $^{\mathrm{*}}$ In \cite{ZhouRSPIDER2019}, they present finite-sum complexities of R-SPIDER as minimum of finite-sum and online complexities, which simply applies online choices of parameters to finite-sum setting. \par}
\end{table*}

\section{Overview of Algorithms}

\subsection{Riemannian SGD and variance reduction}

A default solution for optimizing problem \eqref{problem_formulation_begin} is to use R-SGD that iteratively updates 
\begin{equation}
    x_{t+1} = R_{x_t}\big( -\eta_t \text{grad}f_{\mathcal{I}_t}(x_t) \big), \label{RSGD_update}
\end{equation}
where $\eta_t > 0$ is the step size. The updates move along the retraction curve from current iterate with the direction determined by negative gradient. When $\mathcal{M} \equiv \mathbb{R}^d$, \eqref{RSGD_update} reduces to $x_{t+1} = x_t - \eta_t \nabla f_{\mathcal{I}_t}(x_t)$, which is the standard SGD update on Euclidean space. 

Variance reduction techniques leverage previous gradient information to construct a modified stochastic gradient with variance that decreases as training progresses. R-SVRG adopts a double loop structure where, at the start of each epoch (i.e. outer loop), a snapshot point $\Tilde{x}$ is fixed and full gradient is evaluated. Within each inner iteration, mini-batch stochastic gradients are computed for the current iterate $x_t$ as well as for the snapshot point. A modified gradient at $x_t$ is then constructed by adjusting deviations according to the difference between the stochastic gradient and full gradient at $\Tilde{x}$. Since Riemannian gradients of $x_t$ and $\Tilde{x}$ are defined on disjoint tangent spaces, vector transport is used to combine gradient information. That is, 
\begin{equation}
    v_t = \text{grad}f_{\mathcal{I}_t}(x_{t}) - \mathcal{T}_{\Tilde{x}}^{x_{t}} \big( \text{grad}f_{\mathcal{I}_t}(\Tilde{x}) - \text{grad}f(\Tilde{x}) \big). \label{SVRG_estimator}
\end{equation}
Instead of using gradient information from a distant reference point, R-SRG recursively modifies stochastic gradients based on the previous iterate. That is, after computing batch gradient $v_0 = \text{grad}f(x_0)$ on an initial point, a modified gradient is constructed within each inner loop as
\begin{equation}
    v_t = \text{grad}f_{\mathcal{I}_t}(x_{t}) - \mathcal{T}_{x_{t-1}}^{x_{t}} \big( \text{grad}f_{\mathcal{I}_t}(x_{t-1}) - v_{t-1} \big). \label{recursive_grad_estimator}
\end{equation}
This is followed by a standard retraction update $x_{t+1} = R_{x_t}(- \eta_t v_t )$. Note for both original formulation of R-SVRG \cite{SatoRSVRG2019,ZhangRSVRG2016} and R-SRG \cite{KasaiRSRG2018}, stochastic gradient $\text{grad}f_{i_t}(x_t)$ rather than mini-batch gradient $\text{grad}f_{\mathcal{I}_t}(x_{t})$ is considered in \eqref{SVRG_estimator} and \eqref{recursive_grad_estimator}. R-SPIDER employs the same recursive gradient estimator as in \eqref{recursive_grad_estimator}. A fundamental difference is the use of normalized gradient for its update, which is given by $x_{t+1} = R_{x_t} \big(-\eta_t \frac{v_t}{\|v_t\|} \big)$. Therefore, it requires step size $\eta_t$ to be proportional to the desired accuracy $\epsilon$ and depend on $\| v_t\|$. Also, R-SPIDER does not adopt inner-outer loop framework. This formulation results in a distinct proof idea that shows progress every iteration by bounding the distance between consecutive iterates. 

\subsection{Batch size adaptation}
For all the above mentioned variance reduction methods, batch gradient of a reference point is required on occasions, which potentially hinders their performance for large datasets and slows down their convergence in the initial epochs. The intuition of adaptively increasing the batch size is simple. At early stages, a small batch gradient is sufficient to achieve variance reduction with all stochastic gradients pointing to similar directions. Towards the end of optimization where exploration area becomes smaller, larger batch gradients are needed to correct for deviation of stochastic gradients. This motivates the improved SGD \cite{DeSGDadapt2017,SievertSGDadapt2019}, the improved SVRG \cite{HarikandehP_SVRG2015} and its variant SCSG \cite{LeiSCSG2017} on Euclidean space. Despite all these efforts, few of them design an adaptive batch size scheme based on information throughout training process. Particularly, both the improved SVRG and SCSG simply resorts to an exponential increase of batch size. 

In this paper, we adopt the idea in \cite{JiAdaSPIDERboost2019} to design adaptive batch size based on norm of modified gradients. It is generally believed that gradient norm decreases as optimization proceeds and therefore is indicative of optimization stages. Our primary analysis is based on inner-outer loop formulation of R-SVRG and R-SRG. For R-SPIDER, we defer its analysis to Appendix \ref{r_abaspider_appendix} because we notice the use of variable step size imposes some difficulties in generalizing this adaptive strategy. By assuming a bounded gradient norm, we can similarly prove its convergence. Nevertheless, the total complexity can be worse than its original complexity in practice. Let $s$ and $t$ respectively represent the outer loop and inner loop index. Ideally as shown from the convergence analysis, batch size $B^s$ should be adjusted inversely proportional to $\sum_{t} \| v_t^s \|^2$, which involves modified gradient in epoch $s$. However, $B^s$ should be determined before inner iteration starts where $v_t^s$ are computed and hence this strategy is infeasible. Rather than adopting a backtracking line search approach as in \cite{DeSGDadapt2017}, Ji \emph{et al}. \cite{JiAdaSPIDERboost2019} simply replaces $\sum_{t} \| v_t^s \|^2$ with gradients in the preceding epoch, which is $\sum_t\| v^{s-1}_{t} \|^2$. Instead of focusing on epoch-wise progress, they consider telescoping over all epochs to prove its convergence. 

\subsection{Riemannian variance reduction with batch size adaptation}
Riemannian adaptive batch-size SVRG (R-AbaSVRG) and SRG (R-AbaSRG) are discussed in Section \ref{R_abasvrg_section} and \ref{R_abasrg_section}. At the start of each epoch, $B^s$ is determined by $\frac{\alpha_1 \sigma^2 m}{\sum_t \| v_t^{s-1} \|^2}$ where $\alpha_1$ is a parameter that should be sufficiently large and $m, \sigma^2$ are the size of inner loop and variance bound respectively. As training progresses, $B^s$ should gradually increase to $n$ under finite-sum setting and to $\alpha_2\sigma^2/\epsilon^2$ under online setting. Without-replacement sampling is employed to construct batch gradients. This is to ensure that full batch gradient is computed under finite-sum setting, thus recovering vanilla R-SVRG and R-SRG. Under online setting, it makes no theoretical difference between with- and without-replacement sampling as $n$ approaches infinity. Here we consider setting the initial point (or reference point) as the last iterate of the previous epoch. This is in contrast to some update rules such as uniform selection in R-SRG \cite{KasaiRSRG2018} or Riemannian centroid in R-SVRG \cite{SatoRSVRG2019}. Especially for R-SRG, this simple modification allows us to derive double loop convergence, which is stronger than single loop convergence in \cite{KasaiRSRG2018} under finite-sum setting.

\section{Assumptions and definitions}
We first present three sets of assumptions as follows. Assumption \ref{common_assump} is standard to analysis of all variance reduction methods on Riemannian manifold. Assumption \ref{assumption_svrg} is required for analysing SVRG-type methods and Assumption \ref{R_svrg_additional_assump} is further needed to establish convergence of R-SVRG under traditional Lyapunov analysis. All assumptions are commonly made in the analysis of algorithms using retraction and vector transport, see  \cite{HuangQN2015,KasaiRSRG2018,SatoRSVRG2019,ZhouRSPIDER2019}.

\begin{assumption}
\label{common_assump} ~\\\vspace{-0.15in}
    \begin{enumerate}[leftmargin = 0.85cm, label=(\ref{common_assump}.\arabic*)]
        \item Function $f$ and its component functions $f_i, i = 1,...,n$ are twice continuously differentiable. \label{function_countinuous_assump}
        \item Iterate sequences produced by algorithms stay continuously in a neighbourhood $\mathcal{X} \subset \mathcal{M}$ around an optimal point $x^*$. Additionally, $\mathcal{X}$ is a totally retractive neighbourhood of $x^*$ where retraction $R$ is a diffeomorphism (i.e. bijective with differentiable inverse). \label{compact_neighbourhood_assump} 
        \item Norms of Riemannian gradient and Riemannian Hessian are bounded. That is, for all $x \in \mathcal{X}$ and any component function $f_i$, there exists constants $G, H > 0$ where $\| \emph{grad}f_i(x) \| \leq G$ and $\| \emph{Hess}f_i(x) \| \leq H$ hold. \label{bounded_norm_assump}
        \item  Variance of Riemannian gradient is bounded. That is, for all $x \in \mathcal{X}$, $\mathbb{E} \| \emph{grad}f_i(x) - \emph{grad}f(x) \|^2 \leq \sigma^2$. \label{bounded_var_assump}
        \item \label{retraction_smooth_assump} Function $f$ is retraction $L$-smooth with respect to retraction $R$. That is, for all $x, y = R_x(\xi) \in \mathcal{X}$, there exists a constant $L > 0$ such that
        $$f(y) \leq f(x) + \langle \emph{grad}f(x), \xi \rangle + \frac{L}{2} \|\xi \|^2.$$
        \item \label{retaction_Ll_assump} Function $f$ is average retraction $L_l$-Lipschitz. That is, for all $x, y \in \mathcal{X}$, there exists a constant $L_l > 0$ such that
        $$\mathbb{E} \| \emph{grad}f_i(x) - P^x_y \emph{grad} f_i(y)  \| \leq L_l \| \xi \|,$$
        where $P_y^x$ is the parallel transport operator from $y$ to $x$ along the retraction curve $c(t) := R_x(t\xi)$ with $c(0) = x, c(1) = y$. Note we distinguish $P_y^x$ with $\Gamma_y^x$ where the latter transports along a geodesic between $x$ and $y$. 
        \item \label{diff_vec_par_bound} (Lemma 3.5 in \cite{HuangQN2015}) Difference between vector transport $\mathcal{T}$ and parallel transport $P$ associated with the same retraction $R$ is bounded. That is, for all $x, y = R_x(\xi) \in \mathcal{X}$ and $ \eta \in T_x\mathcal{M}$, there exists a constant $\theta > 0$, such that
        $$\| \mathcal{T}_x^y \eta - P_x^y \eta \| \leq \theta \| \xi \|\| \eta\|.$$
    \end{enumerate}
\end{assumption}

The expectation in Assumption \ref{bounded_var_assump} and \ref{retaction_Ll_assump} is taken with respect to component index $i$ and therefore is equivalent to sample average. For example, \ref{bounded_var_assump} can be rewritten as $\frac{1}{n}\sum_{i=1}^n \| \text{grad}f_i(x) - \text{grad}f(x)  \| \leq \sigma^2$. Assumption \ref{function_countinuous_assump} and \ref{compact_neighbourhood_assump} are basic for standard analysis. Assumption \ref{bounded_norm_assump} is necessary to establish Lipschitzness with vector transport and is generally satisfied for compact manifolds \cite{KasaiRSRG2018}. Note that Assumption \ref{bounded_var_assump} is introduced to bound deviation resulting from inexact batch gradient at the start of epoch. For vanilla R-SRG and R-SVRG, this assumption is not required. Assumption \ref{retraction_smooth_assump} is guaranteed by combining \ref{function_countinuous_assump}, \ref{compact_neighbourhood_assump} and upper-Hessian bounded condition \cite{HuangQN2015} where $f$ satisfies $\frac{d^2f(R_x(t\xi))}{dt^2} \leq L$ for all $x \in \mathcal{M}, \xi \in T_x\mathcal{M}$ with $\| \xi\| =1$. Assumption \ref{retaction_Ll_assump} can be derived from Assumption \ref{function_countinuous_assump} to \ref{bounded_norm_assump} and the condition that vector transport $\mathcal{T} \in C^0$ \cite{HuangQN2015}. Similarly, Assumption \ref{diff_vec_par_bound} is satisfied by requiring $\mathcal{T} \in C^0$ and $P \in C^{\infty}$. 

\begin{assumption}
\label{assumption_svrg} ~\\\vspace{-0.15in}
    \begin{enumerate}[leftmargin = 0.85cm, label=(\ref{assumption_svrg}.\arabic*)]
        \item The neighbourhood $\mathcal{X}$ is also a totally normal neighbourhood of $x^*$ where exponential map is a diffeomorphism. \label{total_normal_assump}
        \item (Lemma 3 in \cite{HuangRTR_rankone2015}) There exists $\mu, \nu, \delta_{\mu,\nu} > 0$ where for all $x, y = R_x(\xi) \in \mathcal{X}$ with $\| \xi \| \leq \delta_{\mu, \nu}$, we have
        $$\| \xi \| \leq \mu d(x, y), \quad \text{and} \quad d(x, y) \leq \nu \| \xi\|,$$ \label{relation_distance_assump}
        where $d(x, y)$ is the Riemannian distance.
    \end{enumerate}
\end{assumption}

These two assumptions are also basic as in \cite{SatoRSVRG2019}. Assumption \ref{total_normal_assump} is to ensure that we can express Riemannian distance in terms of inverse of exponential map. With $\mathcal{X}$ being both a totally retractive and totally normal neighbourhood, Assumption \ref{relation_distance_assump} characterizes relations between exponential map and retraction. Indeed, we have $\|R_x^{-1}(y) \| \leq \mu \| \text{Exp}^{-1}_x(y) \|$ and $\| \text{Exp}^{-1}_x(y) \| \leq \nu \| R_x^{-1}(y)\|$. This assumption is reasonable as retraction serves as first-order approximation to exponential map and can thus be ensured by choosing a sufficiently small neighbourhood $\mathcal{X}$.

\begin{assumption}
\label{R_svrg_additional_assump} ~\\\vspace{-0.15in}
    \begin{enumerate}[leftmargin = 0.85cm, label=(\ref{R_svrg_additional_assump}.\arabic*)]
        \item The neighbourhood $\mathcal{X}$ is compact with its diameter upper bounded by $D$. That is, $\max_{x, y \in \mathcal{X}} d(x,y) \leq D$. In addition, $\mathcal{X}$ has sectional curvature lower bounded by $\kappa$. \label{compact_manifold_assump}
        \item For all $x,y \in \mathcal{X}$, there exists constant $c_R > 0$ such that $\| R^{-1}_x(y) - \emph{Exp}^{-1}_x(y) \| \leq c_R \| R^{-1}_x(y)\|^2$. \label{close_retract_exp_assump}
    \end{enumerate}
\end{assumption}

Assumption \ref{compact_manifold_assump} is required to apply trigonometric distance bound (Lemma \ref{trig_distance_bound_lemma}) and Assumption \ref{close_retract_exp_assump} is ensuring the difference between exponential map and retraction is small within a neighbourhood. Note that \ref{close_retract_exp_assump} can be implied from Assumption \ref{relation_distance_assump} by triangle inequality. That is $\| R^{-1}_x(y) - \text{Exp}^{-1}_x(y) \| \leq \|R^{-1}_x(y) \| + \| \text{Exp}^{-1}_x(y) \| \leq (1 + \nu) \|R^{-1}_x(y) \|$. 

Indeed, Assumption \ref{common_assump} is sufficient to obtain convergence guarantee for recursive gradient based methods, including R-SRG and R-SPIDER. Additional Assumptions \ref{assumption_svrg} and \ref{R_svrg_additional_assump} introduce constraints on exponential map that bound its difference with retraction. The main intuition is that R-SVRG requires tracing the distances between a remote snapshot point and the iterate sequence, which can only be characterized by exponential map. This is in contrast with recursive gradient estimator that only depends on successive iterates. One final remark is that some assumptions, such as \ref{diff_vec_par_bound} and \ref{relation_distance_assump} are presented as Lemmas in other work. The conditions necessary for ensuring validity of these assumptions are outlined above and hence we can for simplicity rely on these assumptions.

Apart from convergence analysis on general non-convex functions, we also consider an important class of non-convex functions which satisfies Polyak--Łojasiewicz inequality \cite{PolyakPLInEQ1963} on Riemannian manifold, also known as gradient dominance condition. It has been shown that the problem of computing leading eigenvector over the space of Hypersphere satisfies this inequality \cite{ZhangRSVRG2016}. 

\begin{definition}[$\tau$-Gradient Dominance]
\label{gradient_dominate_def}
    A differentiable function $f: \mathcal{M} \xrightarrow{} \mathbb{R}$ is $\tau$-gradient dominated in $\mathcal{X} \subset \mathcal{M}$ if for any $x \in \mathcal{X}$, there exists a $\tau > 0$ such that
    \begin{equation}
        f(x) - f(x^*) \leq \tau \|\emph{grad}f(x) \|^2, \nonumber
    \end{equation}
    where $x^* = \argmin_{x \in \mathcal{M}} f(x)$ is a global minimizer of $f$. 
\end{definition}

With a slightly abuse of notation. we in general refer to $x^* \in \mathcal{M}$ as an optimal point within its neighbourhood $\mathcal{X}$. Only Section \ref{gradient_dominance_section} considers the stronger definition of global minima. Algorithm quality is measured by total IFO complexity to achieve $\epsilon$-accurate solution. A milder definition of $\epsilon$-accuracy that bounds the norm of gradient rather than squared norm of gradient is considered mainly because R-SPIDER (under retraction and vector transport) \cite{ZhouRSPIDER2019} is analysed under this specification. 

\begin{definition}[$\epsilon$-accurate solution and IFO complexity]
    $\epsilon$-accurate solution from a stochastic algorithm is an output $x$ with expected gradient norm no larger than $\epsilon$. That is, $\mathbb{E}\| \emph{grad}f(x) \| \leq \epsilon$. Incremental First-Order (IFO) oracle \emph{\cite{AgarwalIFO2014}} takes a component index $i$ and a point $x \in \mathcal{X}$ and outputs an unbiased stochastic gradient $\emph{grad}f_i(x) \in T_x\mathcal{M}$. IFO complexity counts the total number of IFO oracle calls. 
\end{definition}

\section{Riemannian AbaSVRG} \label{R_abasvrg_section}
Riemannian adaptive batch size SVRG is presented in Algorithm \ref{RadaSVRG_algorithm} where the batch size $B^s$ is adjusted based on the accumulated gradient information from last epoch. Note by simply fixing $B^s = n, s= 1,..., S$, Algorithm \ref{RadaSVRG_algorithm} becomes vanilla R-SVRG under finite-sum setting. We first establish a Theorem that proves non-convex convergence for R-SVRG under retraction and vector transport, which is currently missing in the literature.

\begin{theorem}[Convergence and complexity of R-SVRG under standard analysis]
\label{RSVRG_original_theorem}    
Let $x^* \in \mathcal{M}$ be an optimal point of $f$ and suppose Assumptions \ref{common_assump}, \ref{assumption_svrg} and \ref{R_svrg_additional_assump} hold. Consider Algorithm \ref{RadaSVRG_algorithm} with full batch gradients $B^s = n, s = 1,...,S$ under finite-sum setting. Choose a fixed step size $\eta = \frac{\mu_0 b}{(L_l + \theta G)\mu n^{a_1}(\zeta \nu^2 + 2 c_R D)^{a_2}}$, $m = \lfloor n^{3/2a_1}/ 2b\mu_0 (\zeta \nu^2 + 2 c_R D)^{1-2a_2} \rfloor$, $b \leq n^{a_1}$, where $\zeta \geq 1$ is a curvature constant defined in Lemma \ref{trig_distance_bound_lemma}. Select $a_1, \mu_0 \in (0,1)$, $a_2 \in (0,2)$ and choose $\psi > 0$ such that $\psi \leq \frac{\mu_0}{\mu} \big( 1 - \frac{L \mu_0 (e-1)}{2(L_l + \theta G)(\zeta \nu^2 + 2c_R D)^{2-a_2}\mu} - \frac{L \mu_0 b}{2(L_l + \theta G)(\zeta \nu^2 + 2c_R D)^{a_2} \mu n^{a_1}} - \frac{L\mu_0^2 (e-1) b}{2 (L_l + \theta G)(\zeta \nu^2 + 2c_R D)^{a_2}\mu n^{3/2 a_1}} \big)$ holds. Then output $\Tilde{x}$ after running $T = Sm$ iterations satisfies
\begin{equation}
    \mathbb{E}\| \emph{grad}f(\Tilde{x}) \|^2 \leq \frac{(L_l + \theta G) n^{a_1} (\zeta \nu^2 + 2c_R D)^{a_2} \Delta}{bT \psi}, \nonumber
\end{equation}
where $\Delta := f(\Tilde{x}^0) - f(x^*)$. By choosing $a_1 = 2/3, a_2 = 1/2$, the IFO complexity to achieve $\epsilon$-accurate solution is $\mathcal{O}(n + \frac{(L_l + \theta G) (\zeta \nu^2 + 2 c_R D)^{1/2} n^{2/3} }{\epsilon^2})$.
\end{theorem}

\begin{algorithm}[!t]
 \caption{R-AbaSVRG}
 \label{RadaSVRG_algorithm}
 \begin{algorithmic}[1]
  \STATE \textbf{Input:} Step size $\eta$, epoch length $S$, inner loop size $m$, mini-batch size $b$, adaptive batch size parameters $\alpha_1, \alpha_2, \beta_1$, initialization $\Tilde{x}^0$, desired accuracy $\epsilon$.
  \FOR {$s = 1,...,S$}
  \STATE $x_0^{s} = \Tilde{x}^{s-1}$.
  \STATE $B^s = \begin{cases} \min \{\alpha_1 \sigma^2/\beta_s , n\}, & \text{  (finite-sum)}\\
  \min\{\alpha_1 \sigma^2/\beta_s , \alpha_2 \sigma^2/\epsilon^2 \}, & \text{ (online) } \end{cases}$
  \STATE Draw a sample $\mathcal{B}^s$ from $[n]$ of size $B^s$ without replacement.
  \STATE $v_0^{s} = \text{grad}f_{\mathcal{B}^s}({x}_0^s)$. 
  \STATE $\beta_{s+1} = 0$. 
  \FOR {{$t = 0,...,m-1$}}
  \STATE Draw a sample $\mathcal{I}_t^s$ from $[n]$ of size $b$ with replacement.
  \STATE $v_{t}^{s} = \text{grad}f_{\mathcal{I}_t^s}(x_{t}^{s}) - \mathcal{T}_{{x}_0^s}^{x_{t}^{s}} \big( \text{grad}f_{\mathcal{I}_t^s}(x_0^s) - v_0^{s} \big)$. \label{modify_grad_rsvrg_algo}
  \STATE $x_{t+1}^{s} = R_{x_{t}^{s}} ( -\eta v_{t}^{s} )$.
  \STATE $\beta_{s+1} = \beta_{s+1} + \| v_{t}^s \|^2/m$.
  \ENDFOR
  \STATE $\Tilde{x}^{s} = x_m^{s}$.
  \ENDFOR
  \STATE \textbf{Output:} ${\Tilde{x}}$ uniformly selected at random from $\{\{x_t^s\}_{t=0}^{m-1}\}_{s=1}^{S}$.
 \end{algorithmic} 
\end{algorithm}

Proof of Theorem \ref{RSVRG_original_theorem} is included in Appendix \ref{proof_of_theorem_chapter} and the strategy is similar to \cite{ReddiSVRG2016,ZhangRSVRG2016}. That is, we first derive bounds on the norm of modified gradients $\| v_t^s \|^2$ and also on the distance $d^2(x_t^s, x_0^s)$ between current iterate and reference point within an epoch. Then we construct a Lyapunov function $f(x_t^s) + c_t d^2(x_t^s, x_0^s)$. We therefore can show the norm of gradient at current iterate is upper bounded by the difference in Lyapunov functions at consecutive iterates. In this process, trigonometric distance bound is applied to relate $d^2(x_t^s, x_0^s)$ to $d^2(x_{t+1}^s, x_0^s)$. By carefully choosing parameters and managing coefficients $c_t$, we obtain the desired result. Note that the constant $\psi$ is guaranteed to exist when $\mu_0$ is selected sufficiently small and $b \leq n^{a_1}$. The choice of $a_1, a_2$ is suggested in \cite{ZhangRSVRG2016}. Under exponential map and parallel transport, $L_l = L, \theta = 0, \nu = 1, c_R = 0$ and therefore this complexity reduces to $\mathcal{O}(n + \frac{L n^{2/3} \zeta^{1/2}}{\epsilon^2})$ as in \cite{ZhangRSVRG2016}.

Next we present convergence and IFO complexity of R-AbaSVRG. As a simple corollary, we can derive convergence results of R-SVRG with much simpler proof. This also allows analysis of R-SVRG under online setting, which is novel on Riemannian manifold. Define sigma algebras $\mathcal{F}_t^s := \{ \mathcal{B}^1,...,\mathcal{I}_{m-1}^1, \mathcal{B}^2,...,\mathcal{I}_{m-1}^2, ..., \mathcal{B}^s,...,\mathcal{I}_{t-1}^s \}$.  According to the update rule in Algorithm \ref{RadaSVRG_algorithm}, $v_{t-1}^s$ and $x_{t}^s$ are measurable in $\mathcal{F}_t^s$. Therefore, conditional on $\mathcal{F}_t^s$, randomness at current iteration $t$ only comes from sampling $\mathcal{I}_t^s$ or $\mathcal{B}^s$. We first present a Lemma that bounds deviation of modified stochastic gradient to the full gradient.

\begin{lemma}[Gradient estimation error bound for R-AbaSVRG]
\label{lemma_adasvrg}
Suppose Assumptions \ref{common_assump} and \ref{assumption_svrg} hold and consider Algorithm \ref{RadaSVRG_algorithm}. Then we can bound estimation error of the modified gradient $v_t^s$ to the full gradient $\text{grad}f(x_t^s)$ as 
\begin{align}
    &\mathbb{E}[\| v_t^{s} - \emph{grad}f(x_{t}^s) \|^2 | \mathcal{F}_0^s] \nonumber\\
    &\leq \frac{t}{b} (L_l + \theta G)^2 \mu^2 \nu^2 \eta^2 \sum_{i=0}^{t-1} \mathbb{E}[\| v_i^s \|^2 | \mathcal{F}_0^s] + \mathbbm{1}_{\{ B^s < n\}} \frac{\sigma^2}{B^s}. \nonumber
\end{align}
\end{lemma}
Proof of this Lemma is contained in Appendix \ref{converg_r_adasvrg_appendix}. This suggests that conditional on $\mathcal{F}_0^s$, the deviation of $v_t^s$ to the full gradient is bounded by sum of all previous modified gradient norm in current epoch plus some deviation from full gradient at reference point. When choosing $B^s = n$ as in non-adaptive R-SVRG, the second term vanishes and we therefore obtain a better bound on accuracy of the modified gradient $v_t^s$. Next, we present convergence and complexity bounds for R-AbaSVRG and R-SVRG as follows. 

\begin{theorem}[Convergence analysis of R-AbaSVRG]
\label{convergence_radasvrg}
    Let $x^* \in \mathcal{M}$ be an optimal point of $f$ and suppose Assumptions \ref{common_assump} and \ref{assumption_svrg} hold. Consider Algorithm \ref{RadaSVRG_algorithm} with a fixed step size $\eta \leq \frac{2 - \frac{2}{\alpha}}{L + \sqrt{L^2 + 4(1 - \frac{1}{\alpha})\frac{ (L_l + \theta G)^2 \mu^2 \nu^2 m^2}{b} }}$ with $\alpha \geq 4$. Then under both finite-sum and online setting, output $\Tilde{x}$ after running $T = Sm$ iterations satisfies
    \begin{equation}
        \mathbb{E}\| \emph{grad}f(\Tilde{x}) \|^2 \leq \frac{2\Delta}{T \eta} + \frac{\epsilon^2}{2}, \nonumber
    \end{equation}
    where $\Delta := f(\Tilde{x}^0) - f(x^*)$ and $\epsilon$ is the desired accuracy. 
\end{theorem}

\begin{corollary}[IFO complexity of R-AbaSVRG]
\label{complexity_corollary_adasvrg}
    With same Assumptions in Theorem \ref{convergence_radasvrg}, choose $b = m^2, \alpha = 4$, $\eta = \frac{3}{2L + 2\sqrt{L^2 + 3(L_l + \theta G)^2 \mu^2\nu^2}}$. Set $m = \lfloor n^{1/3} \rfloor$ under finite-sum setting and $m = (\frac{\sigma}{\epsilon})^{2/3}$ under online setting. The IFO complexity of Algorithm \ref{RadaSVRG_algorithm} to achieve $\epsilon$-accurate solution is given by
    $$\begin{cases} \mathcal{O}\big( \Tilde{B} + \frac{\Theta_1 \Tilde{B}}{n^{1/3} \epsilon^2} +  \frac{\Theta_1 n^{2/3} }{\epsilon^2} \big), & \text{ (finite-sum) }\\
    \mathcal{O}\big( \frac{\Theta_1 \Tilde{B}}{\sigma^{2/3} \epsilon^{4/3}} + \frac{\Theta_1 \sigma^{4/3}}{\epsilon^{10/3}} \big), & \text{ (online) }\end{cases}$$
    where $\Theta_1 := L + \sqrt{L^2 + \varrho_1 (L_l + \theta G)^2 \mu^2 \nu^2}$ with $\varrho_1 > 0$ being a constant that does not depend on any parameter. $\Tilde{B}$ is the average batch size defined as follows. $\Tilde{B} := \frac{1}{S}\sum_{s=1}^{S} \min\{ {\alpha_1 \sigma^2}/{\beta}_s , n\}$ under finite-sum setting and $\Tilde{B} := \frac{1}{S}\sum_{s=1}^{S} \min\{ {\alpha_1 \sigma^2}/{\beta}_s , {\alpha_2 \sigma^2}/{\epsilon^2}\}$ under online setting.
\end{corollary}

\begin{corollary}[Convergence and IFO complexity of R-SVRG under new analysis]
\label{conver_complexity_svrg_new}
    With the same assumptions as in Theorem \ref{convergence_radasvrg} and consider Algorithm \ref{RadaSVRG_algorithm} with fixed batch size $B^s = B$ for $s=1,..., S$. Choose a fixed step size $\eta \leq \frac{2}{L + \sqrt{L^2 + 4 \frac{(L_l + \theta G)^2\mu^2 \nu^2 m^2}{b}}}$. Output $\Tilde{x}$ after running $T = Sm$ iterations satisfies
    $$\mathbb{E}\| \emph{grad}f(\Tilde{x}) \|^2 \leq \frac{2\Delta}{T\eta} + \mathbbm{1}_{\{B < n\}} \frac{\sigma^2}{B}.$$
    If we further choose $b = m^2, \eta = \frac{2}{L + \sqrt{L^2 + 4 {(L_l + \theta G)^2\mu^2 \nu^2 }}}$ and the following parameters
    \begin{align*}
        B = n, \quad m = \lfloor n^{1/3} \rfloor \quad &\text{ (finite-sum) } \\
        B = \frac{2\sigma^2}{\epsilon^2}, \quad m = (\frac{\sigma}{\epsilon})^{2/3} \quad &\text{ (online) }
    \end{align*}
    IFO complexity to obtain $\epsilon$-accurate solution is 
    \begin{equation*}
        \begin{cases} \mathcal{O}\big( n + \frac{\Theta_1 n^{2/3} }{\epsilon^2} \big), & \text{ (finite-sum) }\\
    \mathcal{O}\big( \frac{\Theta_1 \sigma^{4/3}}{\epsilon^{10/3}} \big), & \text{ (online) }\end{cases}
    \end{equation*}
\end{corollary}

Proof of all these results are deferred to Appendix \ref{converg_r_adasvrg_appendix}. We first draw a comparison between IFO complexities of vanilla R-SVRG under two analysis frameworks. From Theorem \ref{RSVRG_original_theorem}, the complexity is $\mathcal{O}(n + \frac{(L_l + \theta G) (\zeta \nu^2 + 2 c_R D)^{1/2} n^{2/3} }{\epsilon^2})$, which is the same as $\mathcal{O}(n + \frac{\Theta_1 n^{2/3} }{\epsilon^2})$ in Corollary \ref{conver_complexity_svrg_new} up to a constant. Under standard analysis of Lyapunov function, the complexity is further controlled by curvature constant $\zeta$ and diameter $D$. Rather, we adopt a proof idea similar as in \cite{JiAdaSPIDERboost2019}, which is to bound function value difference by accumulated sum of modified gradient norm. This analysis is much simpler without requiring trigonometric distance bound and in turn avoids compactness and bounded curvature assumptions. Additionally, our proof slightly differs from \cite{JiAdaSPIDERboost2019} where we apply $f(x_{t+1}^s) - f(x_t^s) \leq -\frac{\eta}{2}\| \text{grad}f(x_{t}^s) \|^2 + \frac{\eta}{2} \| v_{t}^s - \text{grad}f(x_{t}^s) \|^2 - (\frac{\eta}{2} - \frac{L\eta^2}{2}) \| v_{t}^s \|^2$ rather than $f(x_{t+1}^s) - f(x_t^s) \leq \frac{\eta}{2} \| v_{t}^s - \text{grad}f(x_{t}^s) \|^2 - (\frac{\eta}{2} - \frac{L\eta^2}{2}) \| v_{t}^s \|^2$. The former statement is stronger than the latter and therefore we can directly bound $\| \text{grad}f(x_t^s) \|^2$ based on this inequality, which yields an even simpler proof. Note the new complexity in Corollary \ref{conver_complexity_svrg_new} still depends on parameters $\mu$ and $\nu$ that describe the difference between exponential map and retraction. This is unavoidable as we need to relate distances between iterates and the reference point to norm of modified gradient. 

Comparing with R-SD that requires a complexity of $\mathcal{O}(n + \frac{n}{\epsilon^2})$, R-SVRG is superior with complexity lower by a factor of $\mathcal{O}(n^{1/3})$. The new analysis also provides a complexity of $\mathcal{O}\big(\frac{\Theta_1 \sigma^{4/3}}{\epsilon^{10/3}} \big)$ under online setting, which is the first online complexity established on SVRG-type methods over Riemannian manifold. This corresponds to the best known rate $\mathcal{O}\big(\frac{1}{\epsilon^{10/3}} \big)$ for SVRG-based algorithms on Euclidean space, such as SCSG \cite{LeiSCSG2017} and ProxSVRG$+$ \cite{LiProxSVRG2018}. Given that the complexity of R-SGD is $\mathcal{O}(\frac{1}{\epsilon^4})$, R-SVRG under online setting outperforms R-SGD by a factor of $\mathcal{O}\big(\frac{1}{\epsilon^{2/3}}\big)$.

From Theorem \ref{convergence_radasvrg}, we also note that R-AbaSVRG obtains the same convergence rate as non-adaptive R-SVRG. This suggests that an identical iteration complexity is required to achieve $\epsilon$-accurate solution. Therefore under same choices of parameters, R-AbaSVRG obtains $\mathcal{O}\big( \Tilde{B} + \frac{\Theta_1 \Tilde{B}}{n^{1/3} \epsilon^2} +  \frac{\Theta_1 n^{2/3} }{\epsilon^2} \big)$ under finite-sum setting and $\mathcal{O}\big( \frac{\Theta_1 \Tilde{B}}{\sigma^{2/3} \epsilon^{4/3}} + \frac{\Theta_1 \sigma^{4/3}}{\epsilon^{10/3}} \big)$ under online setting. These complexities can be theoretically much lower than R-SVRG from the definition of $\Tilde{B}$. That is, because $\Tilde{B} \leq n$ under finite-sum setting, the complexity of R-AbaSVRG is at most $\mathcal{O}\big( n + \frac{\Theta_1 n^{2/3}}{\epsilon^2}\big)$, which matches the complexity of R-SVRG. Similar argument holds for online setting.   

Finally, we make one additional comment on the choice of parameters. Theorem \ref{RSVRG_original_theorem} suggests a choice of $m = \mathcal{O}(n/b)$ with $b \leq n^{2/3}$ while both Corollary \ref{complexity_corollary_adasvrg} and \ref{conver_complexity_svrg_new} simply select $b = m^2 =n^{2/3}$. Similar to \cite{JiAdaSPIDERboost2019}, our new analysis does not easily allow more flexible choices of $b$ and $m$ as we do not construct any nontrivial auxiliary variable to achieve this purpose.

\section{Riemannian AbaSRG} \label{R_abasrg_section}

The key step of R-AbaSRG in Algorithm \ref{RAdaSRG_algorithm} is nearly identical to R-AbaSVRG except that the modified gradient $v_t^s$ is constructed recursively from $v_{t-1}^s$. We first similarly present a gradient estimation bound in the following Lemma with $\mathcal{F}_t^s$ representing the same sigma algebras as in the analysis of R-AbaSVRG. 

\begin{algorithm}[!t]
 \caption{R-AbaSRG}
 \label{RAdaSRG_algorithm}
 \begin{algorithmic}[1]
  \STATE \textbf{Input:} Step size $\eta$, epoch length $S$, inner loop size $m$, mini-batch size $b$, adaptive batch size parameters $\alpha_1, \alpha_2, \beta_1$, initialization $\Tilde{x}^0$, desired accuracy $\epsilon$.
  \FOR {$s = 1,...,S$}
  \STATE $x_0^{s} = \Tilde{x}^{s-1}$.
  \STATE $B^s = \begin{cases} \min \{\alpha_1 \sigma^2/\beta_s , n\}, & \text{  (finite-sum)}\\
  \min\{\alpha_1 \sigma^2/\beta_s , \alpha_2 \sigma^2/\epsilon^2 \}, & \text{ (online) } \end{cases}$
  \STATE Draw a sample $\mathcal{B}^s$ from $[n]$ of size $B^s$ without replacement.
  \STATE $v_0^{s} = \text{grad}f_{\mathcal{B}^s}({x}_0^s)$. 
  \STATE $x_1^s = R_{x_0^s}(- \eta v_0^s)$.
  \STATE $\beta_{s+1} = \| v_0^s \|^2/m$. 
  \FOR {{$t = 1,...,m-1$}}
  \STATE Draw a sample $\mathcal{I}_t^s$ from $[n]$ of size $b$ with replacement.
  \STATE $v_t^{s} = \text{grad}f_{\mathcal{I}_t^s}(x_{t}^{s}) -  \mathcal{T}_{{x}_{t-1}^s}^{x_t^{s}} \big( \text{grad}f_{\mathcal{I}_t^s}(x_{t-1}^s) - v_{t-1}^{s} \big)$. \label{modify_grad_rsrg_algo}
  \STATE $x_{t+1}^{s} = R_{x_{t}^{s}} ( -\eta v_{t}^{s} )$.
  \STATE $\beta_{s+1} = \beta_{s+1} + \| v_{t}^s \|^2/m$.
  \ENDFOR
  \STATE $\Tilde{x}^{s} = x_m^{s}$.
  \ENDFOR
  \STATE \textbf{Output:} ${\Tilde{x}}$ uniformly selected at random from $\{\{x_t^s\}_{t=0}^{m-1}\}_{s=1}^{S}$.
 \end{algorithmic} 
\end{algorithm}

\begin{lemma}[Gradient estimation error bound for R-AbaSRG]
\label{lemma1_RSRG}
Suppose Assumption \ref{common_assump} hold and consider Algorithm \ref{RAdaSRG_algorithm}. Then we can similarly bound estimation error of modified gradient $v_t^s$ to the full gradient $\text{grad}f(x_t^s)$ as 
\begin{align}
    &\mathbb{E}[\| v_t^s - \emph{grad}f(x_t^s) \|^2 | \mathcal{F}_0^s] \nonumber\\
    &\leq \frac{(L_l + \theta G)^2\eta^2}{b}  \sum_{i=0}^{t} \mathbb{E}[\| v_i^s \|^2| \mathcal{F}_0^s] + \mathbbm{1}_{\{ B^s < n\}} \frac{\sigma^2}{B^s}. \nonumber
\end{align}
\end{lemma}
A comparison with Lemma \ref{lemma_adasvrg} can be drawn. This bound is tighter than R-AbaSVRG as the first term on the right hand side is smaller by a factor of $\mathcal{O}(t)$. This is mainly due to the use of recursive gradient estimator rather than a distant reference point under SVRG updates. Next, we establish convergence and complexity results for both adaptive and non-adaptive R-SRG.

\begin{theorem}[Convergence analysis of R-AbaSRG]
\label{converg_RSRG}
    Let $x^* \in \mathcal{M}$ be an optimal point of $f$ and suppose Assumption \ref{common_assump} holds. Consider Algorithm \ref{RAdaSRG_algorithm} with a fixed step size $\eta \leq \frac{2 - \frac{2}{\alpha}}{L + \sqrt{L^2 + 4 (1 - \frac{1}{\alpha}) \frac{(L_l + \theta G)^2m}{b} }}$ with $\alpha \geq 4$. Then under both finite-sum and online setting, output $\Tilde{x}$ after running $T = Sm$ iterations satisfies
    \begin{equation*}
        \mathbb{E}\| \emph{grad}f(\Tilde{x}) \|^2 \leq \frac{2\Delta}{T\eta} + \frac{\epsilon^2}{2},
    \end{equation*}
    with $\Delta := f(\Tilde{x}^0) - f(x^*)$ and $\epsilon$ is the desired accuracy.
\end{theorem}

\begin{corollary}[IFO complexity of R-AbaSRG]
\label{complexity_rabasrg_corollary}
    With the same Assumptions and settings in Theorem \ref{converg_RSRG}, choose $b = m, \alpha = 4$. Then consider $\eta = \frac{2 - \frac{2}{\alpha}}{L + \sqrt{L^2 + 4 (1 - \frac{1}{\alpha}) \frac{(L_l + \theta G)^2m}{b} }} = \frac{3}{2L + 2\sqrt{L^2 + 3 {(L_l + \theta G)^2}}}$ with $m = \lfloor n^{1/2} \rfloor$ under finite-sum setting and $m = \frac{\sigma}{\epsilon}$ under online setting. The IFO complexity of Algorithm \ref{RadaSVRG_algorithm} to obtain $\epsilon$-accurate solution is 
    \begin{equation*}
        \begin{cases} \mathcal{O}\big( \Tilde{B} + \frac{\Theta_2 \Tilde{B} }{\sqrt{n} \epsilon^2} + \frac{\Theta_2 \sqrt{n}}{\epsilon^2}  \big), & \text{ (finite-sum) }\\
        \mathcal{O}\big( \frac{\Theta_2 \Tilde{B}}{\sigma \epsilon} + \frac{\Theta_2 \sigma}{\epsilon^3} \big), & \text{ (online) }\end{cases}
    \end{equation*}
    where $\Theta_2 := L + \sqrt{L^2 + \varrho_2 (L_l + \theta G)^2}$ with $\varrho_2 > 0$ independent of any parameter. $\Tilde{B}$ is the same average batch size defined in Corollary \ref{complexity_corollary_adasvrg}. 
\end{corollary}

\begin{corollary}[Double loop convergence and IFO complexity of R-SRG]
\label{converg_complex_rsrg_double_loop}
    With the same assumptions in Theorem \ref{converg_RSRG} and consider Algorithm \ref{RAdaSRG_algorithm} with fixed batch size $B^s = B$, for $s = 1,...,S$. Consider a step size $\eta \leq \frac{2}{L + \sqrt{L^2 + 4\frac{(L_l + \theta G)^2m}{b}}}$. After running $T = Sm$ iterations, output $\Tilde{x}$ satisfies
    \begin{equation*}
        \mathbb{E}\| \emph{grad}f(\Tilde{x}) \|^2 \leq \frac{2\Delta}{T\eta} +  \mathbbm{1}_{\{ B < n\}} \frac{\sigma^2}{B}. 
    \end{equation*}
    If we further choose $b = m$, $\eta = \frac{2}{L + \sqrt{L^2 + 4{(L_l + \theta G)^2}}}$ and following parameters 
    \begin{align*}
        B = n, \quad m = \lfloor n^{1/2} \rfloor, \quad &\text{ (finite-sum) } \\
        B = \frac{2\sigma^2}{\epsilon^2}, \quad m = \frac{\sigma}{\epsilon}, \quad &\text{ (online) }
    \end{align*}
    IFO complexity to obtain $\epsilon$-accurate solution is 
    \begin{equation*}
        \begin{cases} \mathcal{O}\big( n + \frac{\Theta_2 \sqrt{n}}{\epsilon^2}  \big), & \text{ (finite-sum) }\\
        \mathcal{O}\big( \frac{\Theta_2 \sigma}{\epsilon^3} \big), & \text{ (online) }\end{cases}
    \end{equation*}
\end{corollary}

Proof of all these results are presented in Appendix \ref{r_abasrg_proof_appendix} where we adopt the same proof strategies as in R-AbaSVRG. Corollary \ref{converg_complex_rsrg_double_loop} provides complexity bounds for vanilla R-SRG under double loop convergence. Existing work in \cite{KasaiRSRG2018} only established single epoch convergence where the update of $\Tilde{x}^s$ for the next epoch is uniformly chosen from iterates within current epoch. They proved a complexity of $\mathcal{O}(n + \frac{\Theta^2}{\epsilon^4})$ under finite-sum setting, with $\Theta := \max\{L, \sqrt{L_l^2 + \theta^2 G^2}\}$. This is suboptimal when $n \leq \mathcal{O}(\frac{1}{\epsilon^4})$. By simply setting $\Tilde{x}^s$ as the last iterate of current epoch, we can improve on this rate by establishing a double loop convergence. It is aware that under the condition of $n \leq \mathcal{O}(\frac{L^2}{\epsilon^4})$ and $L$-smoothness assumption (Euclidean sense), Fang \emph{et al}. \cite{FangSPIDER2018} proved a lower bound of $\mathcal{O}(n + \frac{L \sqrt{n}}{\epsilon^2})$ for optimizing finite-sum problem over Euclidean space. They proposed SPIDER algorithm to achieve this bound. On the manifold space, R-SPIDER is generalised with complexities matching this state-of-the-art lower bound. Nevertheless, R-SPIDER bears high relevance to R-SRG. In fact, the only key difference of R-SPIDER is to normalize gradient $v_t^s$ before taking a retraction step. Therefore, by selecting a small step size $\eta = \mathcal{O}(\frac{\epsilon}{L})$, they can bound distances between successive iterates $d(x_t, x_{t+1})$ by a small quantity $\mathcal{O}(\epsilon)$. Corollary \ref{converg_complex_rsrg_double_loop} indicates that R-SRG is also able to achieve this optimal complexity up to a constant, which contradicts the claim that gradient normalization is essential for faster rate under recursive gradient estimator \cite{ZhouRSPIDER2019}. Additionally, R-SRG requires IFO complexity of $\mathcal{O}(\frac{\Theta_2 \sigma}{\epsilon^3})$ under online setting, also agreeing with the rate of R-SPIDER. Hence, we can safely conclude that R-SPIDER is equivalent to R-SRG with variable step size ${\eta}/{\|v_t^s \|}$. The superiority of R-SRG lies in its large fixed step size choice. 

Comparing with complexity results obtained by R-SVRG, R-SRG strictly improves by a factor of $\mathcal{O}(n^{1/6})$ under finite-sum setting and $\mathcal{O}\big((\frac{\sigma}{\epsilon})^{1/3}\big)$ under online setting. Furthermore, similar to R-AbaSVRG, R-AbaSRG can achieve the same iteration complexities as R-SRG and therefore with the same choice of inner loop size $m$ and mini batch size $b$, $\epsilon$-accurate solution can be guaranteed with potentially much lower total complexity. Lastly, regarding the choice of parameters, $b = m = \sqrt{n}$ turns out to be non-essential under current complexity analysis. To illustrate, consider R-SRG under finite-sum setting with the choice $m b = n$. From the proof of Corollary \ref{converg_complex_rsrg_double_loop}, the number of epochs required to achieve $\epsilon$-accurate solution is $S = \frac{2\Delta}{m\eta\epsilon^2}  =  \frac{L + \sqrt{L^2 + 4(L_l + \theta G)^2 \frac{m}{b}}}{m \epsilon^2} \leq \frac{2L\sqrt{1 + 4 \frac{(L_l + \theta G)^2 m}{L^2 b} }}{m\epsilon^2}$. Then total IFO complexity is given by $S(n + 2mb) \leq n + \frac{6L \sqrt{b^2 + 4 \frac{(L_l + \theta G)^2n}{L^2}}}{\epsilon^2}$. Hence as long as $b \leq \sqrt{n}$, total complexity is at most $\mathcal{O}(n + \frac{\sqrt{n}}{\epsilon^2})$ ignoring constants. This suggests that we can freely choose $b \in [1, \sqrt{n}]$ and $m \in [\sqrt{n}, n]$ as long as $mb = n$. Step size can also be selected larger when choosing a larger mini-batch size. Finally, note that total complexity does not improve for larger mini-batch size. But it potentially provides linear speedups in distributed systems where $b$ stochastic gradients are computed in parallel \cite{GoyalLB_SGD2017}.

\section{Convergence under gradient dominance}
\label{gradient_dominance_section}

As an important class of non-convex functions, gradient dominated functions (see Definition \ref{gradient_dominate_def}) assume existence of global minima $x^*$ where function value difference of any point to $x^*$ is upper bounded by its gradient. This condition allows linear convergence to be established for non-convex functions. Note that retraction $\varsigma$-strongly convex function is $\frac{1}{2\varsigma}$-gradient dominated.\footnote{Proof of this claim can be seen in Corollary 5 in \cite{ZhangRSVRG2016}. Retraction $\varsigma$-strongly convex $f$ satisfies $f(y) \geq f(x) + \langle \text{grad}f(x), \xi \rangle + \frac{\varsigma}{2}\| \xi \|^2$, for all $x, y = R_x(\xi) \in \mathcal{M}$.} Common strategy of adapting variance reduction methods to gradient dominance condition is by restarting \cite{ZhangRSVRG2016,ZhangRSPIDER2018}. Accordingly, we provide a unified framework shown in Algorithm \ref{RGDVR_algorithm}, similar to \cite{ZhangRSPIDER2018}. The idea is to gradually shrink the desired accuracy at each mega epoch, thus requiring increasing number of iterations $S_k$. By running sufficient number of mega epochs, output $x_K$ is guaranteed to be $\epsilon$-accurate. For vanilla R-SVRG and R-SRG, we consider Algorithm \ref{RadaSVRG_algorithm} and \ref{RAdaSRG_algorithm} with fixed batch size $B^s = B$ for simplicity. 

\begin{algorithm}
 \caption{R-GD-VR}
 \label{RGDVR_algorithm}
 \begin{algorithmic}[1]
  \STATE \textbf{Input:} Initial accuracy $\epsilon_0$ and desired accuracy $\epsilon$, initialization $x_0$.
  \FOR {$k = 1,...,K$}
  \STATE $\epsilon_k = \frac{\epsilon_{k-1}}{2}$ and set other parameters accordingly.
  \STATE (R-SVRG): 
        \begin{equation*}
            x_k = \text{R-AbaSVRG} (x_{k-1}, \epsilon_k, S_k, m_k, b_k, B_k, \eta)
        \end{equation*}
  \STATE (R-AbaSVRG): 
        \begin{equation*}
            x_k = \text{R-AbaSVRG} (x_{k-1}, \epsilon_k, S_k, m_k, b_k, \eta, \alpha_1, \alpha_2, \beta_1)
        \end{equation*}
  \STATE (R-SRG): 
        \begin{equation*}
            x_k = \text{R-AbaSRG} (x_{k-1}, \epsilon_k, S_k, m_k, b_k, B_k, \eta)
        \end{equation*}
  \STATE (R-AbaSRG): 
        \begin{equation*}
            x_k = \text{R-AbaSRG} (x_{k-1}, \epsilon_k, S_k, m_k, b_k, \eta, \alpha_1, \alpha_2, \beta_1)
        \end{equation*}
  \ENDFOR
  \STATE \textbf{Output:} $x_K$.
 \end{algorithmic} 
\end{algorithm}

\begin{theorem}[IFO complexity of R-AbaSVRG and R-SVRG]
\label{complexity_svrg_gd}
    Suppose Assumptions \ref{common_assump} and \ref{assumption_svrg} hold and also suppose function $f$ satisfies $\tau$-gradient dominance condition. Consider Algorithm \ref{RGDVR_algorithm} with any solver and accordingly choose appropriate parameters to achieve $\epsilon_k$-accurate solution. Then at each mega epoch $k$, output $x_k$ satisfies
    \begin{equation}
        \mathbb{E}\| \emph{grad}f(x_k) \| \leq \frac{\epsilon_0}{2^k}, \text{ and } \mathbb{E}[f(x_k) - f(x^*)] \leq \frac{\tau \epsilon_0^2}{4^k}. \nonumber
    \end{equation}
    Consider R-AbaSVRG solver with the following parameters at each mega epoch. $\eta = \frac{3}{2L + 2\sqrt{L^2 + 3(L_l + \theta G)^2 \mu^2 \nu^2}}$, $\alpha = 4, b_k = m_k^2$, where $m_k = \lfloor n^{1/3} \rfloor$ under finite-sum setting and $m_k = (\frac{\sigma}{\epsilon_k})^{2/3}$ under online setting. Then to achieve $\epsilon$-accurate solution, total IFO complexity is given by
    \begin{align}
        \begin{cases} \mathcal{O}\big( \sum_{k=1}^K \Tilde{B}_k (1 + \frac{\Theta_1 \tau}{n^{1/3}}) + ({\Theta_1 n^{2/3} \tau}) \log(\frac{1}{\epsilon}) \big), & \text{(finite-sum) } \\
        \mathcal{O}\big( \frac{\Theta_1 \tau \sum_{k=1}^K \Tilde{B}_k \epsilon_k^{2/3}}{\sigma^{2/3}} + \frac{\Theta_1 \tau \sigma^{4/3}}{\epsilon^{4/3}} \big), & \text{(online) } \end{cases} \nonumber
    \end{align}
    where the average batch size at mega epoch $k$ is $\Tilde{B}_k := \frac{1}{S_k} \sum_{s=1}^{S_k} \min\{\alpha_1 \sigma^2/\beta_s, n \}$ under finite-sum setting and $\Tilde{B}_k := \frac{1}{S_k} \sum_{s=1}^{S_k} \min\{\alpha_1\sigma^2/\beta_s, \alpha_2 \sigma^2/\epsilon_k^2 \}$ under online setting. Consider R-SVRG solver with the same parameters except for $\eta = \frac{2}{L + \sqrt{L^2 + 4(L_l + \theta G)^2 \mu^2 \nu^2)}}$, $B_k = n$ under finite-sum setting and $B_k = \frac{2\sigma^2}{\epsilon_k^2}$ under online setting. To achieve $\epsilon$-accurate solution, we require a total complexity of 
    \begin{align}
        \begin{cases} \mathcal{O}\big( (n + \Theta_1 \tau n^{2/3})\log(\frac{1}{\epsilon}) \big), & \text{ (finite-sum) }\\
        \mathcal{O} \big( \frac{\Theta_1 \tau \sigma^{4/3}}{\epsilon^{4/3}} \big), & \text{ (online) } \end{cases} \nonumber
    \end{align}
\end{theorem}

\begin{theorem}[IFO complexity of R-AbaSRG and R-SRG]
\label{complexity_srg_gd}
    Suppose Assumptions \ref{common_assump} holds and also suppose function $f$ satisfies $\tau$-gradient dominance. By choosing parameters to achieve $\epsilon_k$-accurate solution, output $x_k$ satisfies the same linear convergence as in Theorem \ref{complexity_svrg_gd}. Consider R-AbaSRG solver with $\eta = \frac{3}{2L + 2\sqrt{L^2 + 3(L_l + \theta G)^2}}$, $\alpha = 4, b_k = m_k$ where $m_k = \lfloor n^{1/2} \rfloor$ under finite-sum setting and $m_k = \frac{\sigma}{\epsilon_k}$ under online setting. To achieve $\epsilon$-accurate solution, we require a total IFO complexity of 
    \begin{align}
        \begin{cases} \mathcal{O}\big( \sum_{k=1}^K \Tilde{B}_k (1 + \frac{\Theta_2 \tau}{n^{1/2}}) + ({\Theta_2 n^{1/2} \tau}) \log(\frac{1}{\epsilon}) \big), & \text{(finite-sum) }\\
        \mathcal{O}\big( \frac{\Theta_2 \tau \sum_{k=1}^K \Tilde{B}_k \epsilon_k}{\sigma} + \frac{\Theta_2 \tau \sigma}{\epsilon} \big), & \text{(online) } \end{cases} \nonumber
    \end{align}
    where $\Tilde{B}_k$ is the average batch size defined in Theorem \ref{complexity_svrg_gd}. Consider R-SRG solver with the same parameters except for $\eta = \frac{2}{L + \sqrt{L^2 + 4(L_l + \theta G)^2}}$ and $B_k = n$ under finite-sum setting and $B_k = \frac{2\sigma^2}{\epsilon_k^2}$ under online setting. To achieve $\epsilon$-accurate solution, we require a total complexity of 
    \begin{align}
        \begin{cases} \mathcal{O}\big( (n + \Theta_2 \tau n^{1/2})\log(\frac{1}{\epsilon}) \big), & \text{ (finite-sum) }\\
        \mathcal{O} \big( \frac{\Theta_2 \tau \sigma}{\epsilon} \big), & \text{ (online) } \end{cases} \nonumber
    \end{align}
\end{theorem}

See Appendix \ref{gd_convergence_appendix} for proof of these results. We first note that under gradient dominance condition, R-SD requires a complexity of $\mathcal{O}\big( (n + {L \tau n})\log(\frac{1}{\epsilon}) \big)$ and R-SGD requires $\mathcal{O}\big( \frac{LG^2}{\epsilon^2} \big)$ as shown in Theorem \ref{complexity_RGD_RSGD_GD} (Appendix \ref{gd_convergence_appendix}). These results are consistent with those established on Euclidean space \cite{PolyakPLInEQ1963,KarimiGD_SGD2016}. Similar to general non-convex setting, R-SVRG requires lower complexities, with a factor of $\mathcal{O}(n^{1/3})$ lower than R-SD and a factor of $\mathcal{O}\big(\frac{1}{\epsilon^{2/3}}\big)$ lower than R-SGD. It is aware that Zhang \emph{et al}. \cite{ZhangRSVRG2016} proved a complexity of $\mathcal{O} \big( (n + L \zeta^{1/2} \tau n^{2/3} ) \log(\frac{1}{\epsilon})\big)$ for R-SVRG under standard analysis, which is the same as ours up to a constant factor. R-SRG further improves on these rates by $\mathcal{O}(n^{1/6})$ and $\mathcal{O}\big(\frac{1}{\epsilon^{1/3}}\big)$ under two cases respectively. Similar to the general non-convex case, these IFO complexities can be further improved by batch size adaptation. For example, consider R-AbaSVRG under finite-sum setting with complexity given by $\mathcal{O}\big( \sum_{k=1}^K \Tilde{B}_k (1 + \frac{\Theta_1 \tau}{n^{1/3}}) + ({\Theta_1 n^{2/3} \tau}) \log(\frac{1}{\epsilon}) \big)$. By definition, $\sum_{k=1}^K \Tilde{B}_k (1 + \frac{\Theta_1 \tau}{n^{1/3}}) \leq \sum_{k=1}^K n(1+\frac{\Theta_1 \tau}{n^{1/3}}) = (n + {\Theta_1 n^{2/3}\tau})\log(\frac{1}{\epsilon})$. Hence the complexity is at most the same as non-adaptive R-SVRG, which is $\mathcal{O}\big( (n + \Theta_1 n^{2/3}\tau)\log(\frac{1}{\epsilon}) \big)$. Similar argument holds for R-AbaSRG and online setting. 

Kasai \emph{et al}. \cite{KasaiRSRG2018} proved a complexity of $\mathcal{O}\big( (n + \tau^2 \Theta^2 )\log(\frac{1}{\epsilon^2}) \big)$ for R-SRG. This is because the inner loop convergence does not require restarting the algorithm and simply running $\mathcal{O}\big(\log(\frac{1}{\epsilon^2}) \big)$ outer iterations is sufficient to achieve linear convergence. Comparing with the rate of $\mathcal{O}\big( (n + \Theta_2 \tau n^{1/2}) \log(\frac{1}{\epsilon}) \big)$ under current framework, we again highlight a trade-off between sample size and desired accuracy. When $n$ is small relative to $\epsilon$, our rate is superior. It is noticed that R-SPIDER also achieves the same rate as R-SRG under gradient dominance condition. This further consolidates the belief that R-SRG theoretically performs the same as R-SPIDER, with matching complexities. 

Finally, since retraction strongly convex functions are special types of gradient dominated functions. These results can be readily extended for the stronger class of functions. For example, under finite-sum setting, suppose $f$ is retraction $\varsigma$-strongly convex, R-SVRG requires a complexity of $\mathcal{O}\big( (n + \Theta_1 \varsigma^{-1}n^{2/3}) \log(\frac{1}{\epsilon}) \big)$ and R-SRG requires a complexity of $\mathcal{O} \big( (n + \Theta_2 \varsigma^{-1} n^{1/2}) \log(\frac{1}{\epsilon}) \big)$. 

\section{Convergence under exponential map and parallel transport}
Table \ref{complexity_table_summary} summarizes complexity bounds derived in this paper, with a comparison to existing work on general non-convex functions. Trivially, our analysis of retraction and vector transport easily adapts to more restricted exponential map and parallel transport. That is, we can simply replace retraction $L$-smooth and $L_l$-Lipschitz assumptions by geodesic $L$-smoothness and $L$-Lipschitzness \cite{ZhangRSVRG2016}. Therefore, $\Theta_1, \Theta_2$ reduce to $L$ as $\theta = 0$, $\mu = \nu = 1$. In general, we have $\Theta_1, \Theta_2 > L$ where retraction and vector transport deviate from exponential map and parallel transport. Note that since we do not assume a bounded sectional curvature, which appears in the standard complexity results of R-SVRG under exponential map and parallel transport, our rate is slightly better. Similar conclusions can be made for gradient dominated functions.

\section{Experiments}

This section empirically evaluates effectiveness of batch size adaptation on variance reduction algorithms over a number of tasks. To make a comparison with some first-order baseline methods, we also include results from R-SD, R-SGD as well as Riemannian conjugate gradient (R-CG) \cite{AbsilMatrixManifold2009}. Except for R-SD and R-CG that have inbuilt line search algorithm, all other methods require fine tuning step size. For simplicity, we consider a fixed step size $\eta$ for SVRG and SRG based methods, a decaying step size for R-SGD and an adaptive step size for R-SPIDER. Denote $k$ as the iteration index and $p$ as the batch gradient frequency for R-SPIDER. Then the decaying step size is given by $\eta_k = \eta (1+ \eta \lambda_\eta k)$ and the adaptive step size is $\eta_k = \alpha_\eta^{\lfloor k/p \rfloor} \cdot \beta_\eta$, as suggested in \cite{KasaiRSRG2018,ZhouRSPIDER2019}. Particularly, convergence theory of R-SPIDER requires a small step size proportional to desired accuracy, which hampers convergence speed for initial epochs. The adaptive step size generally performs better. 

Some global parameter settings are as follows. For variance reduction methods and their adaptive batch size versions, we set inner loop size $m$, mini-batch size $b$ and batch gradient frequency $p$ to be $\sqrt{n}$, which agrees with convergence theories. We set $\lambda_\eta = 0.01$ for R-SGD and select $\alpha_\eta$ from $\{0.1, 0.2, ..., 0.8, 0.85, 0.9, 0.95, 0.99 \}$ and $\beta_\eta$ from $\{0.001, 0.005, 0.01, 0.05, 0.1, 0.5\}$ for R-SPIDER. This search grid is more extensive than that in \cite{ZhouRSPIDER2019} as we found for some applications, a smaller search grid is unable to ensure convergence. We set adaptive batch size $B^s = \min\{n, c_\beta/\beta_s \}, s >1$. Initial batch size $B^1$ is set to be $50$ and therefore we only need to tune $c_\beta$. It is noticed that on manifold space, due to error caused by vector transport operator, inexact batch gradients at initial stages can further deviate when $m$ is large. Hence practically, we set inner loop size $m_s = \min \{B^s, m\}$. Also, mini-batch size is set as $b_s = \min \{B^s, b\}$ since it is unreasonable for batch gradient to be less exact than mini-batch gradients. To achieve fairness in comparisons, we first tune step size $\eta$ on vanilla variance reduction methods. Then the best tuned $\eta$ is fixed for their adaptive versions, where $c_\beta$ is tuned accordingly. We select $\eta$ from $\{ 1, 2,...,9 \} \times 10^q$ and $c_\beta$ from $\{1, 3, ..., 15 \} \times 10^l$, where $q, l$ are to be determined for each problem. All results presented are coded in Matlab on a i5-8600 3.1GHz CPU processor.

\begin{figure*}
\captionsetup{justification=centering}
    \centering
    \subfloat[\label{pca_optimiality_ifo_synthetic_figure} Optimiality gap vs. IFO (Synthetic) ]{\includegraphics[width = 0.28\textwidth, height = 0.21\textwidth]{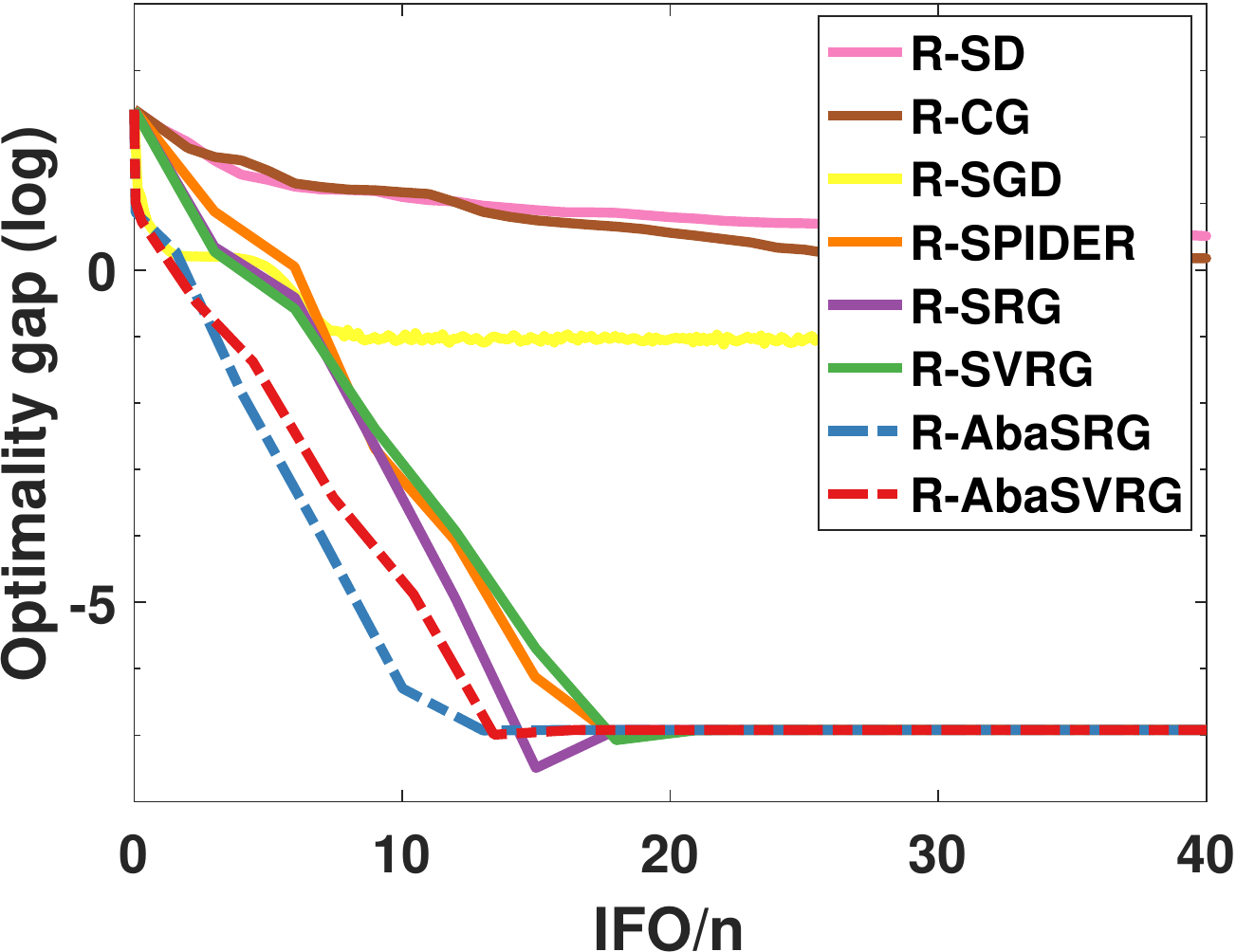}}
    \hspace{0.05in}
    \subfloat[Optimality gap vs. IFO (Mnist)]{\includegraphics[width = 0.28\textwidth, height = 0.21\textwidth]{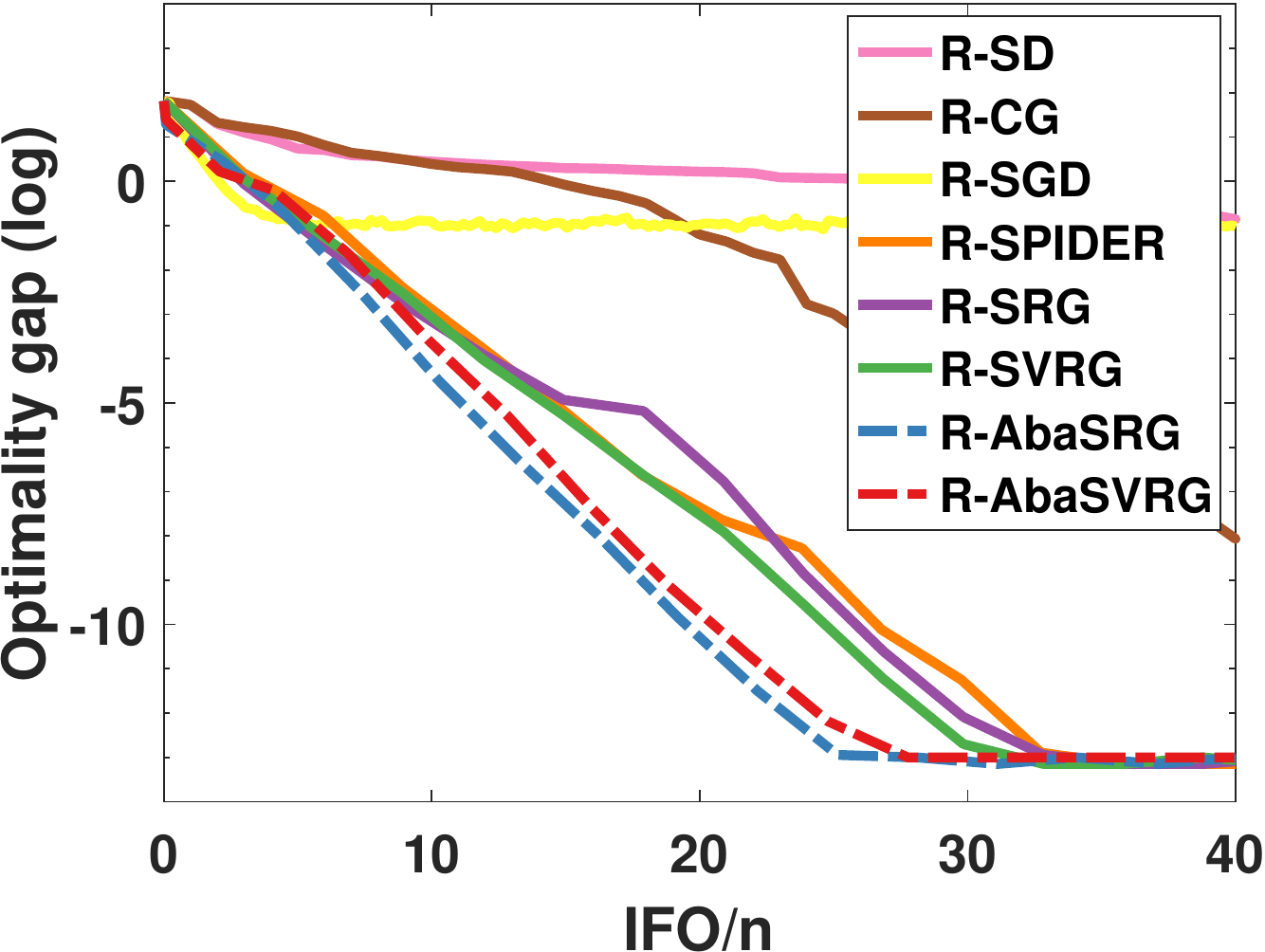}}
    \hspace{0.05in}
    \subfloat[Optimality gap vs. IFO (Ijcnn)]{\includegraphics[width = 0.28\textwidth, height = 0.21\textwidth]{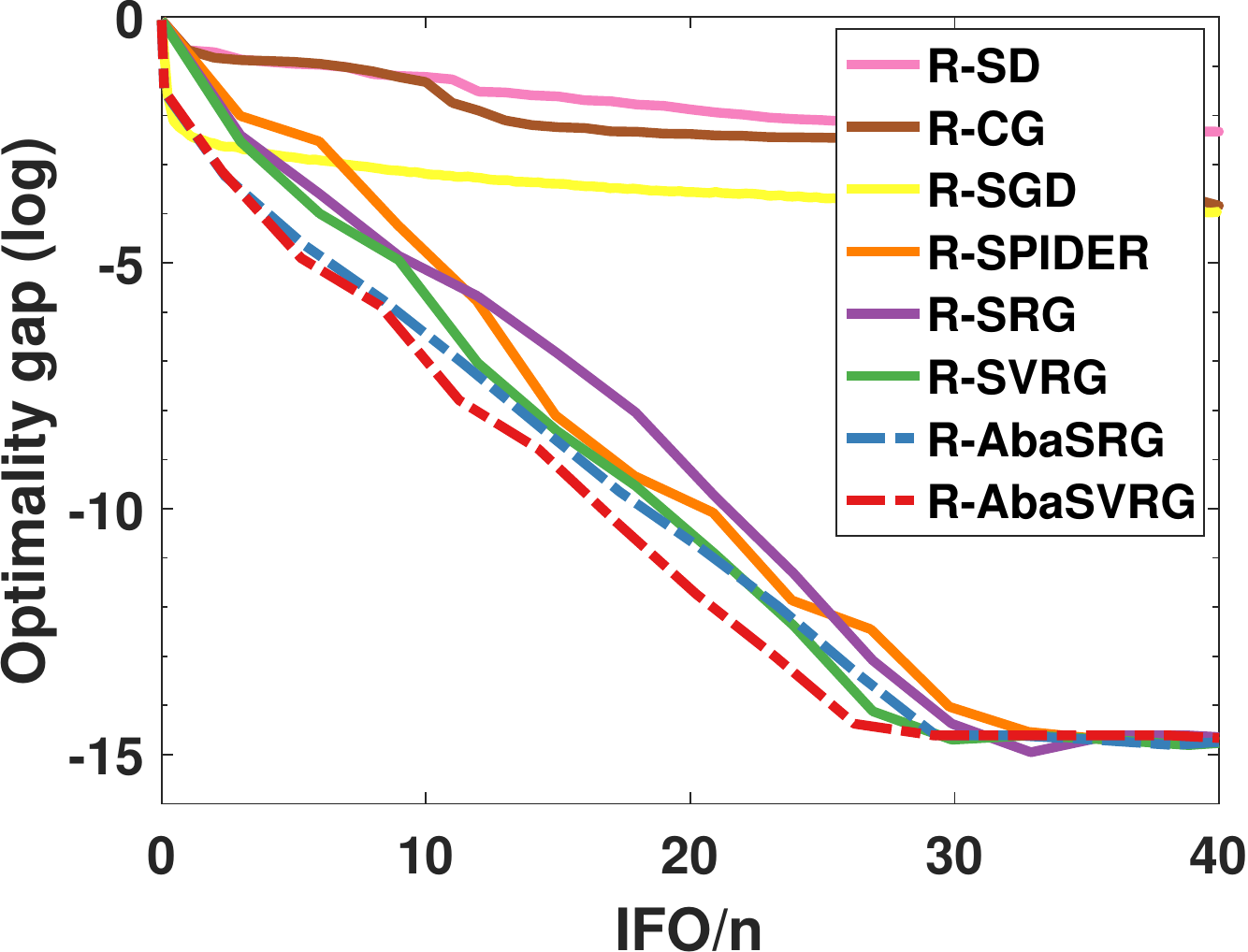}}\\
    \subfloat[Gradient norm vs. IFO (Synthetic)]{\includegraphics[width = 0.28\textwidth, height = 0.21\textwidth]{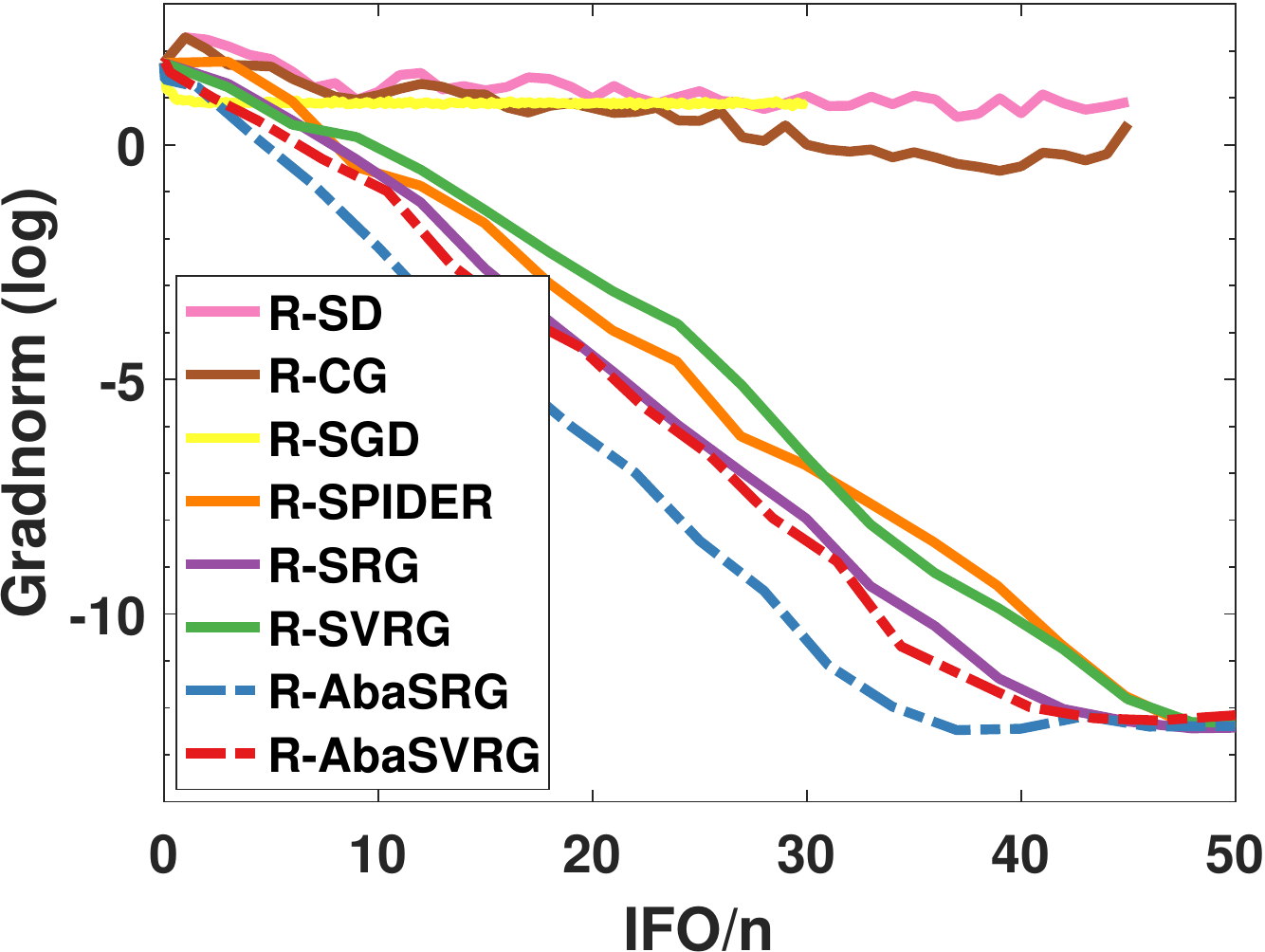}}
    \hspace{0.05in}
    \subfloat[Gradient norm vs. IFO (Mnist)]{\includegraphics[width = 0.28\textwidth, height = 0.21\textwidth]{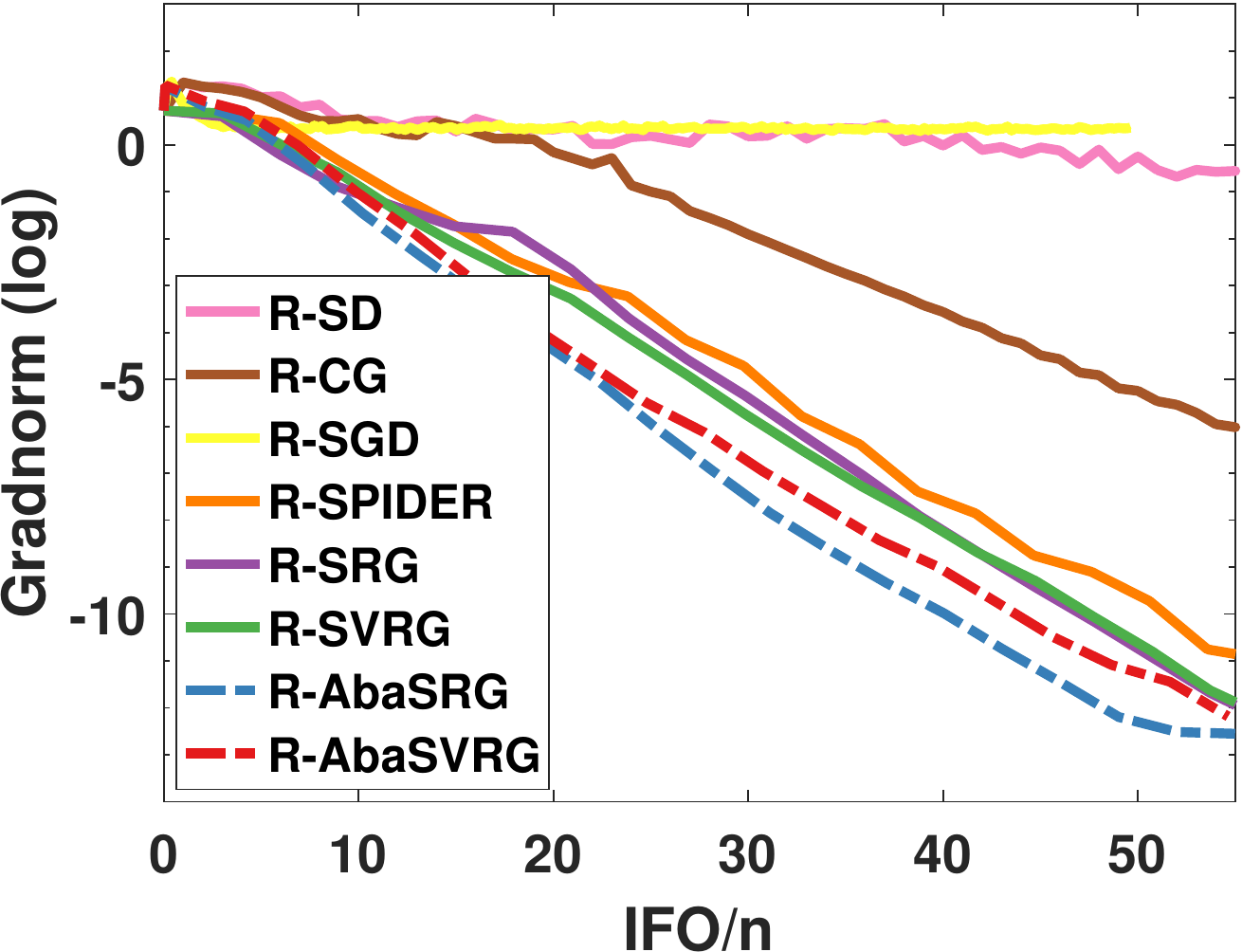}}
    \hspace{0.05in}
    \subfloat[Gradient norm vs. IFO (Ijcnn)]{\includegraphics[width = 0.28\textwidth, height = 0.21\textwidth]{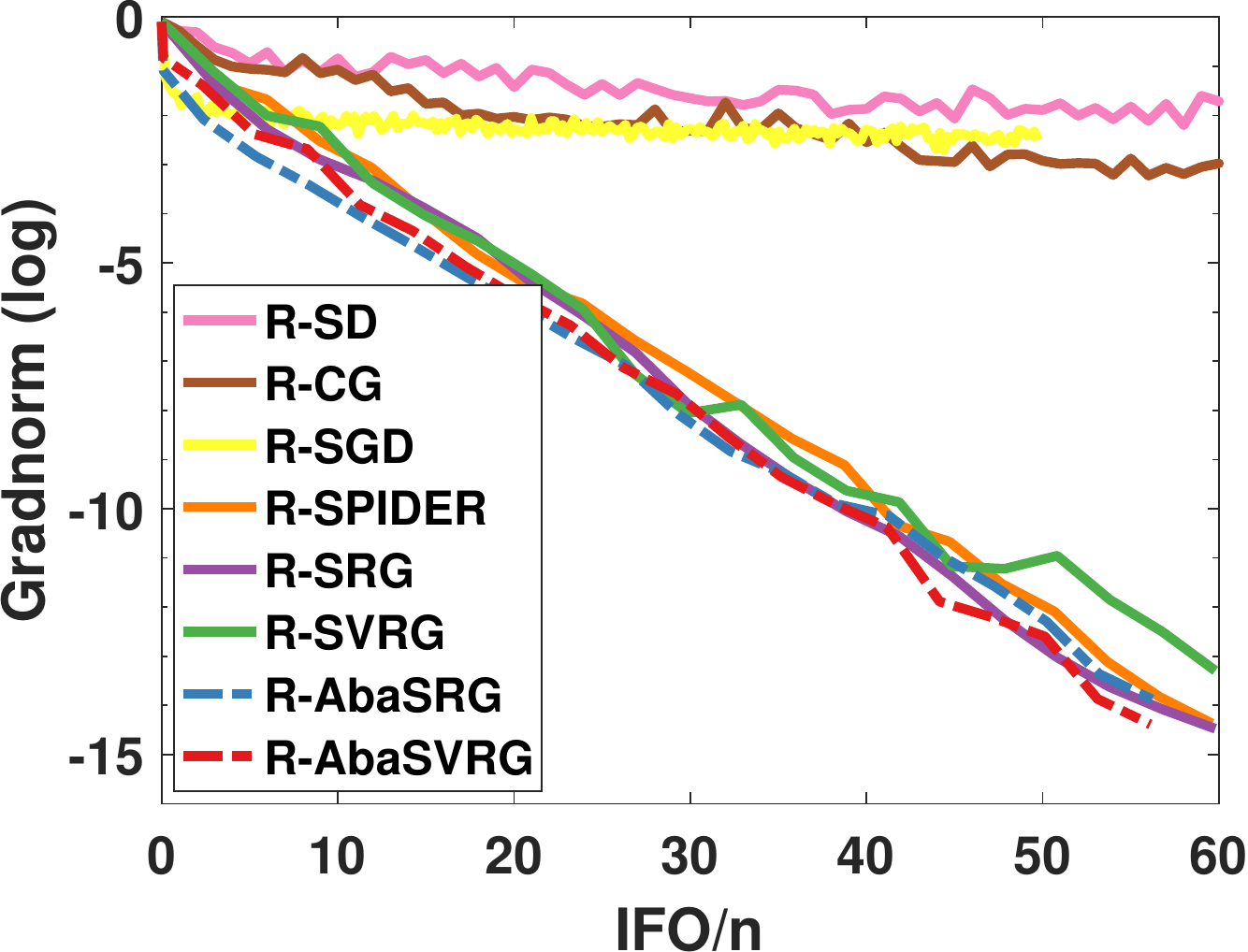}}
    \caption{PCA problem on Grassmann manifold}
    \label{pca_main_figure}
\end{figure*}

\begin{figure*}
\captionsetup{justification=centering}
    \centering
    \subfloat[Optimality gap vs. runtime (Synthetic) \label{pca_optimality_time_figure}]{\includegraphics[width = 0.28\textwidth, height = 0.21\textwidth]{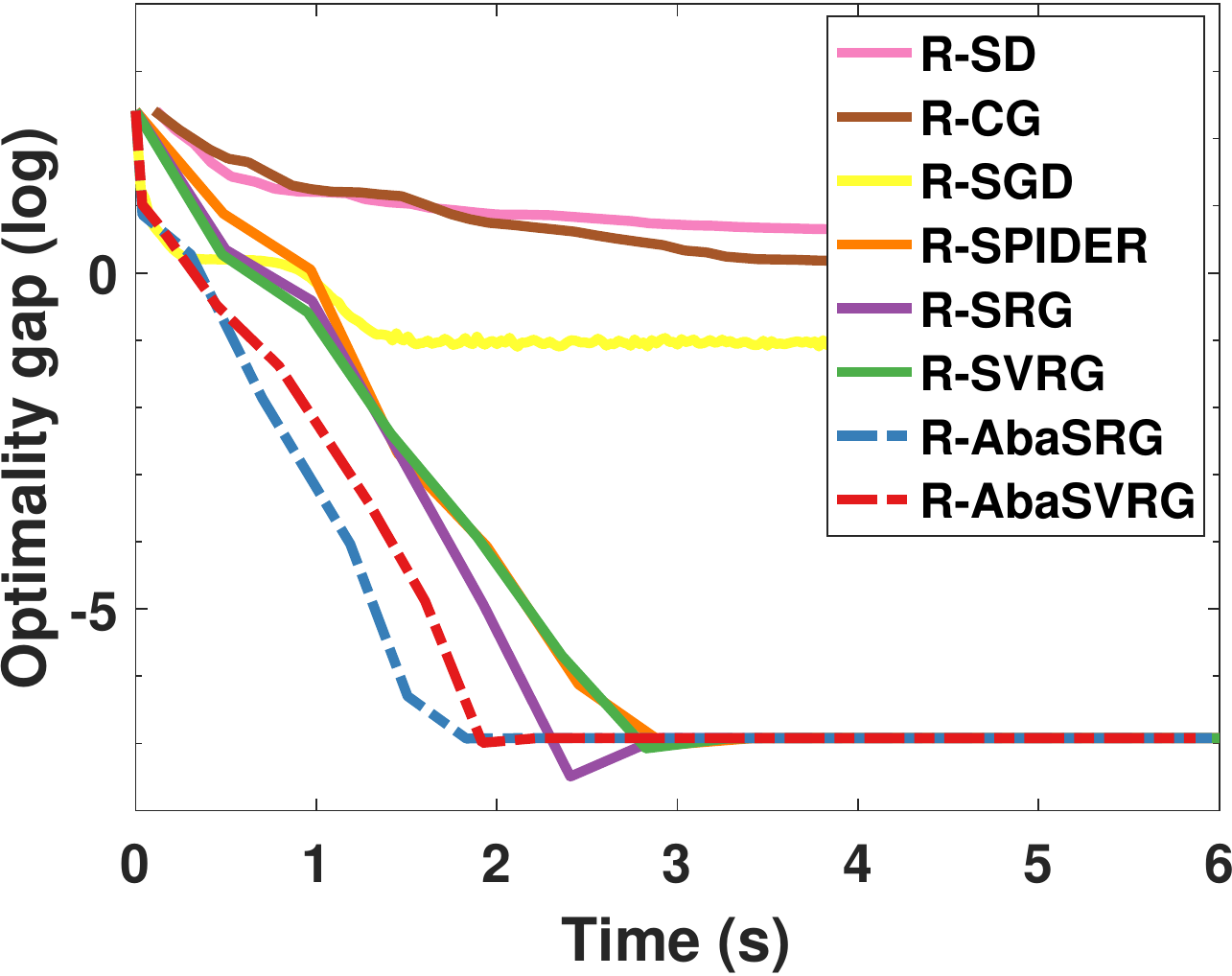}} 
    \hspace{0.05in}
    \subfloat[Sensitivity of R-AbaSVRG to $c_\beta$]{\includegraphics[width = 0.28\textwidth, height = 0.21\textwidth]{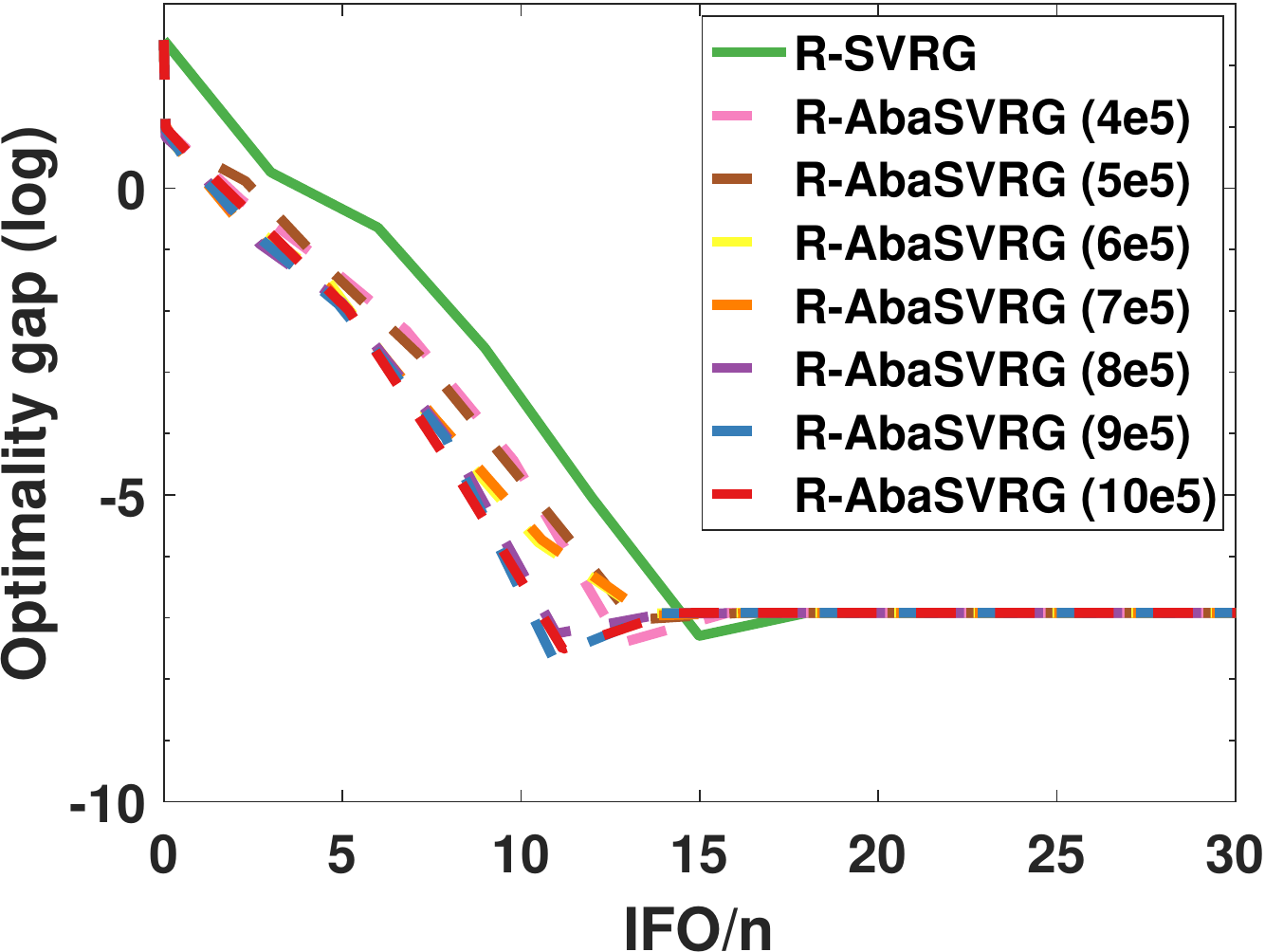}}
    \hspace{0.05in}
    \subfloat[Sensitivity of R-AbaSRG to $c_\beta$]{\includegraphics[width = 0.28\textwidth, height = 0.21\textwidth]{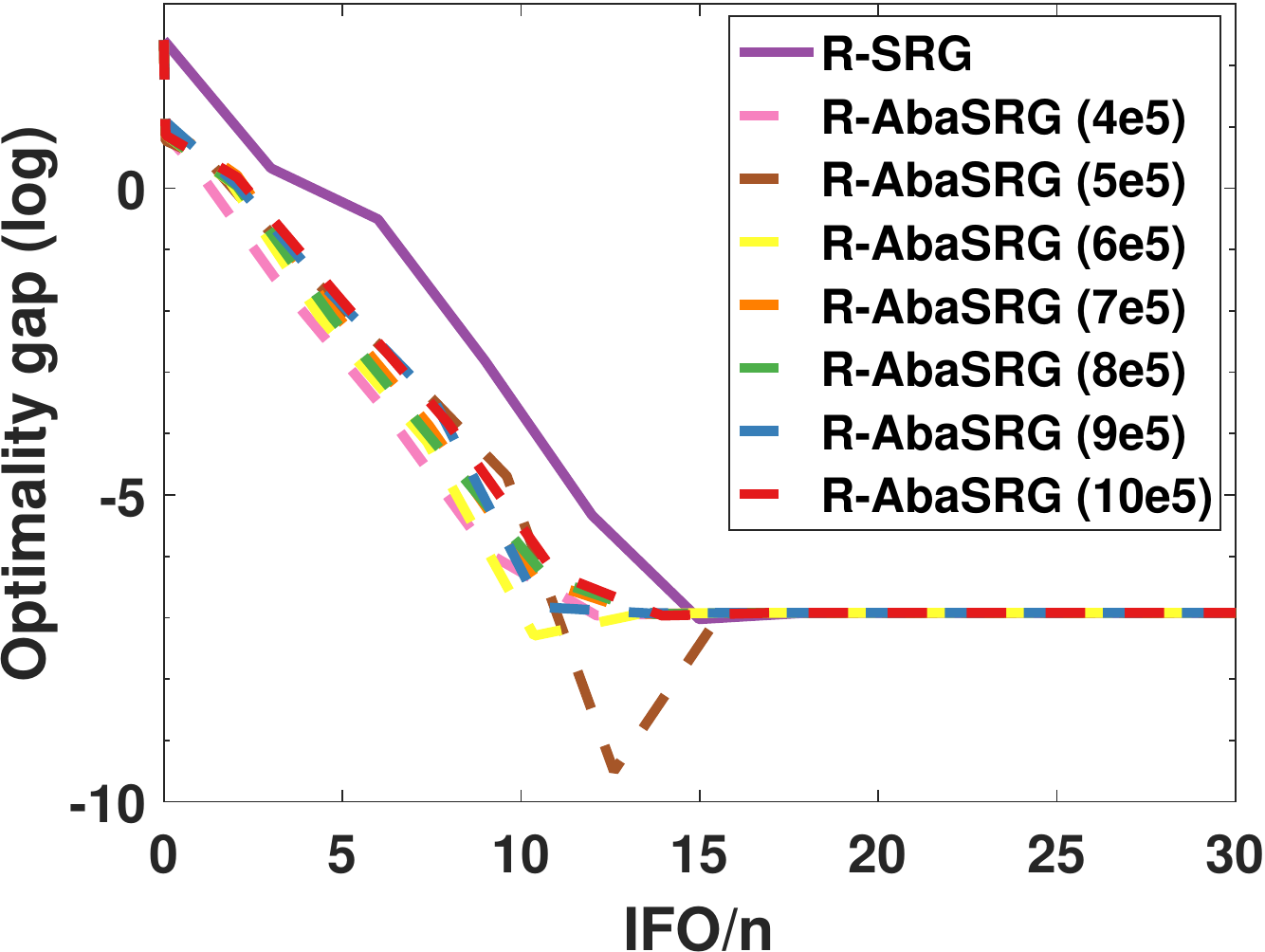}}
    \caption{Additional PCA results on synthetic dataset}
    \label{additional_pca_synthetic}
\end{figure*}

\subsection{PCA and LRMC on Grassmann manifold}
We first consider principal component analysis (PCA) and low rank matrix completion (LRMC) on Grassmann manifold $\mathcal{G}(r, d)$, which consists of $r$-dimensional subspaces in $\mathbb{R}^d$ ($r \leq d$). Points on Grassmann manifold are equivalence classes of column orthonormal matrices under the orthogonal group $ {O}(r)$. That is, a point on Grassmann manifold can be represented by a column orthonormal matrix $\mathbf U \in \mathbb{R}^{d\times r}$ such that $\mathbf U^T \mathbf U = \mathbf I_r$ and any point is deemed equivalent to $\mathbf U$ if they can be represented as $\mathbf {UR}$ for any $\mathbf R \in {O}(r)$. Recall Stiefel manifold $\text{St}(r, d)$ is the set of column orthonormal matrices in $\mathbb{R}^{d\times r}$. Grassmann manifold can also be defined as a quotient manifold of Stiefel manifold, written as $\text{St}(r,d)/ {O}(r)$.

\textbf{The PCA problem.} The PCA problem considers minimizing reconstruction error between projected and original samples over the set of orthonormal projection matrix $\mathbf U \in \text{St}(r,d)$, which is $\min_{ {\mathbf U \in \text{St}(r, d)}} \frac{1}{n} \sum_{i=1}^n \| \mathbf x_i - \mathbf{UU}^T \mathbf x_i \|^2$, where $\mathbf x_i \in \mathbb{R}^d, i = 1,...,n$ represent data samples. Note the objective function is invariant under the action of orthogonal group. That is, $f(\mathbf U) = f(\mathbf {UR})$ for $\mathbf R \in \mathcal{O}(r)$. Thus, the optimization search space is Grassmann manifold and the problem is equivalent to $\min_{ {\mathbf U \in \mathcal{G}(r, d)}} - \frac{1}{n} \sum_{i=1}^n \mathbf x_i^T \mathbf{UU}^T \mathbf x_i$.

We first consider a synthetic dataset with $(n, d, r) = (10^5, 200, 5)$, which is generated by a random normal matrix in $\mathbb{R}^{n \times d}$ with $r$ significant columns. Then we also conduct evaluations on two practical datasets, MNIST hand written digits \cite{LecunMNIST1998} with $(n, d, r) = (60000, 784, 5)$ and ijcnn1 dataset from LibSVM \cite{ChangLibsvm2011} with $(n, d, r) = (49990, 22, 5)$. We set $q = -3, l = 5$ for synthetic and MNIST datasets and $q = -1, l =2$ for ijcnn. Fig. \ref{pca_main_figure} presents convergence results for the PCA problem in terms of both optimality gap and gradient norm. Optimality gap is defined as the function value difference between iterates to the optimal point, pre-calculated by \textsc{pca} function in Matlab. From the figures, it is clear that variance reduction with batch size adaptation outperforms their full batch size versions, especially on large datasets like synthetic and MNIST. Due to small batch size in the first few epochs, R-AbaSVRG and R-AbaSRG behaves similarly to R-SGD with rapid function value decrease, while still maintaining fast convergence owing to variance reduction in the following training phases. A similar observation can be made in terms of gradient norm decrease. Fig. \ref{additional_pca_synthetic} presents additional results on synthetic dataset. Specifically, Fig. \ref{pca_optimality_time_figure} illustrates how optimality gap decreases with algorithm runtime, which aligns closely with Fig. \ref{pca_optimiality_ifo_synthetic_figure}. This suggests the extra cost of tracing gradient norm within each epoch is negligible. Also, we find that performance of R-AbaSVRG and R-AbaSRG is insensitive to parameter $c_\beta$ as long as it is sufficiently large.

\textbf{The LRMC problem.} Given a matrix $\mathbf A \in \mathbb{R}^{d \times n}$ with largely missing entries, the LRMC problem aims to recover the full matrix by assuming a low rank structure. Denote $\Omega$ as an index set corresponding to observed entries and $\mathcal{P}_\Omega$ as an operator that projects known entries while setting unknown entries to zero. Formally, $\Omega := \{ (i,j) \, | \, A_{ij} \text{ is observed} \, \}$. $\mathcal{P}_\Omega (A_{ij}) = A_{ij}$ if $(i, j) \in \Omega$ and $\mathcal{P}_\Omega (A_{ij}) = 0$ otherwise. Then the problem is to $\min_{\mathbf U, \mathbf V} \|\mathcal{P}_\Omega (\mathbf A) - \mathcal{P}_\Omega(\mathbf{UV}) \|^2$, with $\mathbf U \in \mathbb{R}^{d\times r}, \mathbf V \in \mathbb{R}^{r\times n}$. Since the factorization into $\mathbf U, \mathbf V$ is not unique and depends only on column space of $\mathbf U$, the problem is defined on Grassmann manifold $\mathcal{G}(r,d)$. Denote $\mathbf a_1, ..., \mathbf a_n$ as column vectors of $\mathbf A$ and $\mathcal{P}_{\Omega_i}, i = 1,...,n$ as the corresponding projection for the $i$-th column. We can reformulate LRMC into $\min_{\mathbf U \in \mathcal{G}(r,d), \mathbf v_i \in \mathbb{R}^r} \frac{1}{n} \sum_{i=1}^n \|\mathcal{P}_{\Omega_i} (\mathbf a_i) - \mathcal{P}_{\Omega_i}(\mathbf{U} \mathbf v_i) \|^2$. Note given $\mathbf U$, $\mathbf v_i$ has a closed form solution given by least square fit. 

A baseline synthetic dataset with $n = 20000, d = 100, r = 5$ is generated by a similar procedure in \cite{KasaiRQNVR2017}. We set condition number of the generated matrix as $\text{cn} = 50$, which is the ratio of the largest to smallest singular value. Oversampling ratio is set as $\text{os} = 8$, which determines the number of known entries given by $\text{os}\times (n + d - r)r$. The known entries are subsequently perturbed by injecting Gaussian noise with a noise level $\varepsilon = 10^{-10}$. In general, the larger the condition number, the smaller the oversampling ratio, the higher the noise level, the more difficult the LRMC problem is. In addition, we consider two movie recommendation datasets as follows. Netflix prize \cite{BennettNetflix2007} contains over $100$ million movie ratings, which are integers from $1$ to $5$. We first choose a random subset of $10$ million instances and subsequently include movies and users with more than $100$ observed entries. This leaves $1372$ movies ($n$) rated by $13088$ users $(d)$. Movielens-1M \cite{HarperMovielens2015} is a dataset with $6040$ users $(d)$ and $3706$ movies ($n$). For these two datasets, we randomly extract $20$ ratings per user as test sets, which results in $15\%$ and $12\%$ of total observed entries for testing. We set $q = -2, -5, -5$, $l = 2, 8, 8$ for synthetic, Netflix and Movielens datasets respectively. Fig. \ref{LRMC_main_results} presents test mean square error (MSE) on three datasets. We include training MSE results in Appendix \ref{LRMC_appendix}, which display similar patterns. From Fig. \ref{LRMC_main_results}, we conclude that batch size adaptation accelerates variance reduction methods particularly for the first few epochs and thus perform no worse than their non-adaptive versions. 

\begin{figure*}[!t]
\captionsetup{justification=centering}
    \centering
    \subfloat[Test MSE vs. IFO (Synthetic)]{\includegraphics[width = 0.28\textwidth, height = 0.21\textwidth]{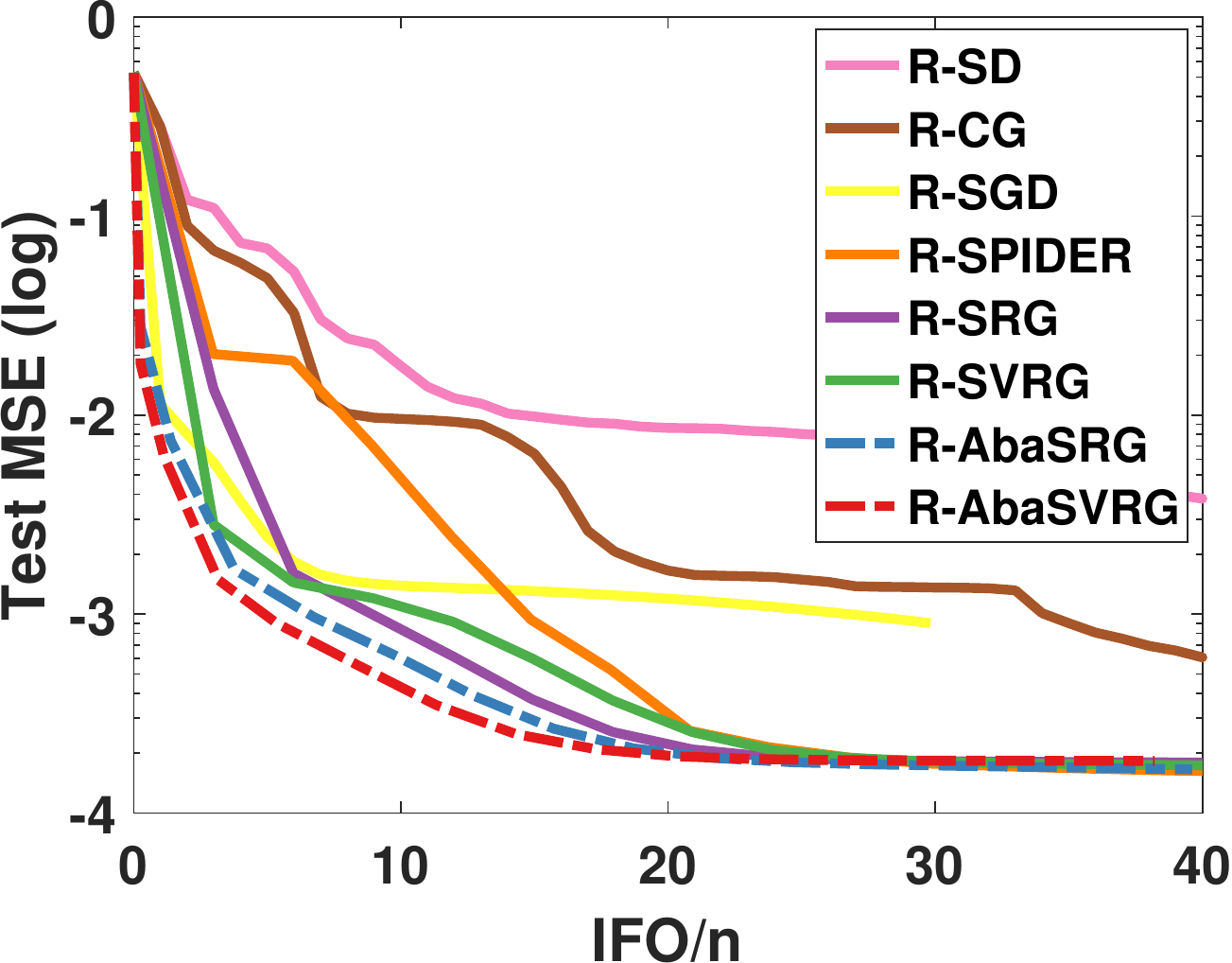}}
    \hspace{0.05in}
    \subfloat[Test MSE vs. IFO (Netflix)]{\includegraphics[width = 0.28\textwidth, height = 0.21\textwidth]{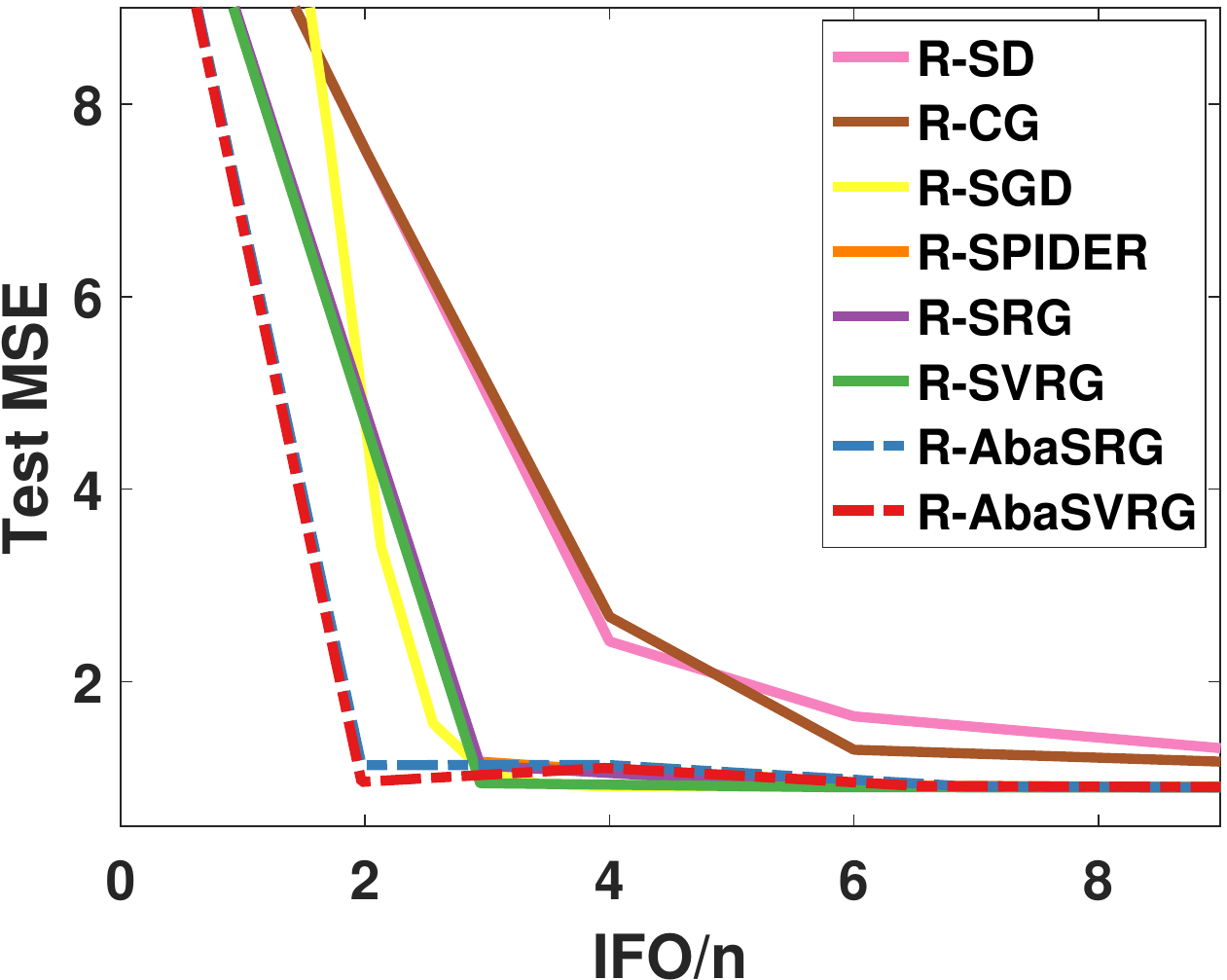}}
    \hspace{0.05in}
    \subfloat[Test MSE vs. IFO (Movielens)]{\includegraphics[width = 0.28\textwidth, height = 0.21\textwidth]{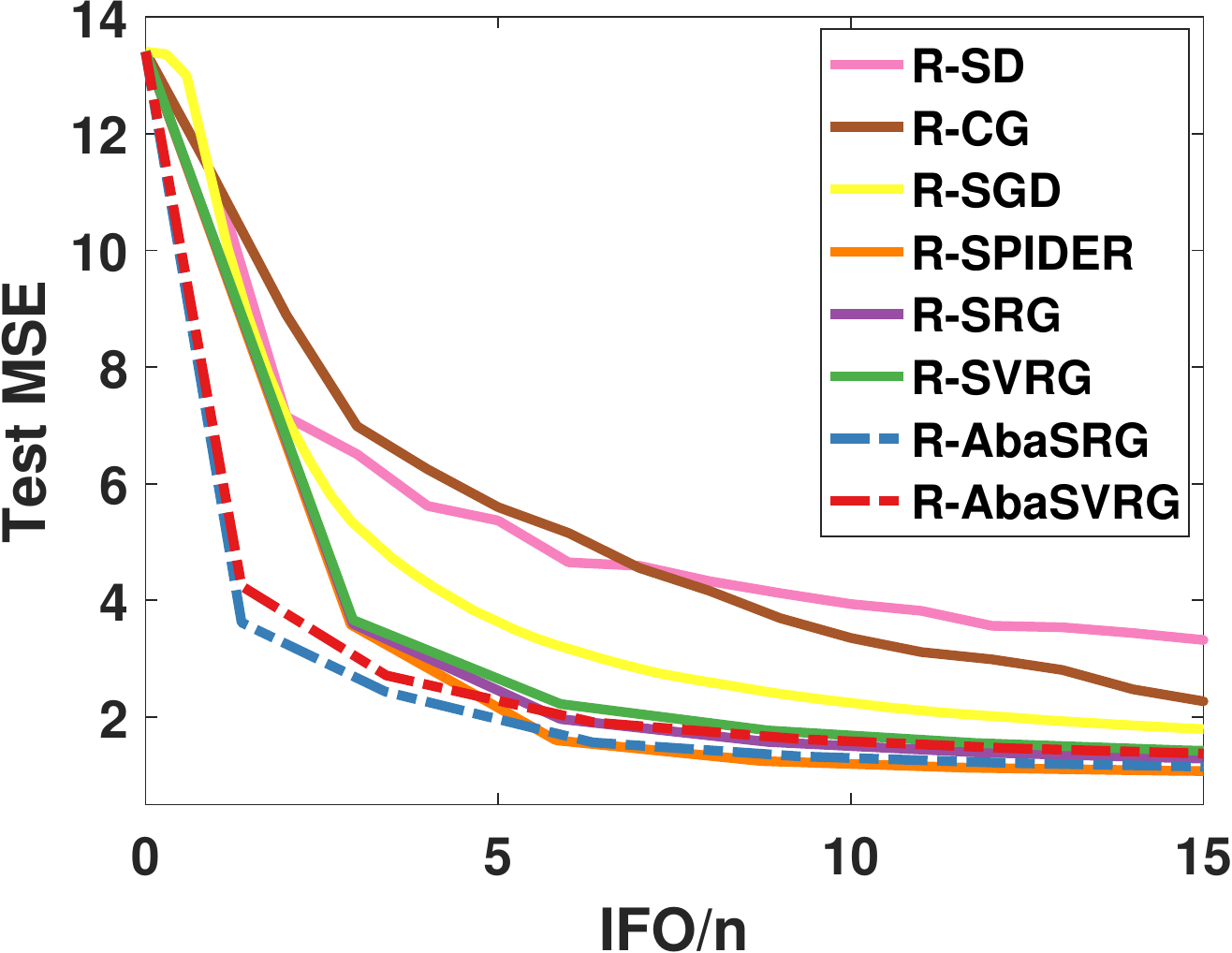}}
    \caption{LRMC problem on Grassmann manifold}
    \label{LRMC_main_results}
\end{figure*}

\begin{figure*}[!t]
\captionsetup{justification=centering}
    \centering
    \subfloat[Optimality gap vs. IFO (Synthetic)]{\includegraphics[width = 0.28\textwidth, height = 0.21\textwidth]{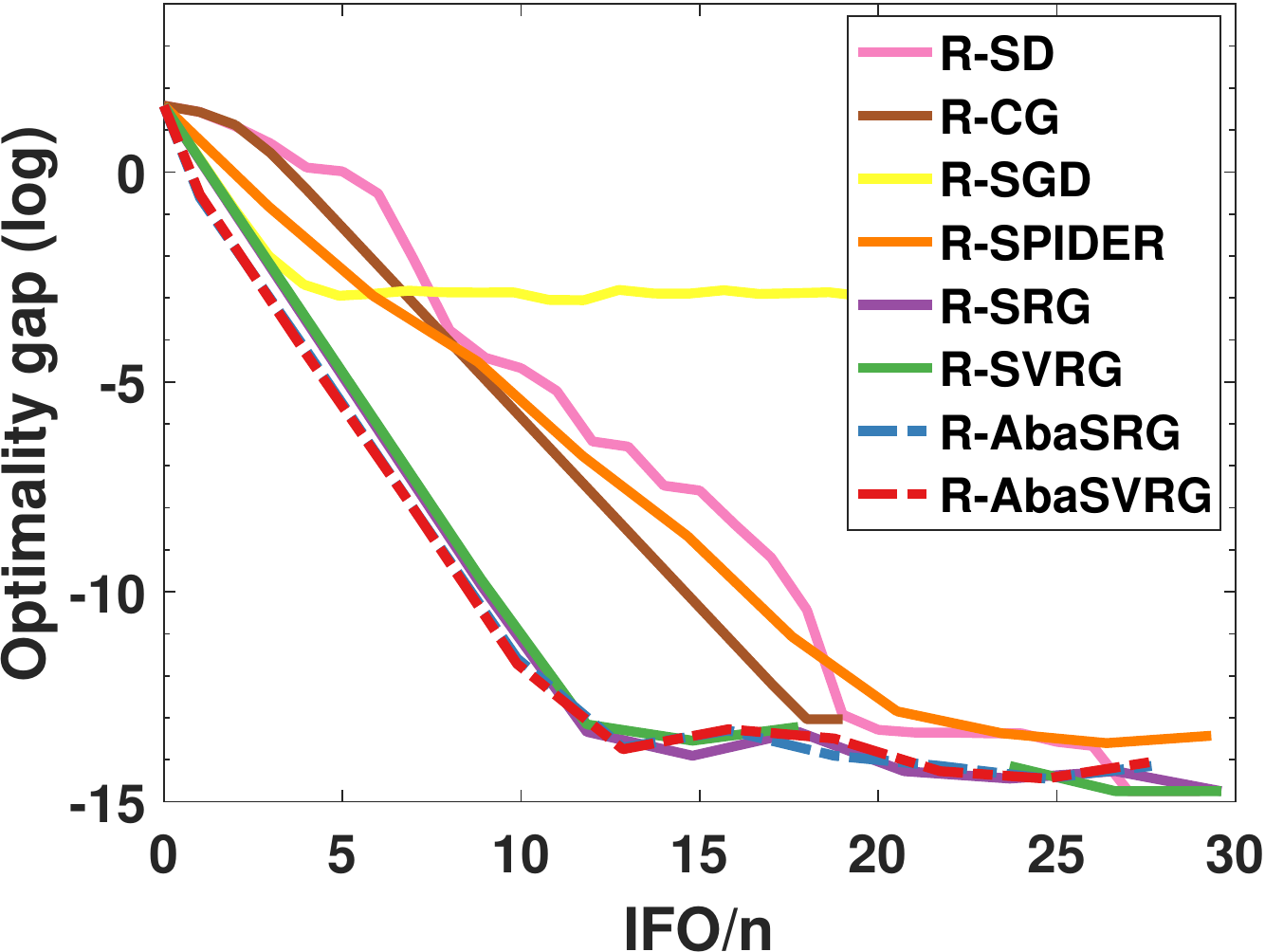}}
    \hspace{0.05in}
    \subfloat[Optimality gap vs. IFO (YaleB)]{\includegraphics[width = 0.28\textwidth, height = 0.21\textwidth]{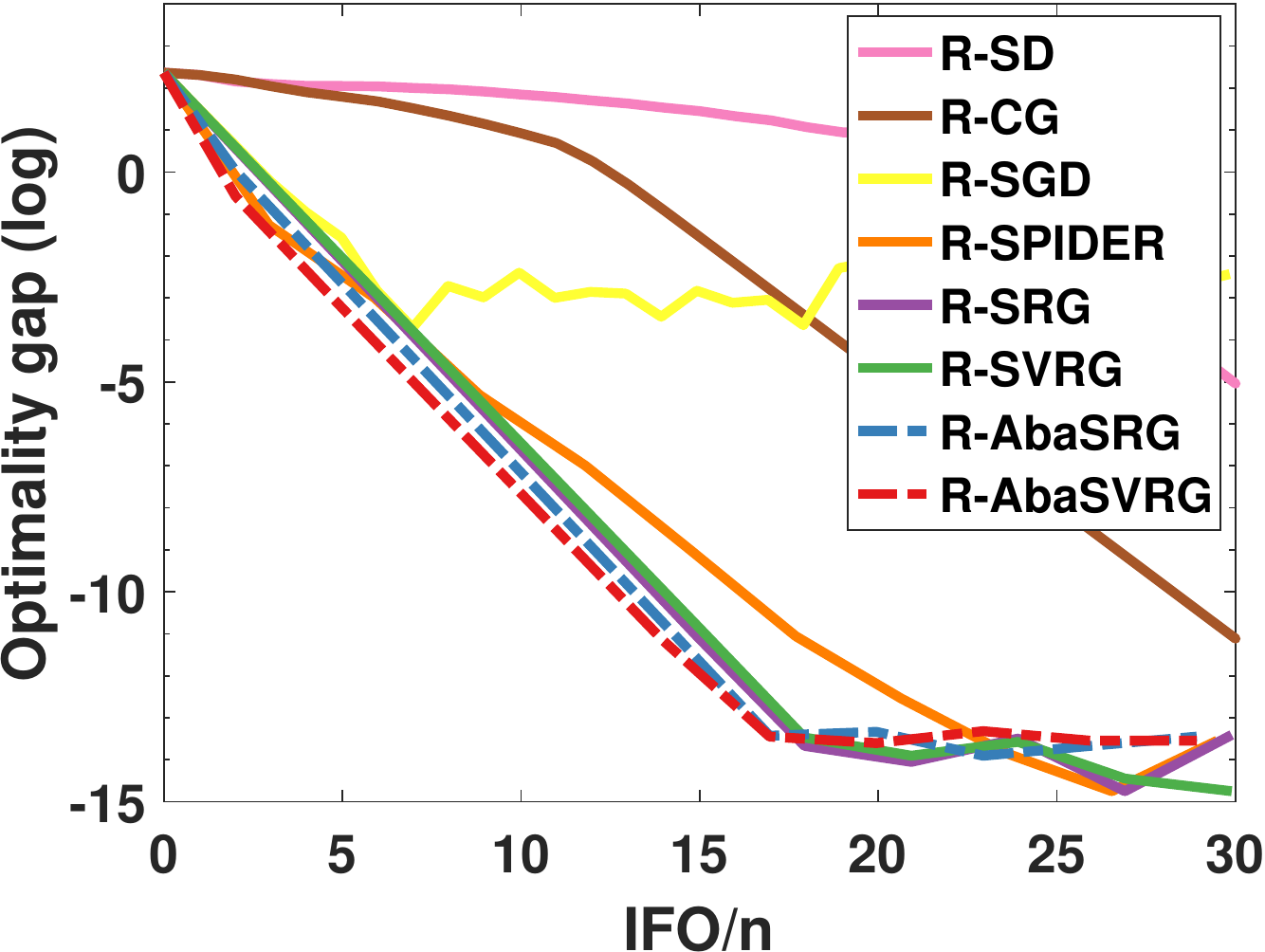}}
    \hspace{0.05in}
    \subfloat[Optimality gap vs. IFO (Kylberg)]{\includegraphics[width = 0.28\textwidth, height = 0.21\textwidth]{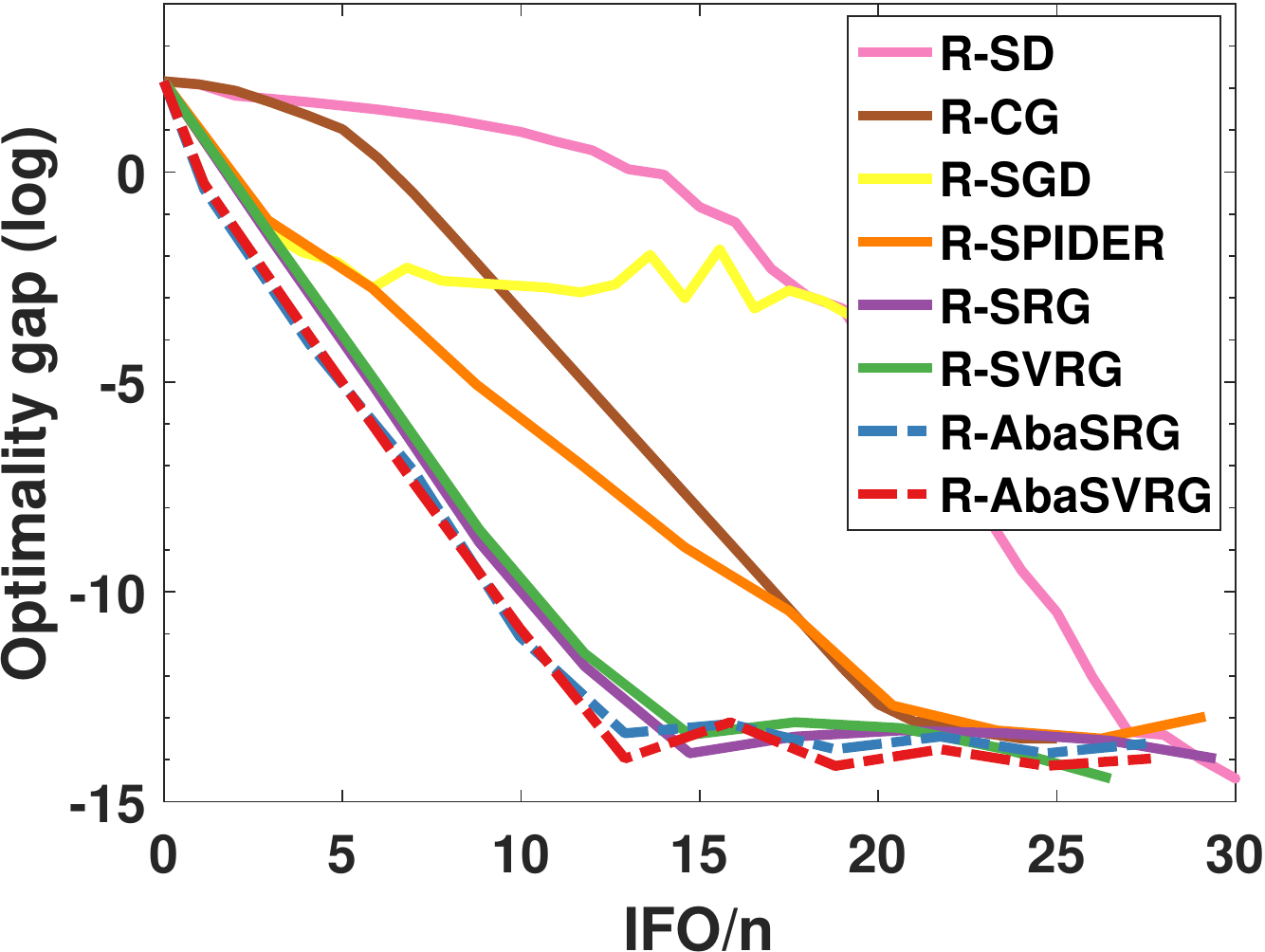}}
    \caption{RKM problem on SPD manifold}
    \label{RKM_main_results}
\end{figure*}

\subsection{RKM on SPD manifold}
We also consider computing Riemannian Karcher mean (RKM) on $d \times d$ symmetric positive definite (SPD) manifold $\mathcal{S}_{++}^d$. Given $n$ sample points $\mathbf X_1, ..., \mathbf X_n \in \mathcal{S}_{++}^d$, Riemannian Karcher mean with respect to affine-invariant Riemannian metric (AIRM) \cite{PennecAIRM2006}, is calculated by solving $\min_{\mathbf C \in  \mathcal{S}_{++}^d} \frac{1}{n} \sum_{i=1}^n \| \log(\mathbf C^{-1/2} \mathbf X_i \mathbf C^{-1/2}) \|_F^2$, where $\log(\cdot)$ represents the principal matrix logarithm. We first test on a synthetic dataset with $(n, d, \text{cn}) = (5000, 10, 20)$ generated as in \cite{BiniKM2013}. In addition, we compare algorithms on Extended Yale B dataset \cite{WrightFaceB2008} that collects $2414$ $(n)$ frontal face images of $38$ individuals under various lighting conditions and Kylberg dataset \cite{KylbergData2011} that contains $4480$ $(n)$ images of $28$ different texture classes. Original images are resized to $32 \times 32$ pixels and region covariance descriptors are constructed for each image. Particularly, we generate 8-dimensional feature vectors consisting of pixel locations, intensity, first- and second-order pixel gradients and edge orientation at each pixel location \cite{PangGabor2008}. As a result, we obtain $n$ $8 \times 8$ SPD matrices for which we calculate Riemannian Karcher mean. For all datasets, we set $q = -2, l = 5$. The optimal solution is obtained by relaxed Richardson iteration \cite{BiniKM2013}. From Fig. \ref{RKM_main_results}, we observe that R-AbaSVRG and R-AbaSRG still perform better compared to R-SVRG and R-SRG. The improvement is not as significant as in PCA and LRMC problem because all methods converge rapidly and therefore batch size adaptation only takes place in the first epoch.

\subsection{Additional remarks}
To further evaluate sensitivity of batch size adaptation, we also include results on synthetic datasets with different characteristics in Appendix \ref{additional_experiment_res_appendix} for all three applications, such as large-scale, high-dimension, high-rank, ill-conditioning. We find in general, R-AbaSVRG and R-AbaSRG are insensitive when characteristics of dataset vary and perform comparatively better across all methods considered. At last, we make some comments on R-SRG and R-SPIDER with matching complexities. We notice a similar performance for PCA and LRMC problem while R-SPIDER fails on RKM problem. {This suggests that the search grid might not be extensive enough to reflect the best performance of R-SPIDER}. For more difficult LRMC problems, we find that R-SPIDER can converge faster near optimal point (Appendix \ref{LRMC_appendix}). This is reasonable as gradient normalization allows magnitude of each step to be dictated precisely by the adaptive step size, which gives more flexibility than fix step size. However, it also requires more effort in tuning two step size parameters $\alpha_\eta, \beta_\eta$, which poses difficulty for large datasets in high dimensions.

\section{Conclusions}
In this paper, we propose R-AbaSVRG and R-AbaSRG by adapting outter loop batch size of state-of-the-art variance reduction methods R-SVRG and R-SRG for Riemannian optimization. Our formulation focuses on more general retraction and vector transport as well as mini-batch stochastic gradients. We prove that batch size adaptation maintains the same iteration complexities while requiring lower per-iteration complexities. This results in lower total complexities compared to non-adaptive methods for both general non-convex and gradient dominated functions under finite-sum and online settings. In addition, the new convergence analysis can be readily applied to non-adaptive variance reduction methods, which yields much simpler proof for R-SVRG and better complexity bounds for R-SRG under double loop convergence. Experiment results over a number of applications validate superiority of batch size adaptation.

\bibliographystyle{IEEEtran}

\newpage
\onecolumn

\par\noindent\rule{\textwidth}{1pt}
\begin{center}
    \Large \textbf{Supplementary Material}
\end{center}
\par\noindent\rule{\textwidth}{1pt}

\appendices
\section{Useful Lemmas}
\begin{lemma}[Variance bound for sampling without replacement]
\label{variance_bound_wors}
    Consider a set of population vectors $\{ \mathbf x_1,..., \mathbf x_N \}$ in $\mathbb{R}^D$ with $\sum_{i=1}^N \mathbf x_i = 0$, and a subset $\mathcal{I}$ of cardinality $b$, which is uniformly drawn at random from $[N]$ without replacement. Then 
    \begin{equation}
        \mathbb{E}_{\mathcal{I}} \| \frac{1}{b}\sum_{i\in \mathcal{I}} x_i \|^2 \leq \frac{1}{Nb} \frac{N-b}{N-1} \sum_{i=1}^N \| x_i\|^2. \nonumber
    \end{equation}
\end{lemma}
\begin{proof}
    See Lemma A.1 in \cite{LeiSCSG2017}.
\end{proof}

\begin{lemma}[Retraction Lipschitzness with vector transport]
\label{vec_trans_lemma}
    Suppose $f$ is average retraction $L_l$-Lipschitz as in Assumption \ref{retaction_Ll_assump} and norm of gradient is bounded by $G$. Also suppose difference between parallel transport $P^x_y$ and vector transport $\mathcal{T}^x_y$ under same retraction is bounded as in Assumption \ref{diff_vec_par_bound}. Then for all $x, y = R_x(\xi) \in \mathcal{X}$, 
    $$ \mathbb{E}\| \emph{grad}f_i(x) - \mathcal{T}_y^x \emph{grad}f_i(y) \| \leq (L_l + \theta G) \|\xi \|,$$
    where expectation is taken with respect to index $i$ and $\theta$ is parameter defined in Assumption \ref{diff_vec_par_bound}.
\end{lemma}
\begin{proof}
    \begin{align}
        \mathbb{E}\| \text{grad}f_i(x) - \mathcal{T}_y^x \text{grad}f_i(y) \| &= \mathbb{E}\| \text{grad}f_i(x) - P_y^x \text{grad}f_i(y) + P_y^x \text{grad}f_i(y) - \mathcal{T}_y^x \text{grad}f_i(y)  \| \nonumber\\
        &\leq \mathbb{E}\|\text{grad}f_i(x) - P_y^x \text{grad}f_i(y) \| + \mathbb{E}\|P_y^x \text{grad}f_i(y) - \mathcal{T}_y^x \text{grad}f_i(y)  \| \nonumber\\
        &\leq L_l \| \xi\| + \theta \| \xi\| \mathbb{E}\| \text{grad}f_i(y)\| \nonumber\\
        &\leq (L_l + \theta G) \| \xi\|,
    \end{align}
    where the first inequality is by triangle inequality and the last two inequalities follow from Assumptions \ref{retaction_Ll_assump} and \ref{diff_vec_par_bound} and the bounded gradient.
\end{proof}

\newpage

\section{Proof of Theorem \ref{RSVRG_original_theorem}}
\label{proof_of_theorem_chapter}

The idea of the proof is mainly based on \cite{ReddiSVRG2016, ZhangRSVRG2016}. We first present a trigonometric distance bound \cite{ZhangFOGeodesicConvex2016} that extends law of cosines on Euclidean space to Riemannian manifold with bounded sectional curvature. Next we show that the norm of gradient is bounded by difference in a properly constructed Lyapunov function. Then telescoping this result completes the proof.

\begin{lemma}[Trigonometric distance bound]
\label{trig_distance_bound_lemma}
If $a,b,c$ are side lengths of a geodesic triangle in a length space with curvature lower bounded by $\kappa$, and $\theta$ is the angle between sides $b$ and $c$,
    $$a^2 \leq \frac{\sqrt{|\kappa|} c}{\tanh(\sqrt{|\kappa|}c)} b^2 + c^2 - 2bc \cos(\theta).$$
Assume Assumption  \ref{R_svrg_additional_assump} holds and define the following curvature constant
\begin{equation}
    \zeta := \begin{cases} \frac{\sqrt{|\kappa|} D}{\tanh(\sqrt{|\kappa|}D)}, & \text{ if } \kappa < 0\\
    1, & \text{ if } \kappa \geq 0 \end{cases} \nonumber
\end{equation}
where $D$ is the diameter of compact set $\mathcal{X}$. Then for $a,b,c$ as side lengths of a geodesic triangle in $\mathcal{X}$,
$$a^2 \leq \zeta b^2 + c^2 - 2bc \cos(\theta).$$
\end{lemma}

\begin{proof}
    See Lemma 5 in \cite{ZhangFOGeodesicConvex2016}. 
\end{proof}

\begin{lemma}
\label{prelim_lemma_svrg_1}
Suppose Assumptions \ref{common_assump}, \ref{assumption_svrg} and \ref{R_svrg_additional_assump} hold. Let
\begin{align}
    c_t &= c_{t+1} + c_{t+1}\eta \lambda + c_{t+1}\frac{(\zeta \nu^2 + 2c_R D)(L_l + \theta G)^2\mu^2 \eta^2}{b} + \frac{L (L_l + \theta G)^2 \mu^2 \eta^2}{2b}, \nonumber\\
    \delta_t &= \eta - \frac{c_{t+1}\eta}{\lambda} - \frac{L \eta^2}{2} - c_{t+1}(\zeta \nu^2 + 2c_R D) \eta^2. \nonumber
\end{align}
Suppose we choose $\{ c_t \}, \eta$ and $\lambda > 0$ such that $\delta_t >0$. Then iterate sequence $\{ x_t^s \}$ produced by Algorithm \ref{RadaSVRG_algorithm} with full batch gradient $B^s = n$ satisfies 
\begin{equation}
    \| \emph{grad}f(x_t^s) \|^2 \leq \frac{\mathbb{E}[ R_t^s - R^s_{t+1} |\mathcal{F}_t^s]}{\delta_t}, \nonumber
\end{equation}
with $R_t^s := f(x_t^s) + c_t d^2(x_t^s, x_0^s)$, for $s = 1,..., S, t = 0,...,m-1$. 
\end{lemma}

\begin{proof}
By retraction $L$-smoothness and taking expectation with respect to $\mathcal{F}_t^s$, we have
\begin{align}
    \mathbb{E}[f(x_{t+1}^s) | \mathcal{F}_t^s] &\leq  f(x_t^s) -\eta \langle \text{grad}f(x_t^s), \mathbb{E}[v_t^s | \mathcal{F}_t^s] \rangle + \frac{L \eta^2}{2}\mathbb{E}[\| v_t^s \|^2 | \mathcal{F}_t^s] \nonumber\\
    &= f(x_t^s) - \eta \| \text{grad}f(x_t^s) \|^2 + \frac{L\eta^2}{2}\mathbb{E}[\| v_t^s \|^2 | \mathcal{F}_t^s]
\end{align}
We first establish a bound on norm of modified gradient $v_t^s$. 
\begin{align}
    &\mathbb{E}[ \| v_t^s\|^2 | \mathcal{F}_t^s] \nonumber\\
    &= \mathbb{E}[ \| \text{grad}f_{\mathcal{I}_t^s}(x_{t}^{s}) - \mathcal{T}_{{x}_0^s}^{x_{t}^{s}} \big( \text{grad}f_{\mathcal{I}_t^s}(x_0^s) - v_0^s \big) \|^2 | \mathcal{F}_t^s] \nonumber\\
    &= \mathbb{E}[\|\text{grad}f_{\mathcal{I}_t^s}(x_{t}^{s}) - \mathcal{T}_{{x}_0^s}^{x_{t}^{s}} \text{grad}f_{\mathcal{I}_t^s}(x_0^s) - \text{grad}f(x_t^s) + \mathcal{T}_{{x}_0^s}^{x_{t}^{s}} \text{grad}f(x_0^s) + \text{grad}f(x_t^s) \|^2 | \mathcal{F}_t^s] \nonumber\\
    &= \mathbb{E}[ \| \text{grad}f_{\mathcal{I}_t^s}(x_{t}^{s}) - \mathcal{T}_{{x}_0^s}^{x_{t}^{s}} \text{grad}f_{\mathcal{I}_t^s}(x_0^s) - \text{grad}f(x_t^s) + \mathcal{T}_{{x}_0^s}^{x_{t}^{s}} \text{grad}f(x_0^s) \|^2 | \mathcal{F}_t^s] + \|\text{grad}f(x_t^s) \|^2 \nonumber\\
    &\leq \mathbb{E}[ \| \text{grad}f_{\mathcal{I}_t^s}(x_{t}^{s}) - \mathcal{T}_{{x}_0^s}^{x_{t}^{s}} \text{grad}f_{\mathcal{I}_t^s}(x_0^s) \|^2 | \mathcal{F}_t^s] + \| \text{grad}f(x_t^s) \|^2 \nonumber\\
    &\leq \frac{(L_l + \theta G)^2\mu^2}{b} d^2(x_t^s, x_0^s) + \| \text{grad}f(x_t^s) \|^2,
\end{align}
where third equality is due to unbiasedness of stochastic gradient. The first inequality holds due to $\mathbb{E}\|x - \mathbb{E}[x] \|^2 \leq \mathbb{E}\|x \|^2$ and the last inequality is by Lemma \ref{vec_trans_lemma} ans Assumption \ref{relation_distance_assump}. Then we use Lemma \ref{trig_distance_bound_lemma} to bound distance $d^2(x_{t+1}^s, x_0^s)$. For a geodesic triangle $\triangle x_{t+1}^s x_t^s x_0^s$, we have
\begin{align}
    \mathbb{E}[d^2(x_{t+1}^s, x_0^s)|\mathcal{F}_t^s] &\leq \mathbb{E}[\zeta d^2(x_{t+1}^s, x_t^s) + d^2(x_t^s, x_0^s) - 2\langle \text{Exp}^{-1}_{x_t^s}(x_{t+1}^s), \text{Exp}^{-1}_{x_t^s} (x_0^s) \rangle | \mathcal{F}_t^s] \nonumber\\
    &\leq \mathbb{E}[\zeta \nu^2 \eta^2 \| v_t^s\|^2 + d^2(x_t^s, x_0^s) - 2\langle \text{Exp}^{-1}_{x_t^s}(x_{t+1}^s), \text{Exp}^{-1}_{x_t^s} (x_0^s) \rangle  | \mathcal{F}_t^s], \label{temp_trig_bound}
\end{align}
where the second inequality is by Assumption \ref{relation_distance_assump}. Also note that 
\begin{align}
    -2\langle  \text{Exp}^{-1}_{x_t^s}(x_{t+1}^s), \text{Exp}^{-1}_{x_t^s}(x_0^s) \rangle &= 2\langle R^{-1}_{x_t^s}(x_{t+1}^s) - \text{Exp}^{-1}_{x_t^s}(x_{t+1}^s), \text{Exp}^{-1}_{x_t^s}(x_0^s) \rangle - 2\langle  R^{-1}_{x_t^s}(x_{t+1}^s) , \text{Exp}^{-1}_{x_t^s}(x_0^s) \rangle \nonumber\\
    &\leq 2\| R^{-1}_{x_t^s}(x_{t+1}^s) - \text{Exp}^{-1}_{x_t^s}(x_{t+1}^s)\| \| \text{Exp}^{-1}_{x_t^s}(x_0^s)\| + 2 \eta \langle v_t^s, \text{Exp}^{-1}_{x_t^s}(x_0^s) \rangle \nonumber\\
    &\leq 2c_R D \eta^2 \| v_t^s\|^2 + 2 \eta \langle v_t^s, \text{Exp}^{-1}_{x_t^s}(x_0^s) \rangle,
\end{align}
where the last inequality uses Assumption \ref{compact_manifold_assump} and \ref{close_retract_exp_assump} with $\| \text{Exp}^{-1}_{x_t^s}(x_0^s) \| = d(x_t^s, x_0^s) \leq D$. Substitute this result back to \eqref{temp_trig_bound} gives
\begin{align}
    \mathbb{E}[d^2(x_{t+1}^s, x_0^s)|\mathcal{F}_t^s] &\leq \mathbb{E}[ (\zeta \nu^2 + 2c_R D) \eta^2 \| v_t^s\|^2 + d^2(x_t^s, x_0^s) + 2 \eta \langle v_t^s, \text{Exp}^{-1}_{x_t^s}(x_0^s) \rangle | \mathcal{F}_t^s] \nonumber\\
    &= (\zeta \nu^2 + 2c_R D) \eta^2 \mathbb{E}[\| v_t^s\|^2 | \mathcal{F}_t^s ] + d^2(x_t^s, x_0^s) + 2\eta \langle \text{grad}f(x_t^s), \text{Exp}^{-1}_{x_t^s}(x_0^s) \rangle \nonumber\\
    &\leq (\zeta \nu^2 + 2c_R D) \eta^2 \mathbb{E}[\| v_t^s\|^2 | \mathcal{F}_t^s ] + d^2(x_t^s, x_0^s) + 2\eta( \frac{1}{2\lambda} \| \text{grad}f(x_t^s) \|^2 + \frac{\lambda}{2} \|\text{Exp}^{-1}_{x_t^s}(x_0^s) \|^2 ) \nonumber\\
    &= (\zeta \nu^2 + 2c_R D) \eta^2 \mathbb{E}[\| v_t^s\|^2 | \mathcal{F}_t^s ]  + (1+ \eta\lambda) d^2(x_t^s, x_0^s) + \frac{\eta}{\lambda} \| \text{grad}f(x_t^s) \|^2.
\end{align}
The second inequality is due to Young's inequality $\langle a, b \rangle \leq \frac{1}{2\lambda} \|b \|^2 + \frac{\lambda}{2} \|a\|^2$ with parameter $\lambda > 0$. Now construct a Lyapunov function $R_t^s := f(x_t^s) + c_t d^2(x_t^s, x_0^s)$. Then, 
\begin{align}
    \mathbb{E}[ R^s_{t+1} | \mathcal{F}_t^s] &= \mathbb{E}[ f(x_{t+1}^s) + c_{t+1} d^2(x_{t+1}^s, x_0^s) | \mathcal{F}_t^s] \nonumber\\
    &\leq f(x_t^s) - \eta \| \text{grad}f(x_t^s) \|^2 + \frac{L\eta^2}{2}\mathbb{E}[\| v_t^s \|^2 | \mathcal{F}_t^s] \nonumber\\
    &+ c_{t+1} \big( (\zeta \nu^2 + 2c_R D) \eta^2 \mathbb{E}[\| v_t^s\|^2 | \mathcal{F}_t^s ]  + (1+ \eta\lambda) d^2(x_t^s, x_0^s) + \frac{\eta}{\lambda} \| \text{grad}f(x_t^s) \|^2 \big) \nonumber\\
    &= f(x_t^s) - (\eta - \frac{c_{t+1}\eta}{\lambda}) \| \text{grad}f(x_t^s) \|^2 + (c_{t+1} + c_{t+1}\eta \lambda) d^2(x_t^s, x_0^s) \nonumber\\
    &+ \big( \frac{L \eta^2}{2} + c_{t+1}(\zeta \nu^2 + 2c_R D) \eta^2  \big) \mathbb{E}[\| v_t^s\|^2 | \mathcal{F}_t^s ] \nonumber\\
    &\leq f(x_t^s) - (\eta - \frac{c_{t+1}\eta}{\lambda}) \| \text{grad}f(x_t^s) \|^2 + (c_{t+1} + c_{t+1}\eta \lambda) d^2(x_t^s, x_0^s) \nonumber\\
    &+ \big( \frac{L \eta^2}{2} + c_{t+1}(\zeta \nu^2 + 2c_R D) \eta^2  \big) \big(  \frac{(L_l + \theta G)^2\mu^2}{b} d^2(x_t^s, x_0^s) + \| \text{grad}f(x_t^s) \|^2 \big) \nonumber\\
    &= f(x_t^s) - \big( \eta - \frac{c_{t+1}\eta}{\lambda} - \frac{L \eta^2}{2} - c_{t+1}(\zeta \nu^2 + 2c_R D) \eta^2 \big) \| \text{grad}f(x_t^s) \|^2 \nonumber\\
    &+ \big( c_{t+1} + c_{t+1}\eta \lambda + c_{t+1}\frac{(\zeta \nu^2 + 2c_R D)(L_l + \theta G)^2\mu^2 \eta^2}{b} + \frac{L (L_l + \theta G)^2 \mu^2 \eta^2}{2b} \big) d^2(x_t^s, x_0^s) \nonumber\\
    &= R_t^s - \delta_t \| \text{grad}f(x_t^s) \|^2,
\end{align}
with $c_t = c_{t+1} + c_{t+1}\eta \lambda + c_{t+1}\frac{(\zeta \nu^2 + 2c_R D)(L_l + \theta G)^2\mu^2 \eta^2}{b} + \frac{L (L_l + \theta G)^2 \mu^2 \eta^2}{2b}$ and $\delta_t := \eta - \frac{c_{t+1}\eta}{\lambda} - \frac{L \eta^2}{2} - c_{t+1}(\zeta \nu^2 + 2c_R D) \eta^2$. Suppose we choose parameters such that $\delta_t > 0$. Then, we have the desired result. 
\end{proof}

\begin{lemma}
\label{prelim_lemma_svrg_2}
    With the same assumptions and settings in Lemma \ref{prelim_lemma_svrg_1}, choose $c_m = 0$ and define $\Tilde{\delta} := \min_{0 \leq t \leq m-1} \delta_t$. Denote $T = Sm$ as the total number of iterations and $\Delta = f(\Tilde{x}^0) - f(x^*)$. Then output $\Tilde{x}$ from Algorithm \ref{RadaSVRG_algorithm} with full batch gradient $B^s = n$ satisfies
    $$\mathbb{E}\| \emph{grad}f(\Tilde{x}) \|^2 \leq \frac{\Delta}{T\Tilde{\delta}}.$$
\end{lemma}
\begin{proof}
    Summing over result over $t = 0,...,m-1$ from Lemma \ref{prelim_lemma_svrg_1} and taking expectation with respect to $\mathcal{F}_0^s$ yields
    \begin{equation}
        \sum_{t=0}^{m-1} \mathbb{E}[\| \text{grad}f(x_t^s) \|^2 | \mathcal{F}_0^s] \leq \frac{\mathbb{E}[ R_0^s - R_m^s | \mathcal{F}_0^s]}{\Tilde{\delta}} = \frac{\mathbb{E}[ f(x_0^s) - f(x_0^{s+1}) | \mathcal{F}_0^s]}{\Tilde{\delta}},
    \end{equation}
    where we note that $R_0^s = f(x_0^s)$ and $R_m^s = f(x_m^s) = f(x_0^{s+1})$ for $c_m = 0$. Telescoping this inequality from $s = 1,..., S$ and taking full expectation, we have
    \begin{equation}
        \frac{1}{T} \sum_{s=1}^S \sum_{t=0}^{m-1} \mathbb{E}\| \text{grad}f(x_t^s) \|^2 \leq \frac{ f(\Tilde{x}^0) - \mathbb{E}[f(x_m^S)] }{T\Tilde{\delta}} \leq \frac{\Delta}{T\Tilde{\delta}}. 
    \end{equation}
    Finally, by noting that output $\Tilde{x}$ is uniformly drawn at random from all iterates and thus $\mathbb{E}\| \text{grad}f(\Tilde{x}) \|^2 = \frac{1}{T}\sum_{s=1}^S \sum_{t=0}^{m-1} \mathbb{E}\| \text{grad}f(x_t^s) \|^2$, the proof is complete.
\end{proof}

\raggedright \textbf{Theorem \ref{RSVRG_original_theorem} (Convergence and complexity of R-SVRG under standard analysis).}
Suppose Assumptions \ref{common_assump}, \ref{assumption_svrg} and \ref{R_svrg_additional_assump} hold and consider Algorithm \ref{RadaSVRG_algorithm} with full batch gradient $B^s = n$. Choose step size $\eta = \frac{\mu_0 b}{(L_l + \theta G)\mu n^{a_1}(\zeta \nu^2 + 2 c_R D)^{a_2}}$, $m = \lfloor n^{3/2a_1}/ 2b\mu_0 (\zeta \nu^2 + 2 c_R D)^{1-2a_2} \rfloor$, $b \leq n^{a_1}$, where $a_1, \mu_0 \in (0,1), a_2 \in (0,2)$. Then for a constant $\psi > 0$ such that
\begin{align}
    \psi &\leq \frac{\mu_0}{\mu} \Big( 1 - \frac{L \mu_0 (e-1)}{2(L_l + \theta G)(\zeta \nu^2 + 2c_R D)^{2-a_2}\mu} - \frac{L \mu_0 b}{2(L_l + \theta G)(\zeta \nu^2 + 2c_R D)^{a_2} \mu n^{a_1}} - \frac{L\mu_0^2 (e-1) b}{2 (L_l + \theta G)(\zeta \nu^2 + 2c_R D)^{a_2}\mu n^{3/2 a_1}} \Big),
\end{align}
the output $\Tilde{x}$ after running $T= Sm$ iterations satisfies
\begin{equation}
    \mathbb{E}\| \text{grad}f(\Tilde{x}) \|^2 \leq \frac{(L_l + \theta G) n^{a_1} (\zeta \nu^2 + 2c_R D)^{a_2} \Delta}{bT \psi},
\end{equation}
where $\Delta := f(\Tilde{x}^0) - f(x^*)$. By choosing $a_1 = 2/3, a_2 = 1/2$, the total IFO complexity to achieve $\epsilon$-accurate solution is $\mathcal{O}\Big( n + \frac{ (L_l + \theta G) n^{2/3} (\zeta \nu^2 + 2 c_R D)^{1/2}}{\epsilon^2} \Big)$.

\begin{proof}
First note that $c_t = c_{t+1} (1 + \eta \lambda + \frac{(\zeta \nu^2 + 2c_R D)(L_l + \theta G)^2\mu^2 \eta^2}{b}) + \frac{L (L_l + \theta G)^2 \mu^2 \eta^2}{2b} = c_{t+1}(1+ \phi) +  \frac{L (L_l + \theta G)^2 \mu^2 \eta^2}{2b}$, where $\phi := \eta \lambda + \frac{(\zeta \nu^2 + 2c_R D)(L_l + \theta G)^2\mu^2 \eta^2}{b}$. Choose $\eta = \frac{\mu_0 b}{(L_l + \theta G)\mu n^{a_1}(\zeta \nu^2 + 2 c_R D)^{a_2}}, \mu_0 \in (0,1)$ and $\lambda = \frac{(L_l + \theta G) \mu (\zeta \nu^2 + 2 c_R D)^{1-a_2}}{n^{a_1/2}}$ gives
\begin{align}
    c_t = (1+ \phi) c_{t+1} + \frac{L\mu_0^2 b}{2n^{2a_1}(\zeta \nu^2 + 2 c_R D)^{2a_2}}, \label{recursive_ct}
\end{align}
Applying \eqref{recursive_ct} recursively to $c_0$ and noting $c_m = 0$, we have
\begin{align}
    c_0 = \frac{L\mu_0^2 b}{2n^{2a_1}(\zeta \nu^2 + 2 c_R D)^{2a_2}} \frac{(1+\phi)^m - 1}{\phi}. \label{c0_formula}
\end{align}
It is noted that the sequence $\{c_t\}_{t=0}^{m-1}$ is a decreasing sequence and achieves its maximum at $c_0$. Therefore we derive a bound on $c_0$. Note that
\begin{align}
    \phi &= \frac{\mu_0 b (\zeta \nu^2 + 2 c_R D)^{1-2a_2}}{n^{3/2a_1}} + \frac{\mu_0^2 b (\zeta \nu^2 + 2 c_R D)^{1-2a_2}}{n^{2a_1}} \nonumber\\
    &\in \Big(  \frac{\mu_0 b (\zeta \nu^2 + 2 c_R D)^{1-2a_2}}{n^{3/2a_1}}, \frac{2\mu_0 b (\zeta \nu^2 + 2 c_R D)^{1-2a_2}}{n^{3/2a_1}} \Big). \label{phi_bound_1}
\end{align}
Choosing $m = \lfloor n^{3/2a_1}/ 2b\mu_0 (\zeta \nu^2 + 2 c_R D)^{1-2a_2} \rfloor$ suggests 
\begin{align}
    \phi \leq \frac{2\mu_0 b (\zeta \nu^2 + 2 c_R D)^{1-2a_2}}{n^{3/2a_1}} \leq \frac{1}{m}, \, \text{ and } (1+\phi)^m \leq e,  \label{phi_bound_2}
\end{align}
where $e$ is the Euler's constant. Note for the second inequality, we loosely use $\leq$ instead of $<$ for consistency. Then applying \eqref{phi_bound_1} and \eqref{phi_bound_2} into \eqref{c0_formula}, we have
\begin{align}
    c_0 \leq \frac{L\mu_0^2 b}{2n^{2a_1}(\zeta \nu^2 + 2 c_R D)^{2a_2}} \times \frac{n^{3/2a_1} (e-1)}{\mu_0 b (\zeta \nu^2 + 2 c_R D)^{1-2a_2}} = \frac{L\mu_0 (e-1)}{2n^{1/2a_1} (\zeta \nu^2 + 2 c_R D)}.
\end{align}
Next we consider a lower bound on $\Tilde{\delta}$.
\begin{align}
    \Tilde{\delta} &= \min_t \Big(  \eta - \frac{c_{t+1}\eta}{\lambda} - \frac{L \eta^2}{2} - c_{t+1}(\zeta \nu^2 + 2c_R D) \eta^2 \Big) \nonumber\\
    &\geq \Big(  \eta - \frac{c_{0}\eta}{\lambda} - \frac{L \eta^2}{2} - c_{0}(\zeta \nu^2 + 2c_R D) \eta^2 \Big) \nonumber\\
    &\geq \eta \Big( 1 - \frac{L \mu_0 (e-1)}{2(L_l + \theta G)(\zeta \nu^2 + 2c_R D)^{2-a_2}\mu} - \frac{L \mu_0 b}{2(L_l + \theta G)(\zeta \nu^2 + 2c_R D)^{a_2} \mu n^{a_1}} \nonumber\\ 
    &- \frac{L\mu_0^2 (e-1) b}{2 (L_l + \theta G)(\zeta \nu^2 + 2c_R D)^{a_2}\mu n^{3/2 a_1}} \Big) \nonumber\\
    &\geq \frac{b \psi}{(L_l + \theta G)n^{a_1}(\zeta \nu^2 + 2c_R D)^{a_2}},
\end{align}
where $\psi > 0$ is a constant such that the last inequality holds. That is, we choose $\psi$ satisfying
\begin{align}
     0 < \psi &\leq \frac{\mu_0}{\mu} \Big( 1 - \frac{L \mu_0 (e-1)}{2(L_l + \theta G)(\zeta \nu^2 + 2c_R D)^{2-a_2}\mu} - \frac{L \mu_0 b}{2(L_l + \theta G)(\zeta \nu^2 + 2c_R D)^{a_2} \mu n^{a_1}} \nonumber\\
     &- \frac{L\mu_0^2 (e-1) b}{2 (L_l + \theta G)(\zeta \nu^2 + 2c_R D)^{a_2}\mu n^{3/2 a_1}} \Big).
\end{align}
This condition holds by setting a sufficiently small $\mu_0 \in (0,1)$ and also $b \leq n^{a_1}$. The requirement on $b$ is to ensure the third term and fourth term do not increase with $n$. Therefore, combining this result with Lemma \ref{prelim_lemma_svrg_2} yields
\begin{equation}
    \mathbb{E}\| \text{grad}f(\Tilde{x}) \|^2 \leq \frac{(L_l + \theta G) n^{a_1} (\zeta \nu^2 + 2c_R D)^{a_2} \Delta}{bT \psi}.
\end{equation}
To achieve $\epsilon$-accurate solution, it is sufficient to require $\mathbb{E}\| \text{grad}f(\Tilde{x}) \|^2 \leq \epsilon^2$. That is, $\mathbb{E}\| \text{grad}f(\Tilde{x}) \| \leq \sqrt{\mathbb{E}\| \text{grad}f(\Tilde{x}) \|^2} \leq \epsilon$ by Jensen's inequality. Therefore, we require at least
    \begin{align}
        S = \frac{(L_l + \theta G)n^{a_1}(\zeta \nu^2 + 2c_R D)^{a_2}}{bm \psi \epsilon^2} &= \lceil \frac{2\mu_0 (L_l + \theta G)(\zeta \nu^2 + 2c_R D)^{1-a_2} n^{-a_1/2} }{\psi \epsilon^2} \rceil \nonumber\\
        &= \mathcal{O}\Big( 1 + \frac{(\zeta \nu^2 + 2c_R D)^{1-a_2} n^{-a_1/2}}{\epsilon^2} \Big)
    \end{align}
    number of epochs. Each epoch requires $n + 2mb$ IFO calls, which is $n + \lfloor n^{3/2a_1}/ \mu_0 (\zeta \nu^2 + 2 c_R D)^{1-2a_2} \rfloor = \mathcal{O}\big( n + n^{3/2a_1} (\zeta \nu^2 + 2 c_R D)^{2a_2-1} \big)$. Hence the total complexity is given by
    \begin{align}
        &\mathcal{O} \Big( \big( 1 + \frac{(\zeta \nu^2 + 2c_R D)^{1-a_2} n^{-a_1/2}}{\epsilon^2} \big)  \big(  n + n^{3/2a_1} (\zeta \nu^2 + 2 c_R D)^{2a_2-1} \big) \Big) \nonumber\\
        &= \mathcal{O} \Big( n + \frac{n^{a_1} (\zeta \nu^2 + 2 c_R D)^{a_2}}{\epsilon^2} + \frac{(\zeta \nu^2 + 2c_R D)^{1-a_2} n^{1-a_1/2}}{\epsilon^2} + n^{3/2a_1} (\zeta \nu^2 + 2 c_R D)^{2a_2-1} \big) \Big). 
    \end{align}
    With the standard choice of $\alpha_1 = \frac{2}{3}$ and $\alpha_2 = \frac{1}{2}$, we obtain the desired result.
\end{proof}

\newpage
\section{Convergence Analysis for R-AbaSVRG}
\label{converg_r_adasvrg_appendix}

\textbf{Lemma \ref{lemma_adasvrg} (Gradient estimation error bound for R-AbaSVRG).}
Suppose Assumptions \ref{common_assump} and \ref{assumption_svrg} hold and consider Algorithm \ref{RadaSVRG_algorithm}. We can bound estimation error of modified gradient $v_t^s$ to the full gradient $\text{grad}f(x_t^s)$ as 
\begin{equation}
    \mathbb{E}[\| v_t^{s} - \text{grad}f(x_{t}^s) \|^2 | \mathcal{F}_0^s] \leq \frac{t}{b} (L_l + \theta G)^2 \mu^2 \nu^2 \eta^2 \sum_{i=0}^{t-1} \mathbb{E}[\| v_i^s \|^2 | \mathcal{F}_0^s] + \mathbbm{1}_{\{ B^s < n\}} \frac{\sigma^2}{B^s}. 
\end{equation}

\begin{proof}
First note that $\mathcal{F}_0^s \subseteq \mathcal{F}_t^s$, for $0 \leq t \leq m-1$ and therefore it holds that $\mathbb{E}[\| v_t^{s} - \text{grad}f(x_{t}^s) \|^2 | \mathcal{F}_0^s] = \mathbb{E}[ \mathbb{E}[ \| v_t^{s} - \text{grad}f(x_{t}^s) \|^2 | \mathcal{F}_t^s] | \mathcal{F}_0^s]$. Hence we first consider bounding $\mathbb{E}[\| v_t^{s} - \text{grad}f(x_{t}^s) \|^2 | \mathcal{F}_t^s]$ as  
    \begin{align}
        &\mathbb{E}[\| v_t^{s} - \text{grad}f(x_{t}^s) \|^2 | \mathcal{F}_t^s] \nonumber\\
        &= \mathbb{E} [\| \text{grad}f_{\mathcal{I}_t^s} (x_{t}^s) - \mathcal{T}_{{x}_0^s}^{x_{t}^{s}} \big( \text{grad}f_{\mathcal{I}_t^s}(x_0^s) - v_0^{s} \big) - \text{grad}f(x_{t}^s) \|^2 | \mathcal{F}_t^s] \nonumber\\ 
        &= \mathbb{E} [\| \text{grad}f_{\mathcal{I}_t^s} (x_{t}^s) - \mathcal{T}_{{x}_0^s}^{x_{t}^{s}} \text{grad}f_{\mathcal{I}_t^s}(x_0^s) - \text{grad}f(x_{t}^s) + \mathcal{T}_{{x}_0^s}^{x_{t}^{s}} \text{grad}f(x_0^s) + \mathcal{T}_{{x}_0^s}^{x_{t}^{s}} v_0^s - \mathcal{T}_{{x}_0^s}^{x_{t}^{s}} \text{grad}f(x_0^s)\|^2 |\mathcal{F}_t^s] \nonumber\\
        &= \mathbb{E} [\| \text{grad}f_{\mathcal{I}_t^s} (x_{t}^s) - \mathcal{T}_{{x}_0^s}^{x_{t}^{s}} \text{grad}f_{\mathcal{I}_t^s}(x_0^s) - \text{grad}f(x_{t}^s) + \mathcal{T}_{{x}_0^s}^{x_{t}^{s}} \text{grad}f(x_0^s) \|^2 | \mathcal{F}_t^s] + \| v_0^s - \text{grad}f(x_0^s)\|^2 \nonumber\\
        &+ \mathbb{E}[\langle \text{grad}f_{\mathcal{I}_t^s} (x_{t}^s) - \mathcal{T}_{{x}_0^s}^{x_{t}^{s}} \text{grad}f_{\mathcal{I}_t^s}(x_0^s) - \text{grad}f(x_{t}^s) + \mathcal{T}_{{x}_0^s}^{x_{t}^{s}} \text{grad}f(x_0^s),  v_0^s - \text{grad}f(x_0^s) \rangle | \mathcal{F}_t^s] \nonumber\\
        &= \mathbb{E} [\| \text{grad}f_{\mathcal{I}_t^s} (x_{t}^s) - \mathcal{T}_{{x}_0^s}^{x_{t}^{s}} \text{grad}f_{\mathcal{I}_t^s}(x_0^s) - \text{grad}f(x_{t}^s) + \mathcal{T}_{{x}_0^s}^{x_{t}^{s}} \text{grad}f(x_0^s) \|^2 | \mathcal{F}_t^s] + \| v_0^s - \text{grad}f(x_0^s)\|^2 \nonumber\\
        &\leq \mathbb{E} [\| \text{grad}f_{\mathcal{I}_t^s}(x^s_{t}) - \mathcal{T}_{{x}_0^s}^{x_{t}^{s}} \text{grad}f_{\mathcal{I}_t^s}(x_0^s) \|^2 | \mathcal{F}_t^s] + \| v_0^s - \text{grad}f(x_0^s)\|^2 \nonumber\\
        &\leq \frac{1}{b} \mathbb{E} [\| \text{grad}f_i(x_{t}^s) - \mathcal{T}_{{x}_0^s}^{x_{t}^{s}} \text{grad}f_{i}(x_0^s)\|^2 | \mathcal{F}_t^s] + \| v_0^s - \text{grad}f(x_0^s)\|^2 \nonumber\\
        &\leq \frac{1}{b} (L_l + \theta G)^2 \| R^{-1}_{x_0^s}(x_{t}^s) \|^2 + \| v_0^s - \text{grad}f(x_0^s)\|^2 \nonumber\\
        &\leq \frac{1}{b} (L_l + \theta G)^2 \mu^2 d^2(x_{t}^s, x_0^s) + \| v_0^s - \text{grad}f(x_0^s)\|^2. \label{bound_on_diff_vt}
    \end{align}
    The fourth equality is based on the facts that $\text{grad}f_{\mathcal{I}_t^s}(x)$ is unbiased estimator of $\text{grad}f(x)$ and also the isometric property of vector transport $\mathcal{T}_{{x}_0^s}^{x_{t}^{s}}$. Note that $\mathcal{T}_{{x}_0^s}^{x_{t}^{s}}$ depends on both $x_0^s$ and $x_t^s$, which are measurable in $\mathcal{F}_t^s$. Therefore $\mathbb{E}[\mathcal{T}_{{x}_0^s}^{x_{t-1}^{s}} \text{grad}f_{\mathcal{I}_t^s}(x_0^s) |\mathcal{F}_t^s] = \mathcal{T}_{{x}_0^s}^{x_{t-1}^{s}} \mathbb{E}[\text{grad}f_{\mathcal{I}_t^s}(x_0^s)|\mathcal{F}_t^s] = \mathcal{T}_{{x}_0^s}^{x_{t-1}^{s}} \text{grad}f(x_0^s)$. The first inequality is due to $\mathbb{E}\|x - \mathbb{E}[x]\|^2 \leq \mathbb{E}\|x\|^2$ and the second inequality is due to independence of with replacement sampling. The last two inequalities are from Assumption \ref{relation_distance_assump} and Lemma \ref{vec_trans_lemma}. Taking expectation with respect to $\mathcal{F}_0^s$ gives $\mathbb{E}[\| v_t^{s} - \text{grad}f(x_{t}^s) \|^2 | \mathcal{F}_0^s] \leq \frac{1}{b} (L_l + \theta G)^2 \mu^2 \mathbb{E}[d^2(x_{t}^s, x_0^s)|\mathcal{F}_0^s] + \mathbb{E}[\| v_0^s - \text{grad}f(x_0^s)\|^2|\mathcal{F}_0^s]$. Next, we further simplify \eqref{bound_on_diff_vt} by telescoping iterates within epoch $s$. Note that by triangle inequality and Assumption \ref{relation_distance_assump},
    \begin{align}
        d^2(x_{t}^s, x_0^s) &\leq \big( d(x_{t}^s, x_{t-1}^s) + d(x_{t-1}^s, x_{t-2}^s) + \cdots + d(x_{1}^s, x_0^s) \big)^2 \nonumber\\
        &\leq \nu^2 \eta^2 \big( \| v_{t-1}^s \| + \cdots + \| v_0^s \| \big)^2 \leq \nu^2 \eta^2 t \sum_{i=0}^{t-1}  \| v_i^s\|^2, \label{distance_telescop_eq}
    \end{align}
    where the last inequality follows from $\| \sum_{i=1}^d w_i \|^2 \leq d \sum_{i=1}^d \|w_i \|^2$. On the other hand, by Lemma \ref{variance_bound_wors} and variance bound assumption \ref{bounded_var_assump}, we have
    \begin{align}
        \mathbb{E}[ \|v_0^s - \text{grad}f(x_0^s) \|^2 | \mathcal{F}_0^s] &= \mathbb{E}[\| \text{grad}f_{\mathcal{B}^s}(x_0^s) - \text{grad}f(x_0^s) \|^2 | \mathcal{F}_0^s] \nonumber\\ 
        &= \mathbb{E}[ \frac{1}{B^s}\sum_{i \in \mathcal{B}^s} \text{grad}f_{i}(x_0^s) - \text{grad}f(x_0^s) | \mathcal{F}_0^s] \nonumber\\
        &\leq \frac{n - B^s}{n-1}\frac{1}{n B^s} \sum_{i=1}^n\| \text{grad}f_i(x_0^s) - \text{grad}f(x_0^s) \| \nonumber\\
        &\leq \frac{n - B^s}{n-1}\frac{\sigma^2}{ B^s} \leq \mathbbm{1}_{\{ B^s < n \}} \frac{\sigma^2}{B^s}. \label{v_0_s_gradient_diff_bound}
    \end{align}
    Note if $\mathcal{B}^s$ is chosen from $[n]$ with replacement or under online setting where $n$ approaches infinity, we simply have $\mathbb{E}[\| v_0^s - \text{grad}f(x_0^s) \|^2 | \mathcal{F}_0^s] = \frac{1}{B^s} \mathbb{E}[\|\text{grad}f_i(x_0^s) - \text{grad}f(x_0^s) \|^2 | \mathcal{F}_0^s ] \leq \frac{\sigma^2}{B^s}$, which does not vanish when $B^s = n$. Substituting \eqref{distance_telescop_eq} and \eqref{v_0_s_gradient_diff_bound} back to \eqref{bound_on_diff_vt} gives the desired result.
\end{proof}

\textbf{Theorem \ref{convergence_radasvrg} (Convergence of R-AbaSVRG).}
Let $x^* \in \mathcal{M}$ be an optimal point of $f$ and suppose Assumptions \ref{common_assump} and \ref{assumption_svrg} hold. Consider Algorithm \ref{RadaSVRG_algorithm} with a fixed step size $\eta \leq \frac{2 - \frac{2}{\alpha}}{L + \sqrt{L^2 + 4(1 - \frac{1}{\alpha})\frac{ (L_l + \theta G)^2 \mu^2 \nu^2 m^2}{b} }}$ with $\alpha \geq 4$. Under both finite-sum and online settings, output $\Tilde{x}$ after running $T = Sm$ iterations satisfies
\begin{equation}
    \mathbb{E}\| \text{grad}f(\Tilde{x}) \|^2 \leq \frac{2\Delta}{T \eta} + \frac{\epsilon^2}{2},
\end{equation}
where $\Delta := f(\Tilde{x}^0) - f(x^*)$ and $\epsilon$ is the desired accuracy. 

\begin{proof}
    By retraction $L$-smoothness in Assumption \ref{retraction_smooth_assump},
    \begin{align}
        f(x_{t+1}^s) - f(x_{t}^s) &\leq - \eta \langle \text{grad}f(x_{t}^s), v_{t}^s\rangle + \frac{L \eta^2}{2}\| v_t^s \|^2 \nonumber\\
        &= - \frac{\eta}{2}\| \text{grad}f(x_{t}^s) \|^2 - \frac{\eta}{2}\|v_t^s \|^2 + \frac{\eta}{2}\| v_t^s - \text{grad}f(x^s_t) \|^2 + \frac{L \eta^2}{2}\|v_t^s \|^2 \nonumber \\
        &= - \frac{\eta}{2}\| \text{grad}f(x_{t}^s) \|^2 + \frac{\eta}{2}\| v_t^s - \text{grad}f(x^s_t) \|^2 - (\frac{\eta}{2} - \frac{L \eta^2}{2}) \| v_t^s \|^2. 
    \end{align}
    Rearranging the term and taking expectation with respect to $\mathcal{F}_0^s$ yields
    \begin{align}
        \mathbb{E}[\| \text{grad}f(x_t^s) \|^2 | \mathcal{F}_0^s] &\leq \frac{2}{\eta} \mathbb{E}[ f(x_t^s) - f(x_{t+1}^s) | \mathcal{F}_0^s] + \mathbb{E}[ | v_t^s - \text{grad}f(x^s_t) \|^2|\mathcal{F}_0^s ] - (1 - L \eta) \mathbb{E}[\| v_t^s \|^2 | \mathcal{F}_0^s] \nonumber\\
        &\leq \frac{2}{\eta} \mathbb{E}[ f(x_t^s) - f(x_{t+1}^s) | \mathcal{F}_0^s] + \frac{t}{b} (L_l + \theta G)^2 \mu^2 \nu^2 \eta^2 \sum_{i=0}^{t-1} \mathbb{E}[\| v_i^s \|^2 | \mathcal{F}_0^s] + \mathbbm{1}_{\{ B^s < n \}} \frac{\sigma^2}{B^s} \nonumber\\
        &- (1 - L \eta) \mathbb{E}[\| v_t^s \|^2 | \mathcal{F}_t^s]. \label{temp_ref_svrg}
    \end{align}
    Summing this result over $t = 0, ..., m-1$ gives
    \begin{align}
        \sum_{t=0}^{m-1} \mathbb{E}[\| \text{grad}f(x_t^s) \|^2 | \mathcal{F}_0^s] &\leq \frac{2}{\eta} \mathbb{E}[ f(x_0^s) - f(x_m^s) | \mathcal{F}_0^s] + \frac{(L_l + \theta G)^2\mu^2\nu^2 \eta^2}{b} \sum_{t=0}^{m-1} t\sum_{i=0}^{t} \mathbb{E}[\| v_i^s \|^2 | \mathcal{F}_0^s] + \mathbbm{1}_{\{ B^s < n \}} \frac{m\sigma^2}{B^s} \nonumber\\
        &- (1 - L \eta) \sum_{t=0}^{m-1} \mathbb{E}[\| v_t^s \|^2 | \mathcal{F}_0^s] \nonumber\\
        &\leq \frac{2}{\eta} \mathbb{E}[ f(x_0^s) - f(x_m^s) | \mathcal{F}_0^s] + \frac{(L_l + \theta G)^2\mu^2\nu^2 \eta^2 m^2}{b} \sum_{t=0}^{m-1} \mathbb{E}[\| v_t^s \|^2 | \mathcal{F}_0^s] + \mathbbm{1}_{\{ B^s < n \}} \frac{m\sigma^2}{B^s} \nonumber\\
        &- (1 - L \eta) \sum_{t=0}^{m-1} \mathbb{E}[\| v_t^s \|^2 | \mathcal{F}_0^s] \nonumber\\
        &= \frac{2}{\eta} \mathbb{E}[ f(x_0^s) - f(x_m^s) | \mathcal{F}_0^s] - (1 - L \eta - \frac{(L_l + \theta G)^2\mu^2\nu^2 \eta^2 m^2}{b}) \sum_{t=0}^{m-1} \mathbb{E}[\| v_t^s \|^2 | \mathcal{F}_0^s] + \mathbbm{1}_{\{ B^s < n \}} \frac{m\sigma^2}{B^s}. \label{temp_result_fdfkd}
    \end{align}
    The second inequality uses the fact $t \leq m-1$. 
     Telescoping \eqref{temp_result_fdfkd} from $s = 1, ..., S$ and taking expectation over all randomness gives
    \begin{align}
        \sum_{s=1}^S\sum_{t=0}^{m-1} \mathbb{E}\| \text{grad}f(x_t^s) \|^2 &\leq - (1 - L \eta - \frac{(L_l + \theta G)^2\mu^2\nu^2 \eta^2 m^2}{b}) \sum_{s=1}^S \sum_{t=0}^{m-1} \mathbb{E}\| v_t^s \|^2 \nonumber\\
        &+ \frac{2}{\eta} \mathbb{E}[f(\Tilde{x}^0) - f(x_m^S)] + \sum_{s=1}^S \mathbb{E}[\mathbbm{1}_{\{ B^s < n \}} \frac{m\sigma^2}{B^s}] \nonumber\\
        &\leq \frac{2\Delta}{\eta} - (1 - L \eta - \frac{(L_l + \theta G)^2\mu^2\nu^2 \eta^2 m^2}{b}) \sum_{s=1}^S \sum_{t=0}^{m-1} \mathbb{E}\| v_t^s \|^2 \nonumber\\
        &+ \sum_{s=1}^S \mathbb{E}[\mathbbm{1}_{\{ B^s < n \}} \frac{m\sigma^2}{B^s}], \label{poieruueir}
    \end{align}
    where $\Delta := f(\Tilde{x}^0) - f(x^*)$ and we use the fact that $\mathbb{E}[f(x_m^S)] \geq f(x^*)$. Since $\mathcal{B}^s$ depends on whether finite-sum or online setting is considered, we consider these two cases separately. (1) Under finite-sum setting, 
    \begin{align}
        \mathbbm{1}_{\{ B^s < n \}} \frac{1}{B^s} = \frac{1}{\min \{\alpha_1 \sigma^2/\beta_s , n\}} \leq \frac{\beta_s}{\alpha_1 \sigma^2} \leq \frac{\beta_s}{\alpha \sigma^2},
    \end{align}
    where we choose $\alpha_1 \geq \alpha$. Note also from the definition of $\beta_s$ and the choice of $\beta_1 \leq \epsilon^2 S$, we have
    \begin{align}
        \sum_{s=1}^S \mathbb{E}[\beta_s] = \beta_1 + \frac{1}{m} \sum_{s=1}^{S-1} \sum_{t=0}^{m-1} \mathbb{E}\| v_t^s \|^2 \leq \epsilon^2 S + \frac{1}{m} \sum_{s=1}^S \sum_{t=0}^{m-1} \mathbb{E}\| v_t^s \|^2, \label{beta_s_sum_bound}
    \end{align}
    Combining these two results and substituting into \eqref{poieruueir} gives
    \begin{align}
        \sum_{s=1}^S\sum_{t=0}^{m-1} \mathbb{E}\| \text{grad}f(x_t^s) \|^2 &\leq - (1 - L \eta - \frac{(L_l + \theta G)^2\mu^2\nu^2 \eta^2 m^2}{b}) \sum_{s=1}^S \sum_{t=0}^{m-1} \mathbb{E}\| v_t^s \|^2 + \frac{2 \Delta}{\eta} \nonumber\\
        &+ \frac{m}{\alpha} [\epsilon^2 S + \frac{1}{m}\sum_{s=1}^S \sum_{t=0}^{m-1} \mathbb{E}\| v_t^s \|^2 ] \nonumber\\
        &= - (1 - L \eta - \frac{(L_l + \theta G)^2\mu^2\nu^2 \eta^2 m^2}{b} - \frac{1}{\alpha}) \sum_{s=1}^S \sum_{t=0}^{m-1} \mathbb{E}\| v_t^s \|^2 \nonumber\\
        &+ \frac{2\Delta}{\eta} + \frac{\epsilon^2 m S}{\alpha}. \label{ertyety}
    \end{align}
    Let $\eta \leq \frac{2 - \frac{2}{\alpha}}{L + \sqrt{L^2 + 4(1 - \frac{1}{\alpha})\frac{ (L_l + \theta G)^2 \mu^2 \nu^2 m^2}{b} }}$, which is the larger root of $1 - L \eta - \frac{(L_l + \theta G)^2\mu^2\nu^2 \eta^2 m^2}{b} - \frac{1}{\alpha} = 0$. The other root is smaller than zero. Therefore, this choice of $\eta$ can ensure coefficients before $\mathbb{E}\|v_t^s \|^2$ is smaller than zero. Then dividing \eqref{ertyety} by $T = Sm$ yields, 
    \begin{align}
        \mathbb{E}\| \text{grad}f(\Tilde{x}) \|^2 = \frac{1}{T} \sum_{s=1}^S\sum_{t=0}^{m-1} \mathbb{E}\| \text{grad}f(x_t^s) \|^2 \leq \frac{2\Delta}{T\eta} + \frac{\epsilon^2}{\alpha},
    \end{align}
    where we note that output $\Tilde{x}$ is uniformly drawn at random from $\{\{x_t^s\}_{t=0}^{m-1}\}_{s=1}^{S}$. (2) Similarly, under online setting, 
    \begin{equation}
         \mathbbm{1}_{\{ B^s < n \}} \frac{1}{B^s} = \frac{1}{\min\{\alpha_1 \sigma^2/\beta_s , \alpha_2 \sigma^2/\epsilon^2 \}} = \max \{ \frac{\beta_s}{\alpha_1 \sigma^2}, \frac{\epsilon^2}{\alpha_2 \sigma^2} \} \leq \frac{\beta_s + \epsilon^2}{\alpha \sigma^2}, \label{online_bs_bound}
    \end{equation}
    where the last inequality uses the fact that $\max \{ a, b \} \leq a+b$ and $\alpha_1, \alpha_2 \geq \alpha$. Following the same procedure and choice of $\eta$, we have 
    \begin{align}
        \mathbb{E}\| \text{grad}f(\Tilde{x}) \|^2 = \frac{1}{T} \sum_{s=1}^S\sum_{t=0}^{m-1} \mathbb{E}\| \text{grad}f(x_t^s) \|^2 \leq \frac{2\Delta}{T\eta} + \frac{2\epsilon^2}{\alpha}. 
    \end{align}
    Hence, by choosing $\alpha \geq 2$ for finite-sum setting and $\alpha \geq 4$ for online setting, we have 
    \begin{equation}
        \mathbb{E}\| \text{grad}f(\Tilde{x}) \|^2 \leq \frac{2\Delta}{T \eta} + \frac{\epsilon^2}{2}. 
    \end{equation}
    For simplicity, we consider $\alpha \geq 4$ for both cases. 
\end{proof}

\textbf{Corollary \ref{complexity_corollary_adasvrg} (IFO complexity of R-AbaSVRG).} 
With same Assumptions in Theorem \ref{convergence_radasvrg}, choose $b = m^2, \alpha = 4$, $\eta = \frac{3}{2L + 2\sqrt{L^2 + 3(L_l + \theta G)^2 \mu^2\nu^2}}$. Set $m = \lfloor n^{1/3} \rfloor$ under finite-sum setting and $m = (\frac{\sigma}{\epsilon})^{2/3}$ under online setting. The IFO complexity of Algorithm \ref{RadaSVRG_algorithm} to achieve $\epsilon$-accurate solution is given by
    $$\begin{cases} \mathcal{O}\big( \Tilde{B} + \frac{\Theta_1 \Tilde{B}}{n^{1/3} \epsilon^2} +  \frac{\Theta_1 n^{2/3} }{\epsilon^2} \big), & \text{ (finite-sum) }\\
    \mathcal{O}\big( \frac{\Theta_1 \Tilde{B}}{\sigma^{2/3} \epsilon^{4/3}} + \frac{\Theta_1 \sigma^{4/3}}{\epsilon^{10/3}} \big), & \text{ (online) }\end{cases}$$
where $\Theta_1 := L + \sqrt{L^2 + \varrho_1 (L_l + \theta G)^2 \mu^2 \nu^2}$ with $\varrho_1 > 0$ being a constant that does not depend on any parameter. $\Tilde{B}$ is the average batch size defined as follows. $\Tilde{B} := \frac{1}{S}\sum_{s=1}^{S} \min\{ {\alpha_1 \sigma^2}/{\beta}_s , n\}$ under finite-sum setting and $\Tilde{B} := \frac{1}{S}\sum_{s=1}^{S} \min\{ {\alpha_1 \sigma^2}/{\beta}_s , {\alpha_2 \sigma^2}/{\epsilon^2}\}$ under online setting.
\begin{proof}
Consider the parameter setting $b = m^2$, $\alpha = 4$ and $\eta = \frac{3}{2L + 2\sqrt{L^2 + 3(L_l + \theta G)^2 \mu^2\nu^2}}$. To obtain $\epsilon$-accurate solution, we require at least 
\begin{equation}
    S = \frac{4\Delta}{\epsilon^2 m \eta} = \frac{8\Delta}{3\epsilon^2 m} \big( L + \sqrt{L^2 + 3(L_l + \theta G)^2 \mu^2 \nu^2} \big) = \mathcal{O}\big( \frac{\Theta_1 }{m\epsilon^2} \big)
\end{equation}
where $\Theta_1 := L + \sqrt{L^2 + \varrho_1 (L_l + \theta G)^2 \mu^2 \nu^2}$, where $\varrho_1 > 0$ is a constant that does not depend on any parameter. Define average batch size $\Tilde{B}$ as
\begin{align}
    \Tilde{B} := \frac{1}{S} \sum_{s=1}^S B^s = \begin{cases} \frac{1}{S}\sum_{s=1}^{S} \min\{ {\alpha_1 \sigma^2}/{\beta}_s , n\}, & \text{ (finite-sum) }\\
\frac{1}{S}\sum_{s=1}^{S} \min\{ {\alpha_1 \sigma^2}/{\beta}_s , {\alpha_2 \sigma^2}/{\epsilon^2}\}, & \text{ (online) }\end{cases} \label{batch_size_def}
\end{align}
Then one epoch requires $\Tilde{B} + 2mb = \mathcal{O}(\Tilde{B} + m^3)$ IFO calls. Choosing $m = \lfloor n^{1/3} \rfloor$ under finite-sum setting and $m = (\frac{\sigma}{\epsilon})^{2/3}$ under online setting, the total IFO complexity is given by 
\begin{align*}
    \mathcal{O}\big( S ( \Tilde{B} + m^3) \big) = \mathcal{O}\big( S\Tilde{B} + Sm^3 \big) = \mathcal{O}\big( \frac{\Theta_1  \Tilde{B}}{m \epsilon^2} + \frac{\Theta_1  m^2}{\epsilon^2} \big) = \begin{cases} \mathcal{O}\big( \Tilde{B} + \frac{\Theta_1 \Tilde{B}}{n^{1/3} \epsilon^2} +  \frac{\Theta_1 n^{2/3} }{\epsilon^2} \big), & \text{ (finite-sum) }\\
    \mathcal{O}\big( \frac{\Theta_1 \Tilde{B}}{\sigma^{2/3} \epsilon^{4/3}} + \frac{\Theta_1 \sigma^{4/3}}{\epsilon^{10/3}} \big), & \text{ (online) }\end{cases}
\end{align*}

\end{proof}

\textbf{Corollary \ref{conver_complexity_svrg_new} (Convergence and complexity of R-SVRG under new analysis)}
With the same assumptions as in Theorem \ref{convergence_radasvrg} and consider Algorithm \ref{RadaSVRG_algorithm} with fixed batch size $B^s = B$ for $s=1,..., S$. Then choose a fixed step size $\eta \leq \frac{2}{L + \sqrt{L^2 + 4 \frac{(L_l + \theta G)^2\mu^2 \nu^2 m^2}{b}}}$. Output $\Tilde{x}$ after running $T = Sm$ iterations satisfies
    $$\mathbb{E}\| \text{grad}f(\Tilde{x}) \|^2 \leq \frac{2\Delta}{T\eta} + \mathbbm{1}_{\{B < n\}} \frac{\sigma^2}{B}.$$
    If we further choose $b = m^2, \eta = \frac{2}{L + \sqrt{L^2 + 4 {(L_l + \theta G)^2\mu^2 \nu^2 }}}$ and the following parameters
    \begin{align*}
        B = n, \quad m = \lfloor n^{1/3} \rfloor \quad &\text{ (finite-sum) } \\
        B = \frac{2\sigma^2}{\epsilon^2}, \quad m = (\frac{\sigma}{\epsilon})^{2/3} \quad &\text{ (online) }
    \end{align*}
    IFO complexity to obtain $\epsilon$-accurate solution is 
    \begin{equation*}
        \begin{cases} \mathcal{O}\big( n + \frac{\Theta_1 n^{2/3} }{\epsilon^2} \big), & \text{ (finite-sum) }\\
    \mathcal{O}\big( \frac{\Theta_1 \sigma^{4/3}}{\epsilon^{10/3}} \big), & \text{ (online) }\end{cases}
    \end{equation*}
    
\begin{proof}
    The proof is nearly identical to the proof of Theorem \ref{convergence_radasvrg}. From \eqref{temp_result_fdfkd}, 
    \begin{align}
        &\sum_{t=0}^{m-1} \mathbb{E}[\| \text{grad}f(x_t^s) \|^2 | \mathcal{F}_0^s] \nonumber\\
        &\leq \frac{2}{\eta} \mathbb{E}[ f(x_0^s) - f(x_m^s) | \mathcal{F}_0^s] - (1 - L \eta - \frac{(L_l + \theta G)^2\mu^2\nu^2 \eta^2 m^2}{b}) \sum_{t=0}^{m-1} \mathbb{E}[\| v_t^s \|^2 | \mathcal{F}_0^s] + \mathbbm{1}_{\{ B < n \}} \frac{m\sigma^2}{B}. \label{dkfkdfjiir}
    \end{align}
    Choosing a fixed step size $\eta \leq \frac{2}{L + \sqrt{L^2 + 4 \frac{(L_l + \theta G)^2\mu^2 \nu^2 m^2}{b}}}$, which ensures $1 - L \eta - \frac{(L_l + \theta G)^2\mu^2\nu^2 \eta^2 m^2}{b} \geq 0$. Telescoping this \eqref{dkfkdfjiir} from $s = 1,...,S$ and dividing by $T = Sm$ gives 
    \begin{align}
        \frac{1}{T} \sum_{s=1}^S\sum_{t=0}^{m-1} \mathbb{E}\| \text{grad}f(x_t^s) \|^2 \leq \frac{2\Delta}{T\eta} + \mathbbm{1}_{\{ B < n \}} \frac{\sigma^2}{B}.
    \end{align}
    Note that output $\Tilde{x}$ satisfies $\mathbb{E}\| \text{grad}f(\Tilde{x}) \|^2 = \frac{1}{T} \sum_{s=1}^S\sum_{t=0}^{m-1} \mathbb{E}\| \text{grad}f(x_t^s) \|^2$. Under finite-sum setting where $B = n$, $\mathbbm{1}_{\{ B < n \}} \frac{m\sigma^2}{B} = 0$, we have $\mathbb{E}\| \text{grad}f(\Tilde{x}) \|^2 \leq \frac{2\Delta}{T\eta}$. Under online setting where $B = \frac{2\sigma^2}{\epsilon^2}$, $\mathbbm{1}_{\{ B < n \}} \frac{\sigma^2}{B} = \frac{\epsilon^2}{2}$, we have $\mathbb{E}\| \text{grad}f(\Tilde{x}) \|^2 \leq \frac{2\Delta}{T\eta} + \frac{\epsilon^2}{2}$. Given $b = m^2$ and $\eta = \frac{2}{L + \sqrt{L^2 + 4 {(L_l + \theta G)^2\mu^2 \nu^2}}}$, under both finite-sum and online settings, to obtain $\epsilon$-accurate solution, we require at least
    \begin{equation}
        S = \mathcal{O} \Big( \frac{\Delta}{m \eta \epsilon^2} \Big) = \mathcal{O} \Big( \frac{\Delta}{m\epsilon^2} \big( L + \sqrt{L^2 + 4(L_l + \theta G)^2 \mu^2\nu^2} \big) \Big) = \mathcal{O} \big( \frac{\Theta_1}{m\epsilon^2} \big).
    \end{equation}
    
    Hence, we obtain the same iteration complexity as adaptive batch size version. Note for one epoch, we require $B + 2mb = \mathcal{O}(B + m^3)$ IFO calls. With the same choice of $m = \lfloor n^{1/3} \rfloor$ under finite-sum setting and $m = (\frac{\sigma}{\epsilon})^{2/3}$ under online setting, total IFO complexity is given by
    \begin{equation}
        \mathcal{O}\big( S ( B + m^3) \big) = \begin{cases} \mathcal{O}\big( n + \frac{\Theta_1 n^{2/3} }{\epsilon^2} \big), & \text{ (finite-sum) }\\
    \mathcal{O}\big( \frac{\Theta_1 \sigma^{4/3}}{\epsilon^{10/3}} \big), & \text{ (online) }\end{cases}
    \end{equation}
\end{proof}

\newpage
\section{Convergence Analysis for R-AbaSRG}
\label{r_abasrg_proof_appendix}
\textbf{Lemma \ref{lemma1_RSRG} (Gradient estimation error bound for R-AbaSRG).}
With $\mathcal{F}_t^s$ denoting the same sigma algebras as in R-AbaSVRG. Suppose Assumption \ref{common_assump} hold and consider Algorithm \ref{RAdaSRG_algorithm}. Then we can similarly bound estimation error of the modified gradient $v_t^s$ to the full gradient $\text{grad}f(x_t^s)$ as 
\begin{equation}
    \mathbb{E}[\| v_t^s - \text{grad}f(x_t^s) \|^2 | \mathcal{F}_0^s] \leq \frac{(L_l + \theta G)^2\eta^2}{b}  \sum_{i=0}^{t} \mathbb{E}[\| v_i^s \|^2| \mathcal{F}_0^s] + \mathbbm{1}_{\{ B^s < n\}} \frac{\sigma^2}{B^s}. \nonumber
\end{equation}

\begin{proof}
Note similarly, we have $\mathbb{E}[\| v_t^s - \text{grad}f(x_t^s) \|^2 | \mathcal{F}_0^s] = \mathbb{E}[ \mathbb{E}[\| v_t^s - \text{grad}f(x_t^s) \|^2 | \mathcal{F}_t^s]  | \mathcal{F}_0^s]$ and we first derive a bound on $\mathbb{E}[\| v_t^s - \text{grad}f(x_t^s) \|^2 | \mathcal{F}_t^s]$.
    \begin{align}
        \mathbb{E}[\| v_t^s - \text{grad}f(x_t^s) \|^2 | \mathcal{F}_t^s] &= \mathbb{E}[\| \text{grad}f_{\mathcal{I}_t^s}(x_{t}^{s}) -  \mathcal{T}_{{x}_{t-1}^s}^{x_t^{s}} \text{grad}f_{\mathcal{I}_t^s}(x_{t-1}^s) + \mathcal{T}_{{x}_{t-1}^s}^{x_t^{s}} v_{t-1}^{s} - \text{grad}f(x_t^s) \|^2 | \mathcal{F}_t^s] \nonumber\\
        &= \mathbb{E}[\| \text{grad}f_{\mathcal{I}_t^s}(x_{t}^{s}) -  \mathcal{T}_{{x}_{t-1}^s}^{x_t^{s}} \text{grad}f_{\mathcal{I}_t^s}(x_{t-1}^s) - \text{grad}f(x_t^s) + \mathcal{T}_{{x}_{t-1}^s}^{x_t^{s}} \text{grad}f(x_{t-1}^s) \nonumber\\
        &+ \mathcal{T}_{{x}_{t-1}^s}^{x_t^{s}} v_{t-1}^s - \mathcal{T}_{{x}_{t-1}^s}^{x_t^{s}} \text{grad}f(x_{t-1}^s)  \|^2 | \mathcal{F}_t^s ]\nonumber\\
        &= \mathbb{E}[ \| \text{grad}f_{\mathcal{I}_t^s}(x_{t}^{s}) -  \mathcal{T}_{{x}_{t-1}^s}^{x_t^{s}} \text{grad}f_{\mathcal{I}_t^s}(x_{t-1}^s) - \text{grad}f(x_t^s) + \mathcal{T}_{{x}_{t-1}^s}^{x_t^{s}} \text{grad}f(x_{t-1}^s) \|^2 | \mathcal{F}_t^s] \nonumber\\
        &+ \mathbb{E}[ \| v_{t-1}^s - \text{grad}f(x^s_{t-1}) \|^2 | \mathcal{F}_t^s] \nonumber\\
        &\leq \mathbb{E}[ \| \text{grad}f_{\mathcal{I}_t^s}(x_{t}^{s}) -  \mathcal{T}_{{x}_{t-1}^s}^{x_t^{s}} \text{grad}f_{\mathcal{I}_t^s}(x_{t-1}^s) \|^2 |\mathcal{F}_t^s] + \mathbb{E}[ \| v_{t-1}^s - \text{grad}f(x^s_{t-1}) \|^2 | \mathcal{F}_t^s]\nonumber\\
        &= \frac{1}{b} \mathbb{E}[ \| \text{grad}f_{i}(x_{t}^{s}) -  \mathcal{T}_{{x}_{t-1}^s}^{x_t^{s}} \text{grad}f_{i}(x_{t-1}^s) \|^2 |\mathcal{F}_t^s] + \mathbb{E}[ \| v_{t-1}^s - \text{grad}f(x^s_{t-1}) \|^2 | \mathcal{F}_t^s] \nonumber\\
        &\leq \frac{1}{b} (L_l + \theta G)^2 \eta^2 \| v_{t-1}^s \|^2  + \| v_{t-1}^s - \text{grad}f(x^s_{t-1}) \|^2. \label{jierieriep}
    \end{align}
    Note the expectation is taken with respect to randomness of sample $\mathcal{I}_t^s$ where both $x_{t-1}^s$ and $x_t^s$ are measurable. The vector transport $\mathcal{T}_{x_{t-1}^s}^{x_t^s}$ is therefore fixed conditional on $\mathcal{F}_t^s$. Hence, the third equality holds due to unbiasedness. The first inequality is due to $\mathbb{E}\| x - \mathbb{E}[x] \|^2 \leq \mathbb{E}\|x\|^2$ and the last inequality is from Lemma \ref{vec_trans_lemma}. Therefore we have $\mathbb{E}[\| v_t^s - \text{grad}f(x_t^s) \|^2 | \mathcal{F}_0^s] \leq \frac{1}{b} (L_l + \theta G)^2 \eta^2 \mathbb{E}[\| v_{t-1}^s \|^2 | \mathcal{F}_0^s]  + \mathbb{E}[\| v_{t-1}^s - \text{grad}f(x^s_{t-1}) \|^2 | \mathcal{F}_0^s]$. Recursively applying this inequality gives
    \begin{align}
        \mathbb{E}[\| v_t^s - \text{grad}f(x_t^s) \|^2 | \mathcal{F}_0^s] 
        &\leq \frac{(L_l + \theta G)^2\eta^2}{b}  \sum_{i=0}^{t-1} \mathbb{E}[\| v_i^s \|^2| \mathcal{F}_0^s] + \mathbb{E}[ \| v_0^s - \text{grad}f(x_0^s) \|^2 | \mathcal{F}_0^s] \nonumber\\
        &\leq \frac{(L_l + \theta G)^2\eta^2}{b}  \sum_{i=0}^{t} \mathbb{E}[\| v_i^s \|^2| \mathcal{F}_0^s] + \mathbb{E}[ \| v_0^s - \text{grad}f(x_0^s) \|^2 | \mathcal{F}_0^s], \label{32_bound_norm}
    \end{align}
    Note that $\mathbb{E}[\| v_0^s - \text{grad}f(x_0^s) \|^2 | \mathcal{F}_0^s] \leq \mathbbm{1}_{\{ B < n\}} \frac{\sigma^2}{B}$ by similar argument in \eqref{poieruueir}. Combining this inequality with \eqref{32_bound_norm} completes the proof.
\end{proof}

\textbf{Theorem \ref{converg_RSRG} (Convergence analysis of R-AbaSRG).}
let $x^* \in \mathcal{M}$ be an optimal point of $f$ and suppose Assumption \ref{common_assump} holds. Consider Algorithm \ref{RAdaSRG_algorithm} with a step size $\eta \leq \frac{2 - \frac{2}{\alpha}}{L + \sqrt{L^2 + 4 (1 - \frac{1}{\alpha}) \frac{(L_l + \theta G)^2m}{b} }}$ and $\alpha \geq 4$. Then under both finite-sum and online setting, output $\Tilde{x}$ after running $T = Sm$ iterations satisfies
\begin{equation*}
    \mathbb{E}\| \text{grad}f(\Tilde{x}) \|^2 \leq \frac{2\Delta}{T\eta} + \frac{\epsilon^2}{2},
\end{equation*}
with $\Delta := f(\Tilde{x}^0) - f(x^*)$ and $\epsilon$ is the desired accuracy.

\begin{proof}
    Here we adopt a similar procedure as the proof of R-AbaSVRG. By retraction $L$-smoothness, we have
    \begin{align}
        f(x_{t+1}^s) - f(x_{t}^s) &\leq - \eta \langle \text{grad}f(x^s_{t}), v_{t}^s \rangle + \frac{L \eta^2}{2} \| v_{t}^s \|^2 \nonumber\\
        &= - \frac{\eta}{2} \| \text{grad}f(x_{t}^s)\|^2 - \frac{\eta}{2}\| v_{t}^s \|^2 + \frac{\eta}{2} \| v_{t}^s - \text{grad}f(x_{t}^s) \|^2 + \frac{L \eta^2}{2}\| v_{t}^s\|^2 \nonumber\\
        &= -\frac{\eta}{2}\| \text{grad}f(x_{t}^s) \|^2 + \frac{\eta}{2} \| v_{t}^s - \text{grad}f(x_{t}^s) \|^2 - (\frac{\eta}{2} - \frac{L\eta^2}{2}) \| v_{t}^s \|^2. 
    \end{align}
    Taking expectation of this inequality with respect to $\mathcal{F}_0^s$ and summing over $t= 0, ..., m-1$ gives
    \begin{align}
        \sum_{t=0}^{m-1} \mathbb{E} [\| \text{grad}f(x_{t}^s) \|^2 | \mathcal{F}_0^s ] &\leq \frac{2}{\eta} \mathbb{E}[  f(x_{0}^s) -f(x_m^s) | \mathcal{F}_0^s] + \sum_{t=0}^{m-1} \mathbb{E}[\| v_{t}^s - \text{grad}f(x_{t}^s) \|^2 | \mathcal{F}_0^s] - (1 - L \eta) \sum_{t=0}^{m-1} \mathbb{E}[\| v_{t}^s \|^2 | \mathcal{F}_0^s] \nonumber\\
        &\leq \frac{2}{\eta} \mathbb{E}[  f(x_{0}^s) -f(x_m^s) | \mathcal{F}_0^s] - (1 - L \eta) \sum_{t=0}^{m-1} \mathbb{E}[\| v_{t}^s \|^2 | \mathcal{F}_0^s] + \frac{(L_l + \theta G)^2\eta^2 }{b} \sum_{t=0}^{m-1} \sum_{i=0}^{t} \mathbb{E}[\| v_i^s \|^2| \mathcal{F}_0^s] \nonumber\\
        &+ \mathbbm{1}_{\{ B^s < n\}} \frac{m \sigma^2}{B^s} \nonumber\\
        &\leq  \frac{2}{\eta} \mathbb{E}[  f(x_{0}^s) -f(x_m^s) | \mathcal{F}_0^s] - (1 - L \eta) \sum_{t=0}^{m-1} \mathbb{E}[\| v_{t}^s \|^2 | \mathcal{F}_0^s] + \frac{(L_l + \theta G)^2 \eta^2 m}{b} \sum_{t=0}^{m-1} \mathbb{E}[\| v_t^s \|^2| \mathcal{F}_0^s] \nonumber\\
        &+ \mathbbm{1}_{\{ B^s < n\}} \frac{m \sigma^2}{B^s} \nonumber\\
        &= \frac{2}{\eta} \mathbb{E}[  f(x_{0}^s) -f(x_m^s) | \mathcal{F}_0^s] - \Big( 1 - L \eta - \frac{(L_l + \theta G)^2 \eta^2 m}{b} \Big) \sum_{t=0}^{m-1} \mathbb{E}[\| v_t^s \|^2| \mathcal{F}_0^s] + \mathbbm{1}_{\{ B^s < n\}} \frac{m \sigma^2}{B^s}, \label{hhhhhhhhh}
    \end{align}
    where the second first inequality is by Lemma \ref{lemma1_RSRG} and the third inequality is due to the fact that $t \leq m-1$. Summing this inequality over $s = 1,..., S$ and taking full expectation, we have
    \begin{align}
        &\sum_{s=1}^S \sum_{t=0}^{m-1} \mathbb{E}\| \text{grad}f(x_{t}^s) \|^2 \nonumber\\
        &\leq \frac{2 \Delta}{\eta} - \Big( 1 - L \eta - \frac{(L_l + \theta G)^2 \eta^2 m}{b} \Big) \sum_{s=1}^S\sum_{t=0}^{m-1} \mathbb{E}\| v_t^s \|^2 + \sum_{s=1}^S \mathbb{E}[\mathbbm{1}_{\{ B^s < n\}} \frac{ m\sigma^2}{B^s}], \label{efefefefe}
    \end{align}
    where $\Delta := f(\Tilde{x}^0) - f(x^*)$. Same as in \eqref{beta_s_sum_bound}, we have $\sum_{s=1}^S \mathbb{E}[\beta_s] \leq \epsilon^2 S + \frac{1}{m} \sum_{s=1}^{S}\sum_{t=0}^{m-1} \mathbb{E}\|  v_t^s\|^2$, with the choice $\beta_1 \leq \epsilon^2 S$. (1) Under finite-sum setting, $\mathbbm{1}_{\{ B^s < n\}} \frac{1}{B^s} \leq \frac{\beta_s}{c_\beta \sigma^2}\leq \frac{\beta_s}{\alpha \sigma^2}$ where we choose $\alpha_1 > \alpha$. This gives 
    \begin{align}
        \sum_{s=1}^S \sum_{t=0}^{m-1} \mathbb{E}\| \text{grad}f(x_{t}^s) \|^2 &\leq \frac{2\Delta}{\eta} - \Big( 1 - L \eta - \frac{(L_l + \theta G)^2 \eta^2 m}{b} \Big) \sum_{s=1}^S\sum_{t=0}^{m-1} \mathbb{E}\| v_t^s \|^2 + \frac{m}{\alpha} \sum_{s=1}^S \mathbb{E}[\beta_s] \nonumber\\
        &\leq \frac{2\Delta}{\eta}  - \Big( 1 - \frac{1}{\alpha} - L \eta -  \frac{(L_l + \theta G)^2 \eta^2 m}{b} \Big) \sum_{s=1}^S\sum_{t=0}^{m-1} \mathbb{E}\| v_t^s \|^2 + \frac{Sm\epsilon^2}{\alpha}. 
    \end{align}
    Let $\eta \leq \frac{2 - \frac{2}{\alpha}}{L + \sqrt{L^2 + 4 (1 - \frac{1}{\alpha}) \frac{(L_l + \theta G)^2m}{b} }}$, which is the larger root of $1 - \frac{1}{\alpha} - L \eta -  \frac{(L_l + \theta G)^2 \eta^2 m}{b} = 0$. Dividing both sides by $T = Sm$ gives
    \begin{equation}
        \mathbb{E}\| \text{grad}f(\Tilde{x}) \|^2 = \frac{1}{T} \sum_{s=1}^S \sum_{t=0}^{m-1} \mathbb{E}\| \text{grad}f(x_{t}^s) \|^2 \leq \frac{2\Delta}{T\eta} + \frac{\epsilon^2}{\alpha}, 
    \end{equation}
    where $\Tilde{x}$ is uniformly selected at random from $\{\{x_t^s\}_{t=0}^{m-1}\}_{s=1}^{S}$. (2) Under online setting, from \eqref{online_bs_bound}, we have $\mathbbm{1}_{\{ B^s < n \}} \frac{1}{B^s} \leq \frac{\beta_s + \epsilon^2}{\alpha \sigma^2}$, where we choose $\alpha_1, \alpha_2 \geq \alpha$. This results in
    \begin{align}
        \sum_{s=1}^S \sum_{t=0}^{m-1} \mathbb{E}\| \text{grad}f(x_{t}^s) \|^2 &\leq \frac{2\Delta}{\eta} - \Big( 1 - \frac{1}{\alpha} - L\eta - \frac{(L_l + \theta G)^2 \eta^2 m}{b} \Big) \sum_{s=1}^S\sum_{t=0}^{m-1} \mathbb{E}\| v_t^s \|^2 + \frac{2Sm\epsilon^2}{\alpha}. 
    \end{align}
    Choose the same $\eta \leq \frac{2 - \frac{2}{\alpha}}{L + \sqrt{L^2 + 4 (1 - \frac{1}{\alpha}) \frac{(L_l + \theta G)^2m}{b} }}$, we have 
    \begin{equation}
        \mathbb{E}\| \text{grad}f(\Tilde{x}) \|^2 = \frac{1}{T} \sum_{s=1}^S \sum_{t=0}^{m-1} \mathbb{E}\| \text{grad}f(x_{t}^s) \|^2 \leq \frac{2 \Delta}{T\eta} + \frac{2\epsilon^2}{\alpha}.  
    \end{equation}

    By simply setting $\alpha = 4$ for both finite-sum and online setting, we can ensure 
    \begin{equation}
        \mathbb{E}\| \text{grad}f(\Tilde{x}) \|^2 \leq \frac{2\Delta}{T\eta} + \frac{\epsilon^2}{2}
    \end{equation}
\end{proof}

\textbf{Corollary \ref{complexity_rabasrg_corollary} (IFO complexity of R-AbaSRG).}
With the same Assumptions as in Theorem \ref{converg_RSRG}, choose $b = m, \alpha = 4$, $\eta = \frac{3}{2L + 2\sqrt{L^2 + 3 {(L_l + \theta G)^2}}}$. Then set $m = \lfloor n^{1/2} \rfloor$ under finite-sum setting and $m = \frac{\sigma}{\epsilon}$ under online setting. The IFO complexity of Algorithm \ref{RadaSVRG_algorithm} to obtain $\epsilon$-accurate solution is 
\begin{equation*}
    \begin{cases} \mathcal{O}\big( \Tilde{B} + \frac{\Theta_2 \Tilde{B} }{\sqrt{n} \epsilon^2} + \frac{\Theta_2 \sqrt{n}}{\epsilon^2}  \big), & \text{ (finite-sum) }\\
    \mathcal{O}\big( \frac{\Theta_2 \Tilde{B}}{\sigma \epsilon} + \frac{\Theta_2 \sigma}{\epsilon^3} \big), & \text{ (online) }\end{cases}
\end{equation*}

\begin{proof}
By choosing $\alpha = 4, b = m$, $\eta = \frac{3}{2L + 2\sqrt{L^2 + 3 {(L_l + \theta G)^2}}}$, to ensure $\mathbb{E}\| \text{grad}f(\Tilde{x}) \| \leq \epsilon$, we require at least
\begin{equation}
    S = \frac{4\Delta}{m \eta \epsilon^2} = \frac{8 \Delta }{3 m \epsilon^2} (L + \sqrt{L^2 + 3 {(L_l + \theta G)^2} }) = \mathcal{O}\big( \frac{\Theta_2}{m \epsilon^2} \big),
\end{equation}
with $\Theta_2 := L + \sqrt{L^2 + \varrho_2 (L_l + \theta G)^2}$ where $\varrho_2 > 0$ is a constant that does not depend on any parameters. Let $\Tilde{B}$ be the average batch size defined in \eqref{batch_size_def}. That is, $\Tilde{B} = \frac{1}{S}\sum_{s=1}^{S} \min\{ {\alpha_1 \sigma^2}/{\beta}_s , n\}$ under finite-sum setting and $\Tilde{B} = \frac{1}{S}\sum_{s=1}^{S} \min\{ {\alpha_1 \sigma^2}/{\beta}_s , {\alpha_2 \sigma^2}/{\epsilon^2}\}$ under online setting. Then one epoch requires $\Tilde{B} + 2mb = \mathcal{O}( \Tilde{B} + m^2)$ IFO calls. Consider the choice of $m = \lfloor {n}^{1/2} \rfloor$ and $m = \frac{\sigma}{\epsilon}$ under finite-sum and online setting respectively. The total IFO complexity is given by 
\begin{align}
    \mathcal{O} \big( S\Tilde{B} + Sm^2 \big) = \mathcal{O} \big( \frac{\Theta_2 \Tilde{B}}{m \epsilon^2} + \frac{\Theta_2 m}{\epsilon^2} \big) = \begin{cases} \mathcal{O}\big( \Tilde{B} + \frac{\Theta_2 \Tilde{B} }{\sqrt{n} \epsilon^2} + \frac{\Theta_2 \sqrt{n}}{\epsilon^2}  \big), & \text{ (finite-sum) }\\
    \mathcal{O}\big( \frac{\Theta_2 \Tilde{B}}{\sigma \epsilon} + \frac{\Theta_2 \sigma}{\epsilon^3} \big), & \text{ (online) }\end{cases}
\end{align}
\end{proof}

\textbf{Corollary \ref{converg_complex_rsrg_double_loop} (Double loop convergence and complexity of R-SRG)}
With the same assumptions in Theorem \ref{converg_RSRG} and consider Algorithm \ref{RAdaSRG_algorithm} with fixed batch size $B^s = B$, for $s = 1,...,S$. Consider a step size $\eta \leq \frac{2}{L + \sqrt{L^2 + 4\frac{(L_l + \theta G)^2m}{b}}}$. After running $T = Sm$ iterations, output $\Tilde{x}$ satisfies
\begin{equation*}
    \mathbb{E}\| \text{grad}f(\Tilde{x}) \|^2 \leq \frac{2\Delta}{T\eta} +  \mathbbm{1}_{\{ B < n\}} \frac{\sigma^2}{B}. 
\end{equation*}
If we further choose $b = m$, $\eta = \frac{2}{L + \sqrt{L^2 + 4{(L_l + \theta G)^2}}}$ and following parameters 
\begin{align*}
    B = n, \quad m = \lfloor n^{1/2} \rfloor \quad &\text{ (finite-sum) } \\
    B = \frac{2\sigma^2}{\epsilon^2}, \quad m = \frac{\sigma}{\epsilon} \quad &\text{ (online) }
\end{align*}
IFO complexity to obtain $\epsilon$-accurate solution is 
\begin{equation*}
    \begin{cases} \mathcal{O}\big( n + \frac{\Theta_2 \sqrt{n}}{\epsilon^2}  \big), & \text{ (finite-sum) }\\
    \mathcal{O}\big( \frac{\Theta_2 \sigma}{\epsilon^3} \big), & \text{ (online) }\end{cases}
\end{equation*}

\begin{proof}
    From \eqref{hhhhhhhhh}, we have 
    \begin{align}
        &\sum_{t=0}^{m-1} \mathbb{E} [\| \text{grad}f(x_{t}^s) \|^2 | \mathcal{F}_0^s ] \nonumber\\
        &\leq \frac{2}{\eta} \mathbb{E}[  f(x_{0}^s) -f(x_m^s) | \mathcal{F}_0^s] - \Big( 1 - L \eta - \frac{(L_l + \theta G)^2 \eta^2 m}{b} \Big) \sum_{t=0}^{m-1} \mathbb{E}[\| v_t^s \|^2| \mathcal{F}_0^s] + \mathbbm{1}_{\{ B < n\}} \frac{m \sigma^2}{B}. 
    \end{align}
    Consider step size choice $\eta \leq \frac{2}{L + \sqrt{L^2 + 4\frac{(L_l + \theta G)^2m}{b}}}$, which ensures $1 - L \eta - \frac{(L_l + \theta G)^2 \eta^2 m}{b} \geq 0$. Summing this result over $s = 1,..., S$ and dividing by $T = Sm$ yields
    \begin{equation}
        \mathbb{E}\| \text{grad}f(\Tilde{x}) \|^2 = \frac{1}{T} \sum_{s=1}^S \sum_{t=0}^{m-1} \mathbb{E}\| \text{grad}f(x_{t}^s) \|^2 \leq \frac{2\Delta}{T\eta} +  \mathbbm{1}_{\{ B < n\}} \frac{\sigma^2}{B}. 
    \end{equation}
    Considering the choice of $b = m$ and $\eta = \frac{2}{L + \sqrt{L^2 + 4{(L_l + \theta G)^2}}}$ and following exactly the same procedures as in proof of Corollary \ref{conver_complexity_svrg_new}, we require at least
    \begin{equation}
        S = \mathcal{O}\Big( \frac{\Delta}{m \eta \epsilon^2} \Big) = \mathcal{O} \Big( \frac{\Delta}{m\epsilon^2} \big( L + \sqrt{L^2 + 4(L_l + \theta G)^2} \big) \Big) = \mathcal{O}\Big( \frac{\Theta_2}{m\epsilon^2} \Big).
    \end{equation}
    One epoch requires $B + 2mb = \mathcal{O}(B + m^2)$ IFO complexity. With the same choice of $m = \lfloor n^{1/2} \rfloor$ under finite-sum setting and $m = \frac{\sigma}{\epsilon}$ under online setting, total IFO complexity is given by 
    \begin{equation}
        \mathcal{O}\big( S (B + m^2) \big) = \begin{cases} \mathcal{O}\big( n + \frac{\Theta_2 \sqrt{n}}{\epsilon^2}  \big), & \text{ (finite-sum) }\\
        \mathcal{O}\big( \frac{\Theta_2 \sigma}{\epsilon^3} \big), & \text{ (online) }\end{cases}
    \end{equation}
\end{proof}

\newpage
\section{Convergence under gradient dominance condition}
\label{gd_convergence_appendix}

\textbf{Theorem \ref{complexity_svrg_gd} (IFO complexity of R-AbaSVRG and R-SVRG).}
    Suppose Assumptions \ref{common_assump} and \ref{assumption_svrg} hold and also suppose function $f$ satisfies $\tau$-gradient dominance. Consider Algorithm \ref{RGDVR_algorithm} with any solver and accordingly choose appropriate parameters to achieve $\epsilon_k$-accurate solution. Then at each mega epoch $k$, output $x_k$ satisfies
    \begin{equation}
        \mathbb{E}\| \text{grad}f(x_k) \| \leq \frac{\epsilon_0}{2^k}, \text{ and } \mathbb{E}[f(x_k) - f(x^*)] \leq \frac{\tau \epsilon_0^2}{4^k}. 
    \end{equation}
    Consider R-AbaSVRG solver with the following parameters at each mega epoch. $\eta = \frac{3}{2L + 2\sqrt{L^2 + 3(L_l + \theta G)^2 \mu^2 \nu^2}}$, $\alpha = 4, b_k = m_k^2$, where $m_k = \lfloor n^{1/3} \rfloor$ under finite-sum setting and $m_k = (\frac{\sigma}{\epsilon_k})^{2/3}$ under online setting. Then to achieve $\epsilon$-accurate solution, total IFO complexity is given by
    \begin{align}
        \begin{cases} \mathcal{O}\big( \sum_{k=1}^K \Tilde{B}_k (1 + \frac{\Theta_1 \tau}{n^{1/3}}) + ({\Theta_1 n^{2/3} \tau}) \log(\frac{1}{\epsilon}) \big), & \text{ (finite-sum) }\\
        \mathcal{O}\big( \frac{\Theta_1 \tau \sum_{k=1}^K \Tilde{B}_k \epsilon_k^{2/3}}{\sigma^{2/3}} + \frac{\Theta_1 \tau \sigma^{4/3}}{\epsilon^{4/3}} \big), & \text{ (online) } \end{cases}
    \end{align}
    where the average batch size $\Tilde{B}_k := \frac{1}{S_k} \sum_{s=1}^{S_k} \min\{\alpha_1 \sigma^2/\beta_s, n \}$ under finite-sum setting and $\Tilde{B}_k := \frac{1}{S_k} \sum_{s=1}^{S_k} \min\{\alpha_1\sigma^2/\beta_s, \alpha_2 \sigma^2/\epsilon_k^2 \}$ under online cases. Consider R-SVRG solver with the same parameters except for $\eta = \frac{2}{L + \sqrt{L^2 + 4(L_l + \theta G)^2 \mu^2 \nu^2)}}$, $B_k = n$ under finite-sum setting and $B_k = \frac{2\sigma^2}{\epsilon_k^2}$ under online setting. To achieve $\epsilon$-accurate solution, we require a total complexity of 
    \begin{align}
        \begin{cases} \mathcal{O}\big( (n + \Theta_1 \tau n^{2/3})\log(\frac{1}{\epsilon}) \big), & \text{ (finite-sum) }\\
        \mathcal{O} \big( \frac{\Theta_1 \tau \sigma^{4/3}}{\epsilon^{4/3}} \big), & \text{ (online) } \end{cases}
    \end{align}

\begin{proof}
    Firstly, we establish linear convergence accompanying complexity results. At mega epoch $k$, we have 
    \begin{align}
        &\mathbb{E}\| \text{grad}f(x_k) \| \leq \epsilon_k = \frac{\epsilon_0}{2^k}, \text{ and } \\
        &\mathbb{E}[ f(x_k) - f(x^*) ] \leq \tau \mathbb{E}\| \text{grad}f(x_k) \|^2 \leq \frac{\tau \epsilon_0^2}{4^k}.
    \end{align}

    Define $\Delta_k := \mathbb{E}[f(x_k) - f(x^*)]$. At mega epoch $k$, to obtain $\epsilon_k$-accurate solution, we require number of epochs \begin{equation}
        S_k = \frac{4\Delta_{k-1}}{m_k \eta \epsilon_k^2} = \frac{8\Theta_1 \Delta_{k-1}}{3m_k\epsilon_k^2} \leq  \frac{8\Theta_1}{3m_k\epsilon_k^2} \tau \mathbb{E}\| \text{grad}f(x_{k-1})\|^2 \leq \frac{8\Theta_1 \tau}{3 m_k} \frac{\epsilon_{k-1}^2}{\epsilon_k^2} = \frac{32\Theta_1 \tau}{3m_k},
    \end{equation}
    where the first inequality uses the definition of gradient dominance and the second inequality is by the fact that $x_{t-1}$ is output from the preceding mega epoch and hence has gradient bounded by desired accuracy $\epsilon_{k-1}$. The last equality is from the choice of $\epsilon_k$. Define the average batch size $\Tilde{B}_k = \frac{1}{S_k} \sum_{s=1}^{S_k} \min\{\alpha_1 \sigma^2/\beta_s, n \}$ and $\Tilde{B}_k = \frac{1}{S_k} \sum_{s=1}^{S_k} \min\{\alpha_1\sigma^2/\beta_s, \alpha_2 \sigma^2/\epsilon_k^2 \}$ under finite-sum and online cases respectively. IFO complexity at mega epoch $k$ is
    \begin{align}
        S_k(\Tilde{B}_k + 2m_k b_k) = \mathcal{O}\big( S_k \Tilde{B} + S_k m_k^3 \big) = \begin{cases} \mathcal{O} \big( \Tilde{B}_k + \frac{\Theta_1 \Tilde{B}_k \tau}{n^{1/3}} + {\Theta_1 n^{2/3} \tau} \big), & \text{ (finite-sum) }\\
        \mathcal{O}\big( \frac{\Theta_1 \Tilde{B}_k \tau \epsilon_k^{2/3}}{\sigma^{2/3}} + \frac{\Theta_1 \sigma^{4/3} \tau}{\epsilon_k^{4/3}} \big), & \text{ (online) } \end{cases} \label{fkdfjskdririt}
    \end{align}
    To ensure $\mathbb{E}\| \text{grad}f(x_K) \|^2 \leq \epsilon^2$, it is equivalent to requiring $\epsilon_K^2 = \frac{\epsilon_0^2}{2^{2K}} \leq \epsilon^2$. Therefore, we require at least $K = \log(\frac{\epsilon_0}{\epsilon})$ mega epochs. Accordingly, under finite-sum setting, the complexity in \eqref{fkdfjskdririt} depends on mega epoch $k$ only through $\Tilde{B}_k$. So total IFO complexity after running $K$ mega epochs is simply $\mathcal{O}\big( \sum_{k=1}^K \Tilde{B}_k (1 + \frac{\Theta_1 \tau}{n^{1/3}}) + ({\Theta_1 n^{2/3} \tau}) \log(\frac{1}{\epsilon}) \big)$. Under online setting, both $\Tilde{B}_k$ and $\epsilon_k$ of its complexity depend on mega epoch $k$. Hence we need to sum this result from $k= 1,..., K = \log(\frac{\epsilon_0}{\epsilon})$. Note that $\sum_{k =1}^K \frac{1}{\epsilon_k^{4/3}} = \frac{2^{4/3}}{\epsilon_0^{4/3}} \frac{(2^k)^{4/3}- 1}{2^{4/3} -1 } \leq 2^{4/3} (\frac{2^k}{\epsilon_0})^{4/3} = \mathcal{O} \big((\frac{1}{\epsilon})^{4/3} \big)$. Hence, total IFO complexity can be written as $\mathcal{O}\big( \frac{\Theta_1 \tau \sum_{k=1}^K \Tilde{B}_k \epsilon_k^{2/3}}{\sigma^{2/3}} + \frac{\Theta_1 \tau \sigma^{4/3}}{\epsilon^{4/3}} \big)$. Similarly, R-SVRG with fixed batch size $B_k = n$ under finite-sum cases and $B_k = \frac{2\sigma^2}{\epsilon_k^2}$ under online cases requires complexities $\mathcal{O}\big( (n + \Theta_1 \tau n^{2/3})\log(\frac{1}{\epsilon}) \big)$ and $\mathcal{O} \big( \frac{\Theta_1 \tau \sigma^{4/3}}{\epsilon^{4/3}} \big)$ respectively. The proof is exactly the same except that we replace $\Tilde{B}_k$ with $B_k$. 
\end{proof}

\textbf{Theorem \ref{complexity_srg_gd} (IFO complexity of R-AbaSRG and R-SRG).}
    Suppose Assumptions \ref{common_assump} holds and also suppose function $f$ satisfies $\tau$-gradient dominance. By choosing parameters to achieve $\epsilon_k$-accurate solution, output $x_k$ satisfies the same linear convergence as in Theorem \ref{complexity_svrg_gd}. Consider R-AbaSRG solver with $\eta = \frac{3}{2L + 2\sqrt{L^2 + 3(L_l + \theta G)^2}}$, $\alpha = 4, b_k = m_k$ where $m_k = \lfloor n^{1/2} \rfloor$ under finite-sum setting and $m_k = \frac{\sigma}{\epsilon_k}$ under online setting. To achieve $\epsilon$-accurate solution, we require a total IFO complexity of 
    \begin{align}
        \begin{cases} \mathcal{O}\big( \sum_{k=1}^K \Tilde{B}_k (1 + \frac{\Theta_2 \tau}{n^{1/2}}) + ({\Theta_2 n^{1/2} \tau}) \log(\frac{1}{\epsilon}) \big), & \text{ (finite-sum) }\\
        \mathcal{O}\big( \frac{\Theta_2 \tau \sum_{k=1}^K \Tilde{B}_k \epsilon_k}{\sigma} + \frac{\Theta_2 \tau \sigma}{\epsilon} \big), & \text{ (online) } \end{cases}
    \end{align}
    Consider R-SRG solver with the same parameters except for $\eta = \frac{2}{L + \sqrt{L^2 + 4(L_l + \theta G)^2}}$ and $B_k = n$ under finite-sum setting and $B_k = \frac{2\sigma^2}{\epsilon_k^2}$ under online setting. To achieve $\epsilon$-accurate solution, we require a total complexity of 
    \begin{align}
        \begin{cases} \mathcal{O}\big( (n + \Theta_2 \tau n^{1/2})\log(\frac{1}{\epsilon}) \big), & \text{ (finite-sum) }\\
        \mathcal{O} \big( \frac{\Theta_2 \tau \sigma}{\epsilon} \big), & \text{ (online) } \end{cases}
    \end{align}
\begin{proof}
    The proof is exactly the same as that for R-AbaSVRG and R-SVRG and hence skipped.
\end{proof}

Next, we provide complexity results for R-SD and R-SGD under gradient dominance condition. We simply restart the algorithms similar to variance reduction methods.

\begin{algorithm}
 \caption{R-GD-SD/SGD}
 \label{RGDGD_algorithm}
 \begin{algorithmic}[1]
  \STATE \textbf{Input:} Initial accuracy $\epsilon_0$ and desired accuracy $\epsilon$, initialization $x_0$.
  \FOR {$k = 1,...,K$}
  \STATE $\epsilon_k = \frac{\epsilon_{k-1}}{2}$.
  \STATE Set $T_k$ sufficient to achieve $\epsilon_k$-accurate solution and choose step size $\eta$ accordingly. 
  \STATE $x_0^k = x^{k-1}$. 
  \FOR{$t = 1,...,T_k$}
    \STATE (R-SD): $x_t^k = R_{x_{t-1}^k}({-\eta \text{grad}f(x_{t-1}^k)})$.
    \STATE (R-SGD): $x_t^k = R_{x_{t-1}^k}({- \eta \text{grad}f_{i_t^k}(x_{t-1}^k)}),$ where $i_t^k \in [n]$ is a random index. 
  \ENDFOR
  \STATE $x^k$ is chosen uniformly at random from $\{ x_t^k \}_{t=0}^{T_k - 1}$.
  \ENDFOR
  \STATE \textbf{Output:} $x^K$.
 \end{algorithmic} 
\end{algorithm}

\begin{theorem}[IFO complexity of R-SD and R-SGD under gradient dominance condition]
\label{complexity_RGD_RSGD_GD}
    Suppose $f$ is retraction $L$-smooth and also $\tau$-gradient dominated. Consider Algorithm \ref{RGDGD_algorithm} with R-SD solver. Then total IFO complexity to achieve $\epsilon$-accurate solution is given by $\mathcal{O}\big( (n + L \tau n)\log(\frac{1}{\epsilon}) \big)$. Suppose additionally that $f$ has $G$-bounded gradient. That is, $\| \text{grad}f_{i}(x) \| \leq G$, with $i$ being a random index from $[n]$. Consider Algorithm \ref{RGDGD_algorithm} with R-SGD solver. Total IFO complexity to achieve $\epsilon$-accurate solution is $\mathcal{O}\big( \frac{LG^2}{\epsilon^2} \big)$.
\end{theorem}
\begin{proof} 
    The proof idea is similar to Theorems \ref{complexity_svrg_gd} and \ref{complexity_srg_gd}. We first consider a single epoch $k$. By retraction $L$-smoothness, 
    \begin{align}
        f(x_{t+1}^k) &\leq f(x_t^k) + \langle \text{grad}f(x_t^k) , -\eta \text{grad}f(x_{t}^k) \rangle + \frac{L}{2} \| -\eta \text{grad}f(x_t^k) \|^2 \nonumber\\
        &= f(x_t^k) - (\eta -\frac{L\eta^2}{2}) \| \text{grad}f(x_t^k) \|^2.  
    \end{align}
    Choose $\eta = \frac{1}{L}$ and summing this inequality from $t = 0,...,T_k-1$ gives
    \begin{align}
        \frac{1}{T_k}\sum_{t=0}^{T_k-1} \mathbb{E}\| \text{grad}f(x_t^k) \|^2 \leq  \frac{2L \mathbb{E}[f(x_0^k) - f(x_{T_k}^k)]}{T_k} \leq \frac{2L \Delta_{k-1}}{T_k},
    \end{align}
    where $\Delta_{k-1} := \mathbb{E}[f(x^{k-1}) - f(x^*)]$. Note the update rule of $x^k$ gives $\mathbb{E}\| \text{grad}f(x^k) \|^2 = \frac{1}{T_k}\sum_{t=0}^{T_k-1} \mathbb{E}\| \text{grad}f(x_t^k) \|^2$. Therefore, to ensure $\mathbb{E}\| \text{grad}f(x^k) \|^2 \leq \epsilon_k^2$, we require at least 
    \begin{equation}
        T_k = \frac{2L \Delta_{k-1}}{\epsilon_{k}^2} \leq \frac{2L \tau \mathbb{E}\| \text{grad}f(x^{k-1}) \|^2}{\epsilon_k^2} \leq \frac{2L \tau \epsilon_{k-1}^2}{\epsilon_{k}^2} = 8L \tau.         
    \end{equation}
    IFO complexity of a single epoch is given by $8L\tau n = \mathcal{O}(n + L \tau n)$. By similar argument, to ensure $\mathbb{E}\| \text{grad}f(x^K) \|^2 \leq \epsilon^2$, we require $\log(\frac{1}{\epsilon})$ epochs. Hence the total IFO complexity of R-SD is given as $\mathcal{O}\big( (n + L \tau n)\log(\frac{1}{\epsilon}) \big)$. This result matches the complexity of Euclidean gradient descent under gradient dominance condition (see \cite{ReddiSVRG2016,PolyakPLInEQ1963}).Similarly, for R-SGD, we have 
    \begin{align}
        \mathbb{E}[f(x_{t+1}^k)] &\leq \mathbb{E}[f(x_t^k) + \langle \text{grad}f(x_t^k) , -\eta \text{grad}f_{i_t^k}(x_{t}^k) \rangle + \frac{L}{2} \| -\eta \text{grad}f_{i_t^k}(x_t^k) \|^2] \nonumber\\
        &= \mathbb{E}[f(x_t^k)] - \eta \mathbb{E}\| \text{grad}f(x_t^k) \|^2 + \frac{L \eta^2 G^2}{2}. 
    \end{align}
    Choosing $\eta = \frac{z}{\sqrt{T_k}}$ where $z > 0$ is a constant and summing over $t = 0,...,T_k-1$, we have
    \begin{align}
        \frac{1}{T_k} \sum_{t=0}^{T_k -1} \mathbb{E}\| \text{grad}f(x_t^k) \|^2 \leq \frac{\Delta_{k-1}}{z\sqrt{T_k}} + \frac{L G^2 z}{2\sqrt{T_k}}.
    \end{align}
    Choose $z = \sqrt{\frac{2\Delta_{k-1}}{L G^2}}$ to minimize right hand side as $\frac{\sqrt{2 L G^2\Delta_{k-1}}}{\sqrt{T_k}}$. Hence to ensure $\mathbb{E}\| \text{grad}f(x^k) \|^2 \leq \epsilon_k^2$, we require at least
    \begin{equation}
        T_k = \frac{2 LG^2\Delta_{k-1}}{\epsilon_k^4} \leq \frac{2 LG^2 \epsilon_{k-1}^2}{\epsilon_k^4} = \frac{8LG^2}{\epsilon_k^2}. 
    \end{equation}
    IFO complexity of a single epoch is therefore $\mathcal{O}(\frac{LG^2}{\epsilon_k^2})$. To achieve $\epsilon$-accurate solution, we require $\log(\frac{1}{\epsilon})$ epochs and hence, the total IFO complexity of R-SGD is $\mathcal{O}\big( \frac{LG^2}{\epsilon^2} \big)$.
\end{proof}

\newpage
\section{Riemannian AbaSPIDER}
\label{r_abaspider_appendix}

Here we propose R-AbaSPIDER in Algorithm \ref{RAbaSPIDER_algorithm}, which is R-SPIDER with batch size adaptation. Let $k_0 = \lfloor k/p \rfloor \cdot p$, so that $k_0 \leq k \leq k_0 + p -1$. Define sigma algebras $\mathcal{F}_{k} := \{ \mathcal{S}_{\cdot, 0}, \mathcal{S}_{\cdot, 1}, ..., \mathcal{S}_{\cdot, k-1}\}$ and we can similarly establish a bound on the difference between $v_k$ and $\text{grad}f(x_k)$.

\begin{lemma}[Gradient estimation error bound for R-AbaSPIDER]
\label{estimation_bound_abaspider}
Suppose Assumption \ref{common_assump} holds and consider Algorithm \ref{RAbaSPIDER_algorithm}, the estimation bound of $v_k$ to full gradient is bounded as
    $$\mathbb{E}[\| v_k - \emph{grad}f(x_k) \|^2 | \mathcal{F}_{k_0}] \leq  \frac{(L_l + \theta G)^2}{S_2} \sum_{i=k_0}^{k_0+p-1} {\eta_{i}^2} + \mathbbm{1}_{\{ S_{1,k_0} < n \}} \frac{\sigma^2}{S_{1,k_0}}.$$
\end{lemma}
\begin{proof}
    The proof technique is exactly the same as the proof of Lemma \ref{lemma1_RSRG}. 
\end{proof}

\begin{algorithm}
 \caption{R-AbaSPIDER}
 \label{RAbaSPIDER_algorithm}
 \begin{algorithmic}[1]
  \STATE \textbf{Input:} Epoch length $K$, batch gradient frequency $p$, step size $\{\eta_k\}$, batch size $S_2$, auxiliary parameters $\alpha_1, \alpha_2, \beta_{0}$, initialization ${x}_0$.
  \FOR {$k = 0,...,K-1 $}
  \IF{mod$(k, p) = 0$}
    \STATE $S_{1,k} = \begin{cases} \min \{\alpha_1 \sigma^2/\beta_k , n\}, & \text{  (finite-sum)}\\
    \min\{\alpha_1 \sigma^2/\beta_k , \alpha_2 \sigma^2/\epsilon^2 \}, & \text{ (online) } \end{cases}$
    \STATE Draw a sample $\mathcal{S}_{1,k}$ from $[n]$ of size $S_{1,k}$ without replacement.
    \STATE $v_k = \text{grad}f_{\mathcal{S}_{1,k}}(x_k)$.
    \STATE $\beta_{k+1} = {\| v_k \|^2}/{p}$. 
  \ELSE 
    \STATE Draw a sample $\mathcal{S}_{2,k}$ from $[n]$ of size $S_2$ with replacement.
    \STATE $v_k = \text{grad}f_{\mathcal{S}_{2,k}}(x_k) -  \mathcal{T}_{x_{k-1}}^{x_k}(\text{grad}f_{\mathcal{S}_{2,k}}(x_{k-1}) - v_{k-1})$.
    \STATE $\beta_{k+1} = \beta_{k} + {\| v_k\|^2}/{p}$.
  \ENDIF
    \STATE $x_{k+1} = R_{x_k}(-\eta_k \frac{v_k}{\| v_k \|})$.
  \ENDFOR
  \STATE \textbf{Output:} $\Tilde{x}$ uniformly selected at random from $\{x_k\}_{k=0}^{K-1}$. 
 \end{algorithmic} 
\end{algorithm}

\begin{theorem}[Convergence analysis of R-AbaSPIDER]
    Let $x^* \in \mathcal{M}$ be an optimal point and suppose Assumption \ref{common_assump} holds. Consider Algorithm \ref{RAbaSPIDER_algorithm} with step size $\eta_k = \min\{ \frac{\epsilon}{\max\{L, L_l + \theta G \} n_0}, \frac{\| v_k \|}{2 \max\{L, L_l + \theta G\} n_0} \}, \alpha \geq 2$ and following parameters  
    \begin{align*}
        S_2 = \frac{\sqrt{n}}{n_0}, \quad p = n_0\sqrt{n}, \quad n_0 \in [1, \sqrt{n}], \quad &\text{ (finite-sum) }  \nonumber\\
        S_2 = \frac{2\sigma}{\epsilon n_0}, \quad p = \frac{\sigma n_0}{\epsilon}, \quad n_0 \in [1, \frac{2\sigma^2}{\epsilon}], \quad &\text{ (online) }
    \end{align*}
    where $n_0$ is a free parameter to ensure $S_2 \geq 1$. Under both finite-sum and online setting, output $\Tilde{x}$ after running $K$ iterations satisfies
    \begin{equation}
        \mathbb{E}\| \emph{grad}f(\Tilde{x}) \|^2 \leq \frac{2\Delta}{K \Tilde{\eta}}  + \frac{3}{2}\epsilon^2, \nonumber
    \end{equation}
    where $\Tilde{\eta} := \min\{ \frac{\epsilon}{\max\{L, L_l + \theta G\} n_0 G}, \frac{1}{2\max\{L, L_l + \theta G\}n_0}\}$ is an upper bound on $\eta_k$. 
\end{theorem}
\begin{proof}
    Denote $\Tilde{\eta}_k := \eta_k/\| v_k\|$ as the effective step size. By applying retraction $L$-smoothness we have,
    \begin{align}
        f(x_{k+1}) &\leq f(x_k) - \Tilde{\eta}_k \langle \text{grad}f(x_k), v_k \rangle + \frac{\Tilde{\eta}_k^2 L}{2} \|v_k \|^2 \nonumber\\
        &= f(x_k) - \frac{\Tilde{\eta}_k}{2} \| \text{grad}f(x_k) \|^2 - \frac{\Tilde{\eta}_k}{2}\| v_k \|^2 + \frac{\Tilde{\eta}_k}{2} \| v_k - \text{grad}f(x_k) \|^2 + \frac{\Tilde{\eta}_k^2 L}{2}\| v_k \|^2 \nonumber\\
        &= f(x_k )- \frac{\Tilde{\eta}_k}{2} \| \text{grad}f(x_k) \|^2 + \frac{\Tilde{\eta}_k}{2} \| v_k - \text{grad}f(x_k) \|^2 - \frac{\Tilde{\eta}_k }{2}\big( 1 - \Tilde{\eta}_k L  \big) \| v_k \|^2. 
    \end{align}
    Rearranging this inequality and taking expectation with respect to $\mathcal{F}_{k_0}$ gives
    \begin{align}
        \mathbb{E}[\| \text{grad}f(x_k) \|^2 | \mathcal{F}_{k_0}] &\leq \frac{2}{\Tilde{\eta}_k} \mathbb{E}[f(x_k) - f(x_{k+1})| \mathcal{F}_{k_0}] + \mathbb{E}[\| v_k - \text{grad}f(x_k) \|^2  | \mathcal{F}_{k_0}] - \big( 1 - \Tilde{\eta}_k L  \big) \mathbb{E} [\| v_k \|^2 | \mathcal{F}_{k_0}]. \label{jfidjidjfidj} 
    \end{align}
    Given $\Tilde{\eta}_k = \eta_k/\|v_k \| = \min\{ \frac{\epsilon}{\max\{L, L_l + \theta G\} n_0 \| v_k\|}, \frac{1}{2\max\{L, L_l + \theta G\}n_0}\}$, we have 
    \begin{align}
         \Tilde{\eta}_k &\leq \frac{1}{2\max\{ L, L_l +\theta G \} n_0} \leq \frac{1}{2\max\{ L, L_l +\theta G \} }, \\
         (1-\Tilde{\eta}_k L) \| v_k\|^2 &\geq \Big(  1 - \frac{L}{2\max\{ L, L_l +\theta G \}}\Big) \| v_k \|^2 \geq \frac{\|v_k \|^2}{2},
    \end{align}
    where we note $n_0 \geq 1$. Next, we further simplify bound on gradient estimation error. Given parameter $\eta_k = \min\{ \frac{\epsilon}{\max\{L, L_l + \theta G\} n_0}, \frac{\| v_k \|}{2\max\{L, L_l + \theta G\}n_0} \} \leq \frac{\epsilon}{\max\{L, L_l + \theta G\} n_0}$, we obtain the following bounds by Lemma \ref{estimation_bound_abaspider}. Under finite-sum setting, 
    \begin{align}
        \mathbb{E}[\| v_k - \text{grad} f(x_k)\|^2 | \mathcal{F}_{k_0}] \leq \frac{(L_l + \theta G)^2}{S_2} \sum_{t=k_0}^{k_0+p-1} \eta_t^2 + \mathbbm{1}_{\{ S_{1,k_0} < n \}} \frac{\sigma^2}{S_{1,k_0}} &\leq \frac{(L_l + \theta G)^2}{S_2} \frac{p \epsilon^2}{(\max\{L, L_l + \theta G\})^2 n_0^2} + \frac{\beta_{k_0}}{\alpha_1} \nonumber\\
        &\leq {\epsilon^2} + \frac{\beta_{k_0}}{\alpha},
    \end{align}
    where we choose $\alpha_1 \geq \alpha$. Similarly, under online setting, 
    \begin{align}
        \mathbb{E}[\| v_k - \text{grad} f(x_k)\|^2 | \mathcal{F}_{k_0}] &\leq \frac{(L_l + \theta G)^2}{S_2} \sum_{t=k_0}^{k_0+p-1} \eta_t^2 + \mathbbm{1}_{\{ S_{1,k_0} < n \}} \frac{\sigma^2}{S_{1,k_0}} \nonumber\\
        &\leq \frac{(L_l + \theta G)^2}{S_2} \frac{p \epsilon^2}{(\max\{L, L_l + \theta G\})^2 n_0^2} + \max\{ \frac{\beta_{k_0}}{\alpha}, \frac{\epsilon^2}{\alpha} \} \nonumber\\
        &\leq \frac{\epsilon^2}{2} + \frac{\beta_{k_0}}{\alpha} + \frac{\epsilon^2}{\alpha}
    \end{align}
    where we choose $\alpha_1, \alpha_2 \geq \alpha$ and use the fact that $\max\{a,b \} \leq a +b$. As long as $\alpha \geq 2$, for both finite-sum and online setting, we have 
    \begin{equation}
        \mathbb{E}[\| v_k - \text{grad} f(x_k)\|^2 | \mathcal{F}_{k_0}] \leq \epsilon^2 + \frac{\beta_{k_0}}{\alpha}. 
    \end{equation}
    Substituting these results back to \eqref{jfidjidjfidj} gives
    \begin{align}
        \mathbb{E}[\| \text{grad}f(x_k) \|^2 | \mathcal{F}_{k_0}] &\leq \frac{2}{\Tilde{\eta}_k} \mathbb{E}[f(x_k) - f(x_{k+1})| \mathcal{F}_{k_0}] +{\epsilon^2} + \frac{\beta_{k_0}}{\alpha} - \frac{1}{2} \mathbb{E} [\| v_k \|^2 | \mathcal{F}_{k_0}]. \label{fdkfjdkfde}
    \end{align}
    Note that by Assumption \ref{bounded_norm_assump} and triangle inequality, norm of $v_k$ is bounded by a constant.
    \begin{align}
        \| v_k \| &= \| \text{grad}f_{\mathcal{S}_{2,k}}(x_k) -  \mathcal{T}_{x_{k-1}}^{x_k}(\text{grad}f_{\mathcal{S}_{2,k}}(x_{k-1}) - v_{k-1}) \| \leq \|\text{grad}f_{\mathcal{S}_{2,k}}(x_k) \| + \| \text{grad}f_{\mathcal{S}_{2,k}}(x_{k-1}) \| + \|v_{k-1} \| \leq 3 G. 
    \end{align}
    Hence, we have 
    \begin{align}
        \Tilde{\eta}_k &=  \min\{ \frac{\epsilon}{\max\{L, L_l + \theta G\} n_0 \| v_k\|}, \frac{1}{2\max\{L, L_l + \theta G\}n_0}\} \nonumber\\
        &\geq \min\{ \frac{\epsilon}{\max\{L, L_l + \theta G\} n_0 G}, \frac{1}{2\max\{L, L_l + \theta G\}n_0}\} =: \Tilde{\eta}. 
    \end{align}
    Substitute this result back in \eqref{fdkfjdkfde} and telescoping from $k = 0,..., K-1$ gives 
    \begin{align}
        \frac{1}{K}\sum_{k=0}^{K-1} \mathbb{E}\| \text{grad}f(x_k) \|^2 &\leq \frac{2}{K \Tilde{\eta}} \mathbb{E}[f(x_0) - f(x_K)] + \epsilon^2 + \frac{1}{\alpha K }\sum_{k=0}^{K-1} \beta_{k_0} - \frac{1}{2K}\sum_{k=0}^{K-1}\mathbb{E}\| v_k\|^2 \nonumber\\
        &\leq \frac{2}{K \Tilde{\eta}} \Delta + \epsilon^2 + \frac{1}{\alpha K }\sum_{k=0}^{K-1} \beta_{k_0} - \frac{1}{2K}\sum_{k=0}^{K-1}\mathbb{E}\| v_k\|^2, \label{fdooppppoo}
    \end{align}
    where $\Delta := f(x_0) - f(x^*)$. By definition of $\beta_{k}$ and $k_0$, we obtain
    \begin{align}
        \sum_{k=0}^{K-1} \beta_{k_0} &= p (\beta_0 + \beta_p + \cdots + \beta_{K_0 - p}) + (K - K_0) \beta_{K_0} \nonumber\\
        &\leq p( \beta_0 + \beta_p + \cdots + \beta_{K_0}) \nonumber\\
        &= p(\beta_0 + \frac{1}{p} \sum_{j=0}^{p-1} \| v_j\|^2 + \frac{1}{p} \sum_{j=p}^{2p-1} \| v_j \|^2 + \cdots + \frac{1}{p} \sum_{K_0 - p}^{K_0 -1}\| v_j \|^2) \nonumber\\
        &= p\beta_0 + \sum_{j=0}^{K_0-1}\|v_j \|^2 \leq \epsilon^2 K + \sum_{j=0}^{K-1}\| v_j\|^2, \label{fdderrttrtrt}
    \end{align}
    where $K_0 := \lfloor K/p \rfloor \cdot p$ and we choose $\beta_0 \leq \epsilon^2K/p$. Combining \eqref{fdderrttrtrt} with \eqref{fdooppppoo} gives
    \begin{align}
        \frac{1}{K}\sum_{k=0}^{K-1} \mathbb{E}\| \text{grad}f(x_k) \|^2 &\leq \frac{2}{K \Tilde{\eta}} \Delta + \epsilon^2 + \frac{\epsilon^2}{\alpha} - (\frac{1}{2K} - \frac{1}{\alpha K}) \sum_{k=0}^{K-1}\mathbb{E}\| v_k \|^2 \nonumber\\
        &\leq \frac{2}{K \Tilde{\eta}} \Delta + \frac{3}{2}\epsilon^2.
    \end{align}
    With the choice of $\alpha \geq 2$. By noting output $\Tilde{x}$ is uniformly chosen from $\{ x_k \}_{k=1}^{K-1}$, we have 
    \begin{equation}
        \mathbb{E}\| \text{grad}f(x_k) \|^2 \leq \frac{2}{K \Tilde{\eta}} \Delta + \frac{3}{2}\epsilon^2.
    \end{equation}
    To ensure $\mathbb{E}\| \text{grad}f(\Tilde{x}) \|^2 \leq 2\epsilon^2$, we require at least 
    \begin{equation}
        K = \frac{4\Delta}{\Tilde{\eta} \epsilon^2} = \frac{4\Delta}{\epsilon^2} \max\{ \frac{\Theta n_0 G}{\epsilon}, 2\Theta n_0\} = \frac{4\Xi\Delta}{\epsilon^2} = \mathcal{O}\big( \frac{\Xi}{\epsilon^2} \big),
    \end{equation}
    where $\Xi := \max\{ \frac{\Theta n_0 G}{\epsilon}, 2\Theta n_0\}$. Denote a similar average batch size as $\Tilde{S} := \frac{1}{K} \sum_{k=0}^{K-1} S_{1,k}$. Then the total IFO complexity is given by 
    \begin{equation}
        \lceil \frac{K}{p} \rceil \Tilde{S} + K S_2 = 
        \begin{cases} \mathcal{O}\big( \Tilde{S} + \frac{\Xi \Tilde{S}}{\sqrt{n}\epsilon^2} + \frac{\Xi \sqrt{n}}{\epsilon^2} \big), & \text{ (finite-sum) }\\
        \mathcal{O} \big( \frac{\Xi \Tilde{S}}{\sigma \epsilon} + \frac{\Xi \sigma}{\epsilon^3} \big), & \text{ (online) } \end{cases} \label{spider_complexity}
    \end{equation}
    Note by definition of $\Xi$, when $\epsilon$ is small, $\Xi = \mathcal{O} \big(\frac{\Theta G}{\epsilon} \big)$ and therefore, complexities in \eqref{spider_complexity} are worse than that of R-SPIDER by a factor of $1/\epsilon$. 
\end{proof}

\newpage
\section{Additional Experiment Results}
\label{additional_experiment_res_appendix}

\subsection{PCA problem on Grassmann manifold}
We here present results on synthetic datasets by varying $n$ and $d$ and also examine result sensitivity on all datasets by conducting three independent runs.

\begin{figure}[h]
\captionsetup{justification=centering}
    \centering
    \subfloat[{Run} 1]{\includegraphics[width = 0.28\textwidth, height = 0.21\textwidth]{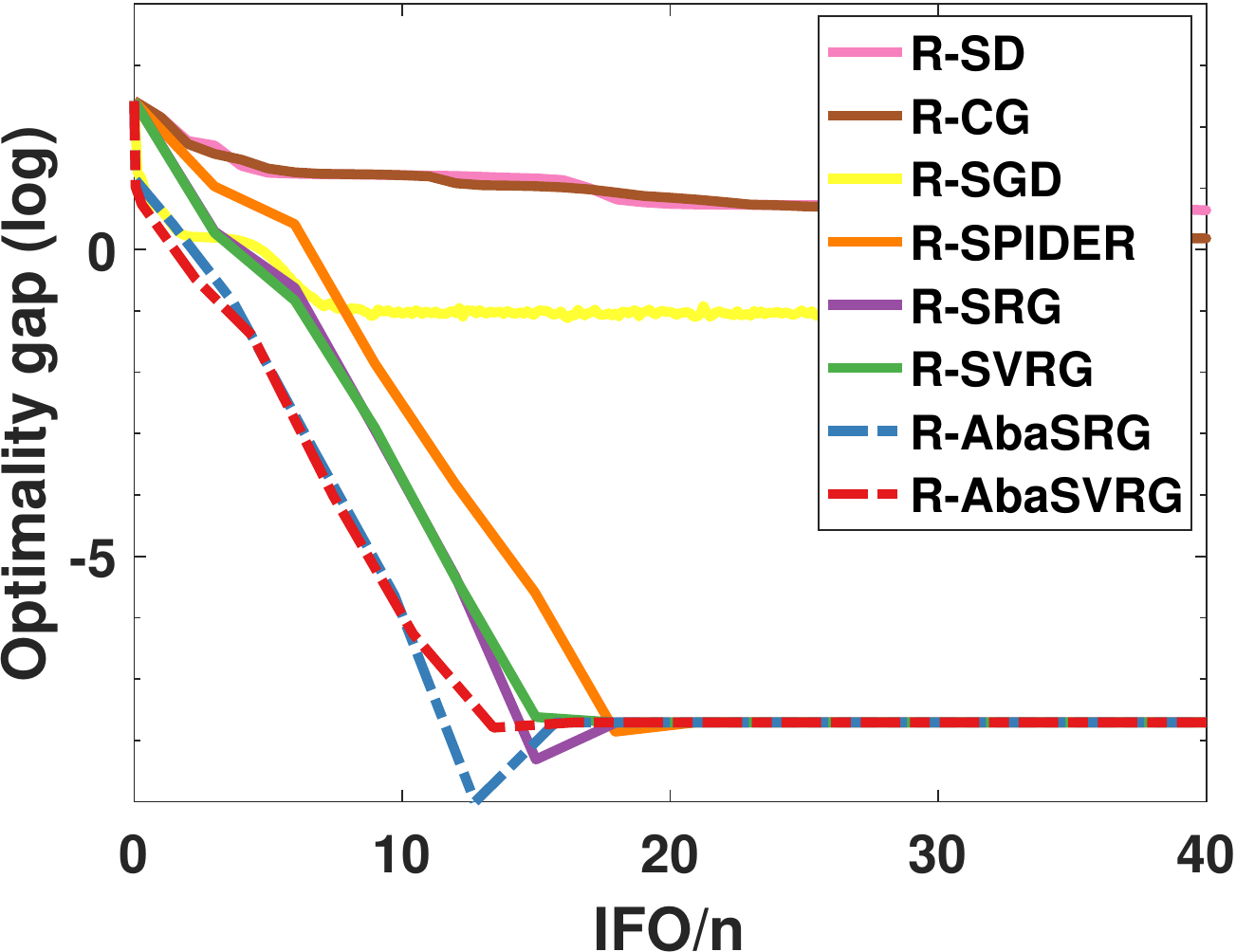}}
    \hspace{0.05in}
    \subfloat[{Run} 2]{\includegraphics[width = 0.28\textwidth, height = 0.21\textwidth]{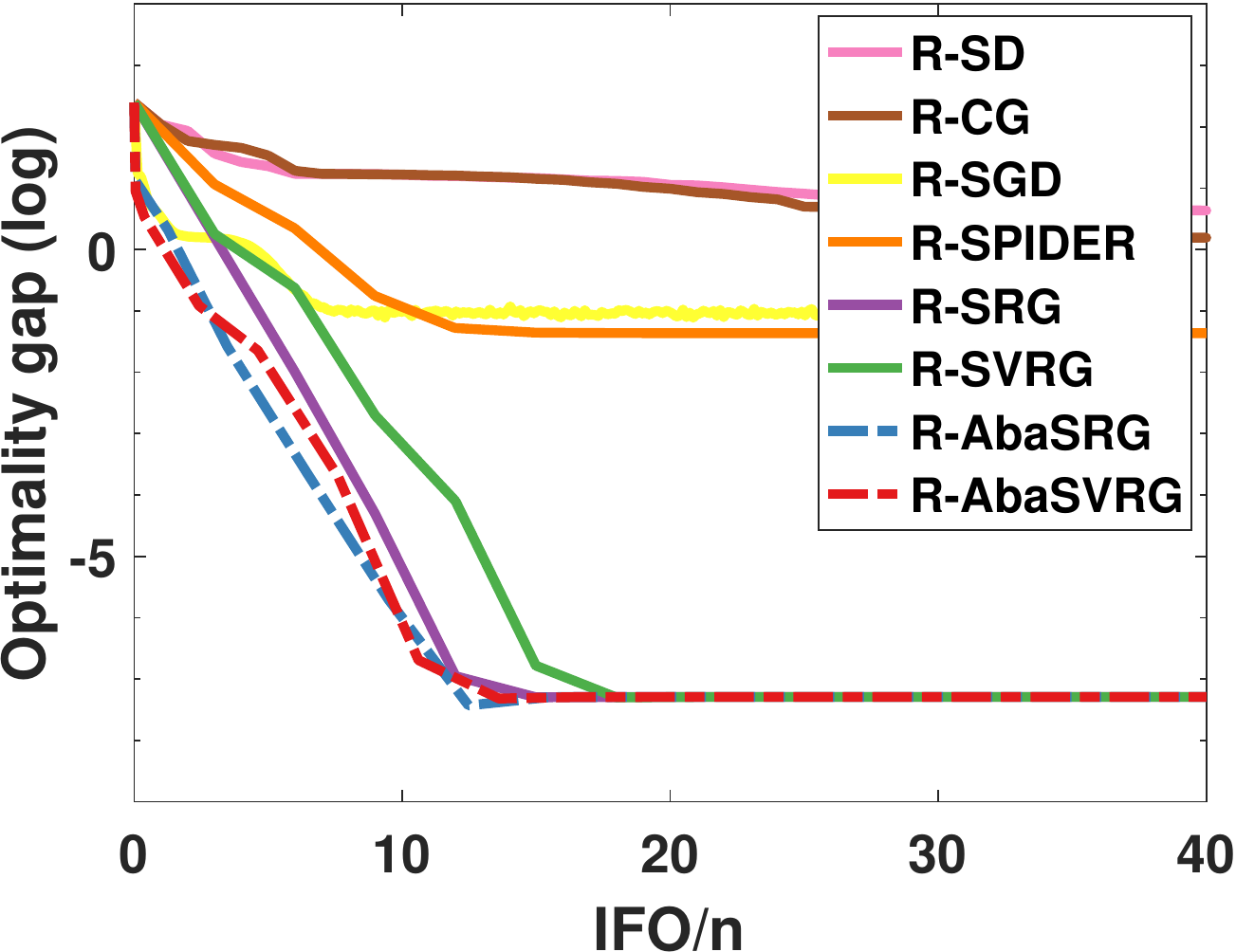}}
    \hspace{0.05in}
    \subfloat[{Run} 3]{\includegraphics[width = 0.28\textwidth, height = 0.21\textwidth]{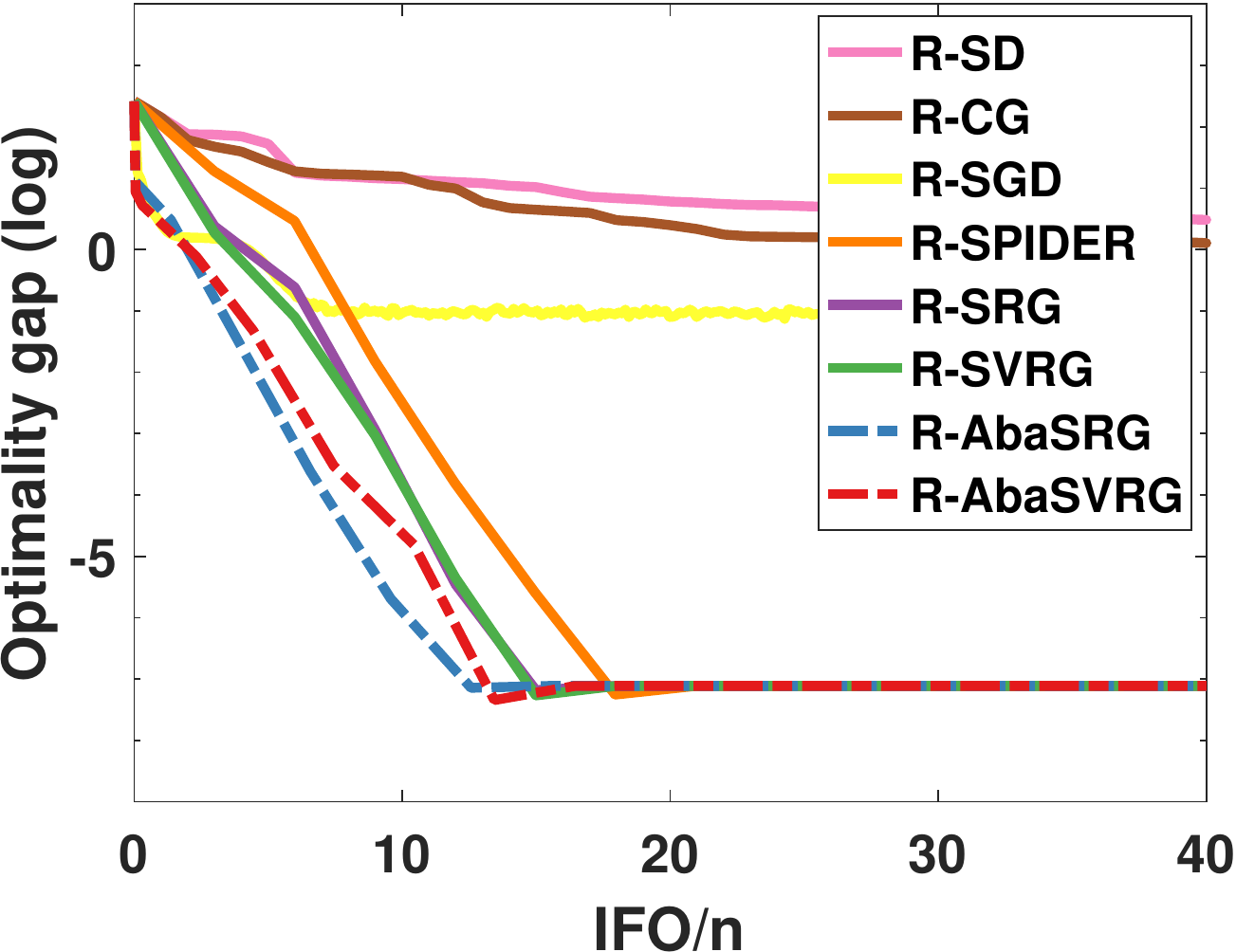}}
    \hspace{0.05in}\\
    \caption{Synthetic dataset with $n = 100000, d = 200, r = 5$.}
    \label{add_pca_1}
\end{figure}

\begin{figure}[h]
\captionsetup{justification=centering}
    \centering
    \subfloat[{Run} 1]{\includegraphics[width = 0.28\textwidth, height = 0.21\textwidth]{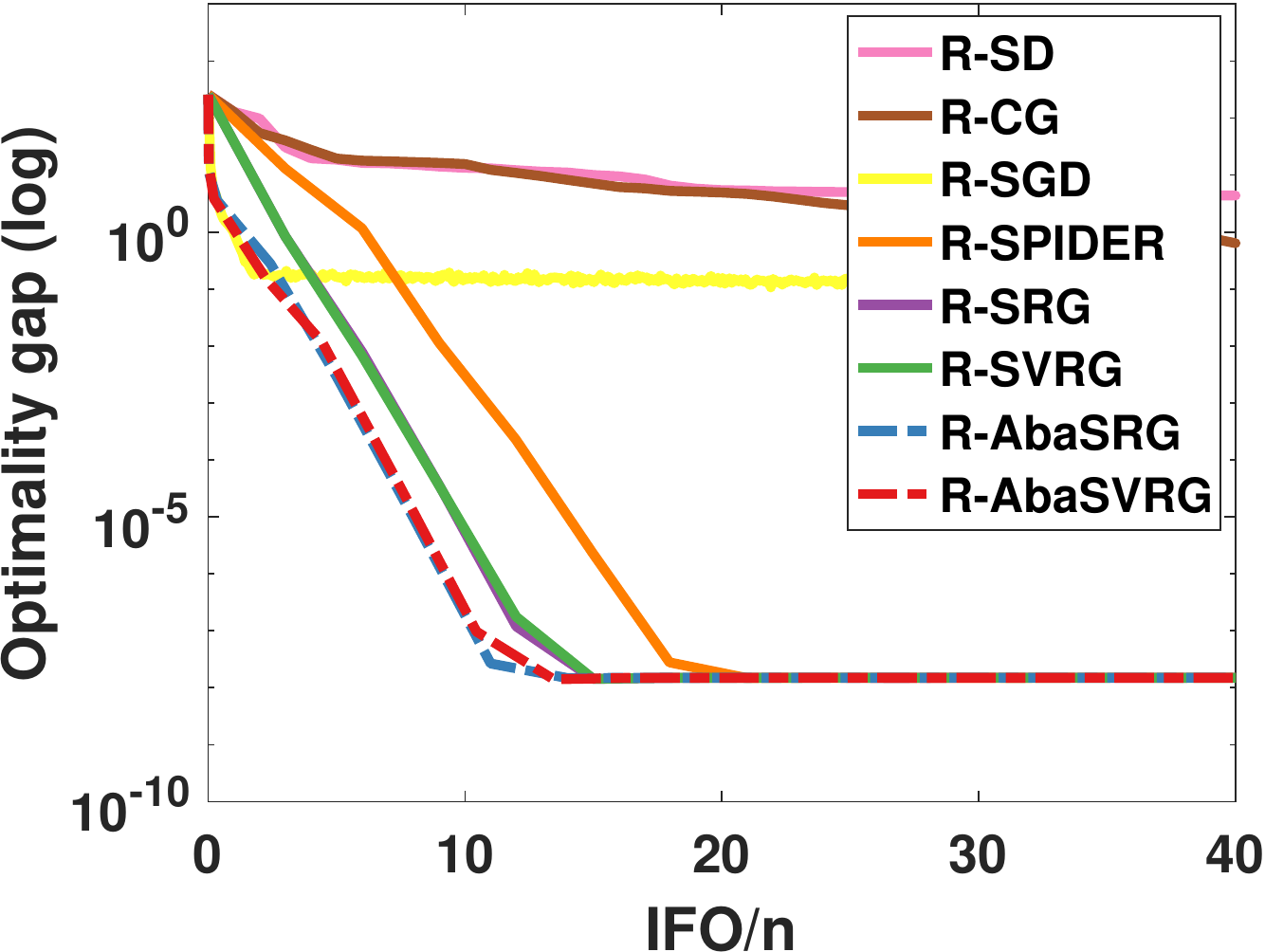}}
    \hspace{0.05in}
    \subfloat[{Run} 2]{\includegraphics[width = 0.28\textwidth, height = 0.21\textwidth]{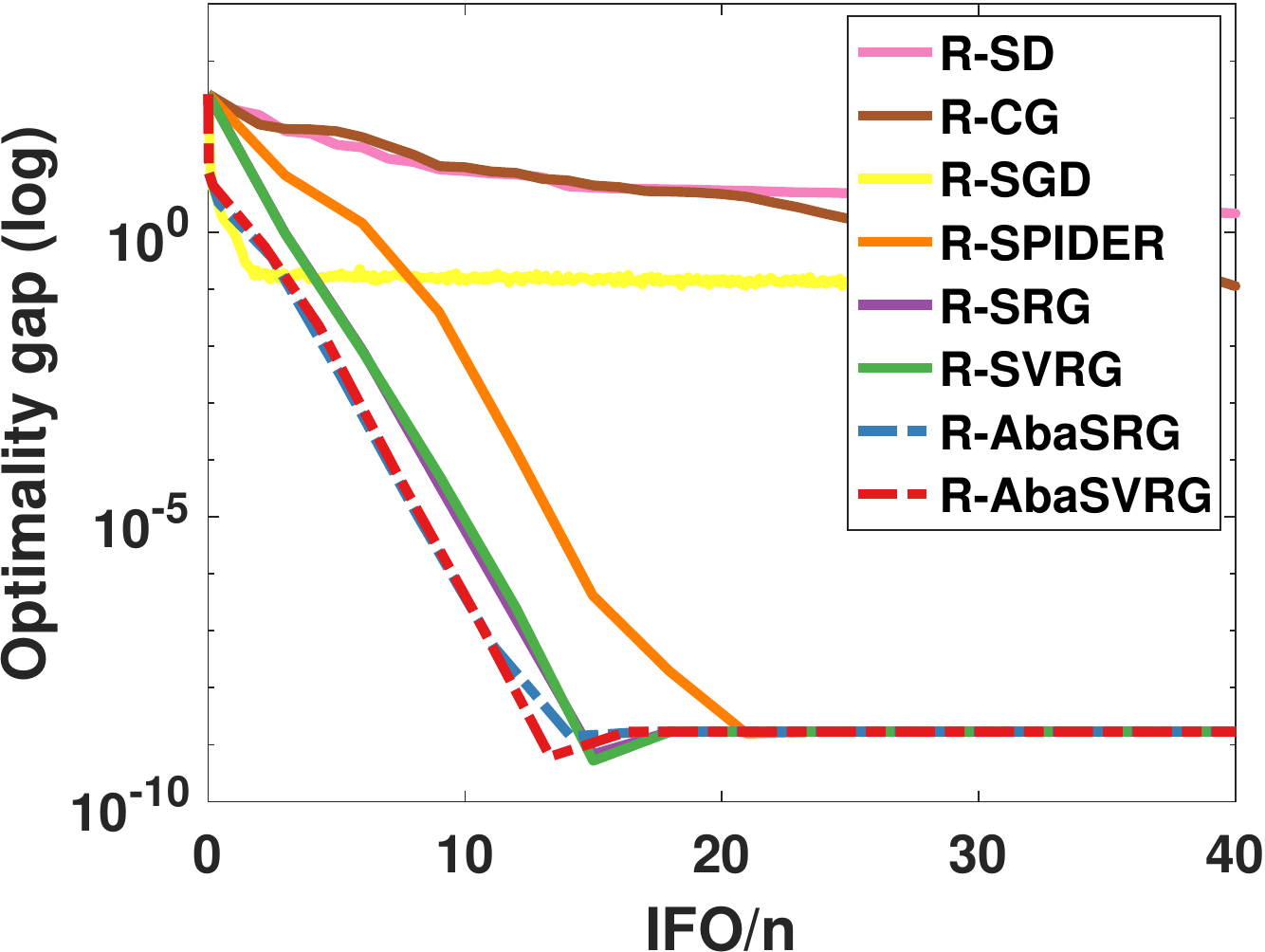}}
    \hspace{0.05in}
    \subfloat[{Run} 3]{\includegraphics[width = 0.28\textwidth, height = 0.21\textwidth]{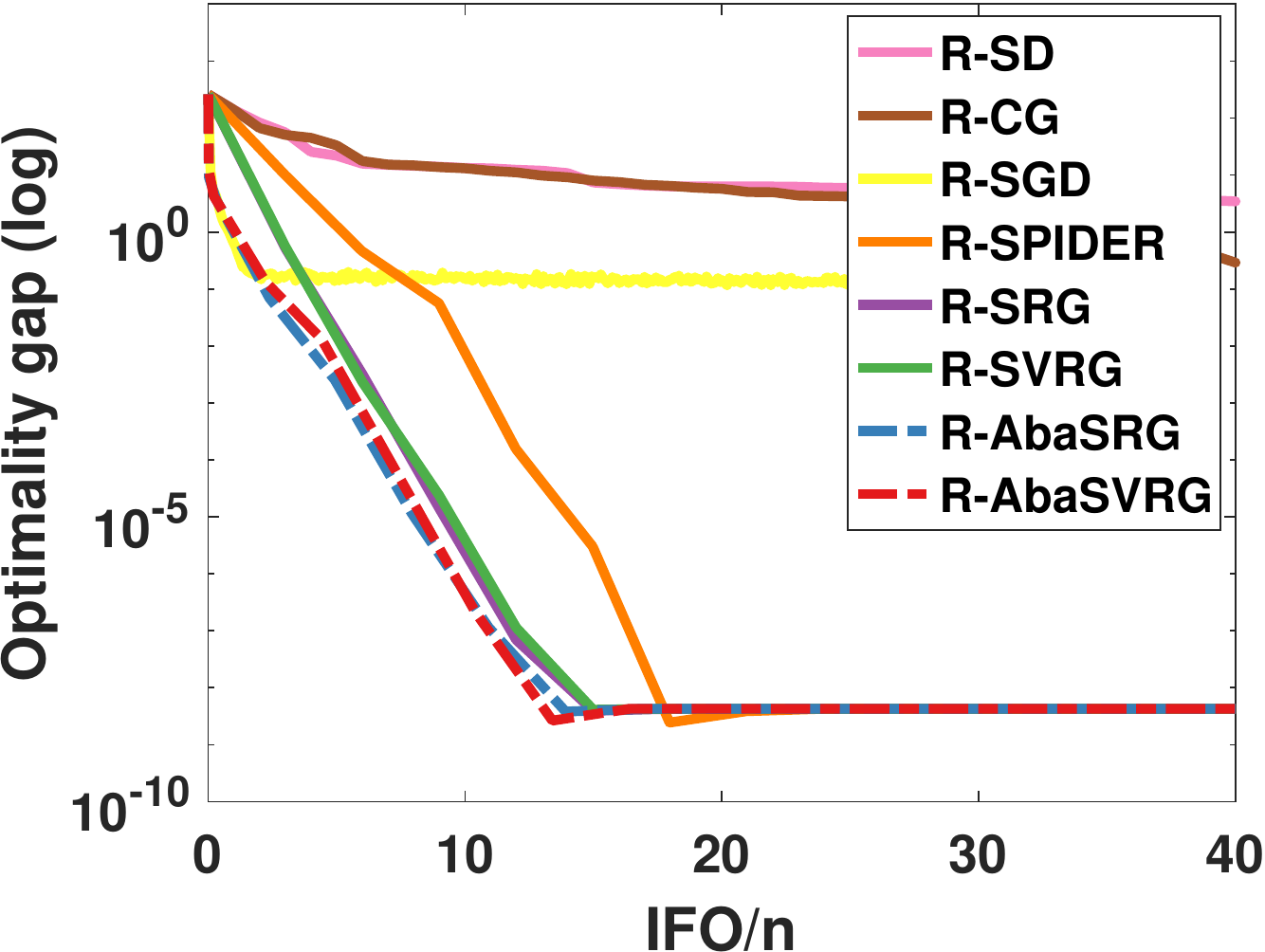}}
    \hspace{0.05in}\\
    \caption{Synthetic dataset with $n = 200000, d = 200, r = 5$.}
    \label{add_pca_2}
\end{figure}

\begin{figure}[h]
\captionsetup{justification=centering}
    \centering
    \subfloat[{Run} 1]{\includegraphics[width = 0.28\textwidth, height = 0.21\textwidth]{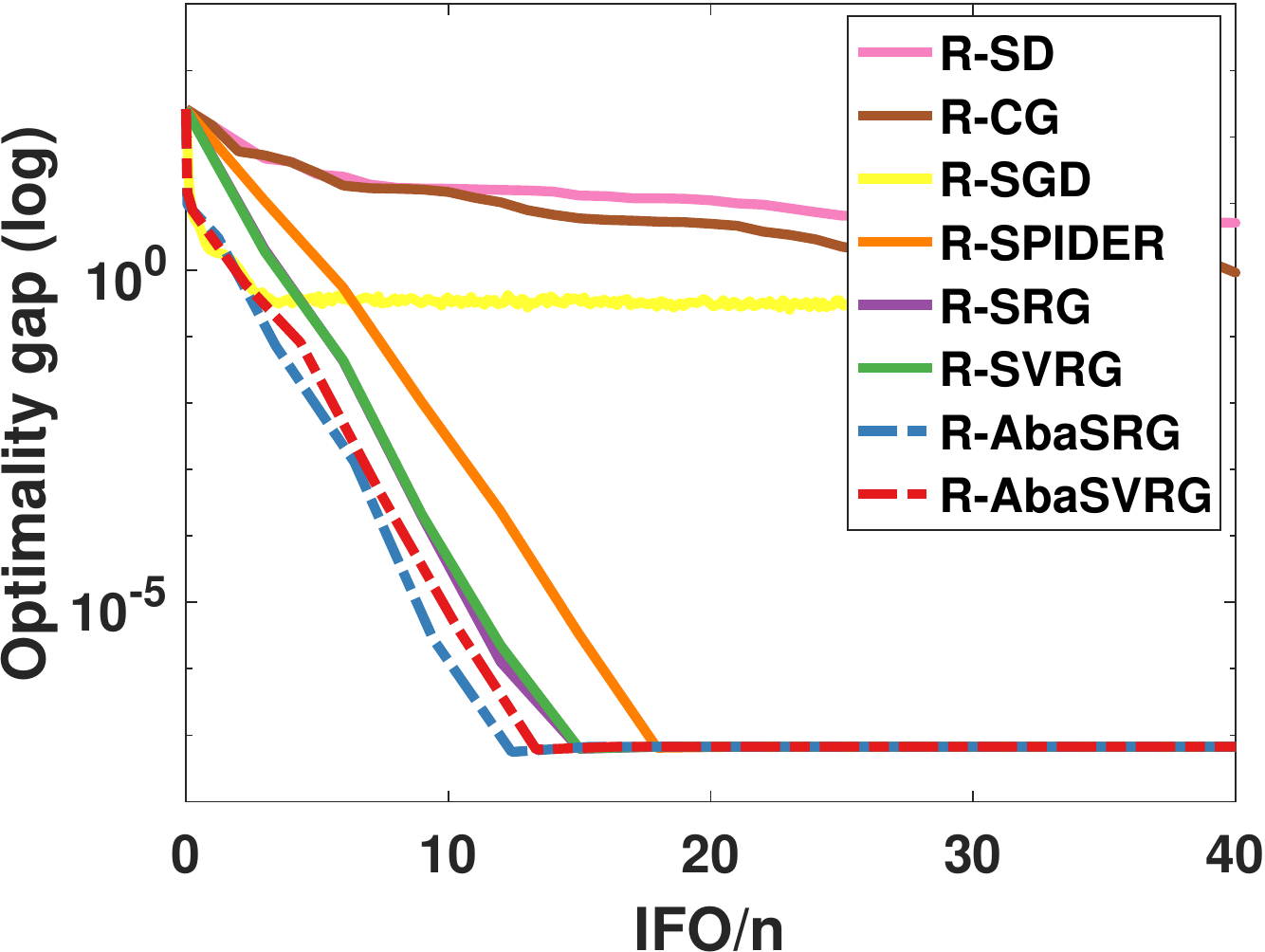}}
    \hspace{0.05in}
    \subfloat[{Run} 2]{\includegraphics[width = 0.28\textwidth, height = 0.21\textwidth]{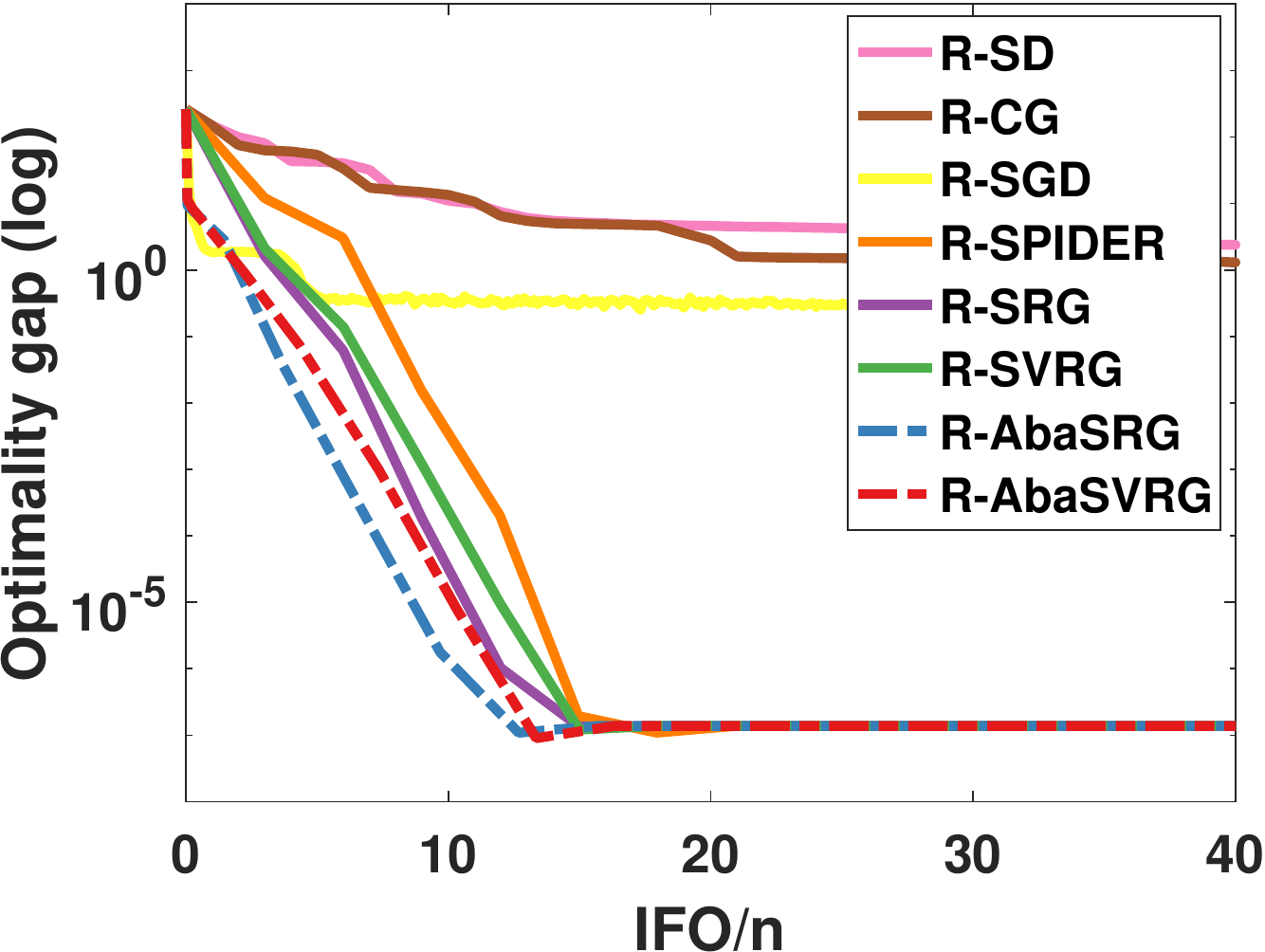}}
    \hspace{0.05in}
    \subfloat[{Run} 3]{\includegraphics[width = 0.28\textwidth, height = 0.21\textwidth]{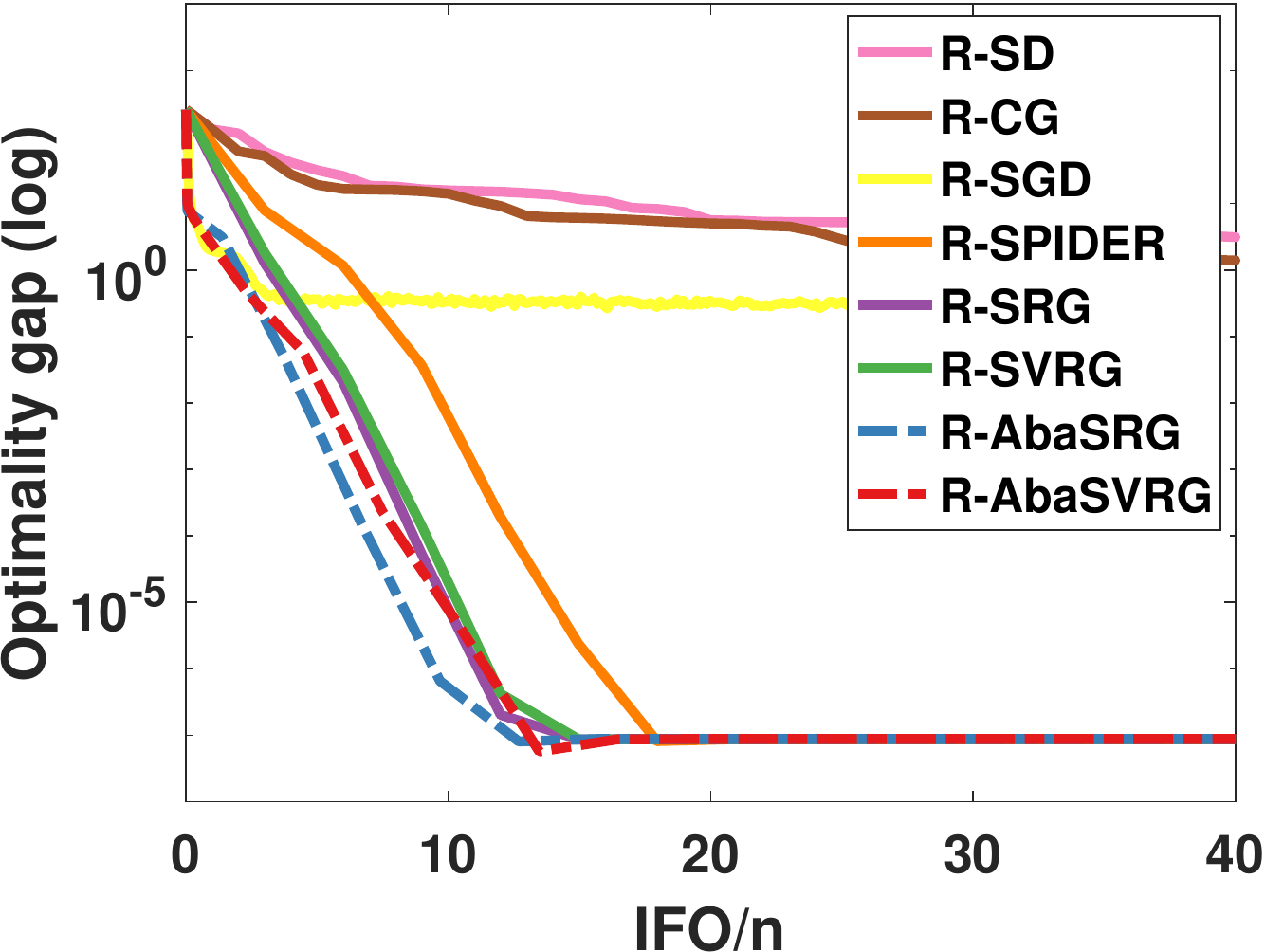}}
    \hspace{0.05in}\\
    \caption{Synthetic dataset with $n = 100000, d = 300, r = 5$.}
    \label{add_pca_3}
\end{figure}

\subsection{LRMC on Grassmann manifold}
\label{LRMC_appendix}
\textbf{Additional results on synthetic datasets.} We first present three independent runs in Fig. \ref{add_LRMC_syn_run_appendix} to test the sensitivity of batch size adaptation on baseline synthetic dataset with $n = 20000, d = 100, r = 5, \text{cn} = 50, \text{os} = 8, \varepsilon = 10^{-10}$. We also compare algorithms on datasets with different characteristics. Specifically, we consider a large-scale dataset with $n = 40000$, a high dimensional dataset with $d = 200$, a high-rank dataset with $r = 10$, an ill-conditioned dataset with $\text{cn} = 100$, a low-sampling dataset with $\text{os} = 4$ and a noisy dataset with $\varepsilon = 10^{-8}$. Test MSE results are presented in Fig. \ref{add_LRMC_syn_charcter_appendix}.

\begin{figure}[H]
\captionsetup{justification=centering}
    \centering
    \subfloat[Run 1]{\includegraphics[width = 0.28\textwidth, height = 0.21\textwidth]{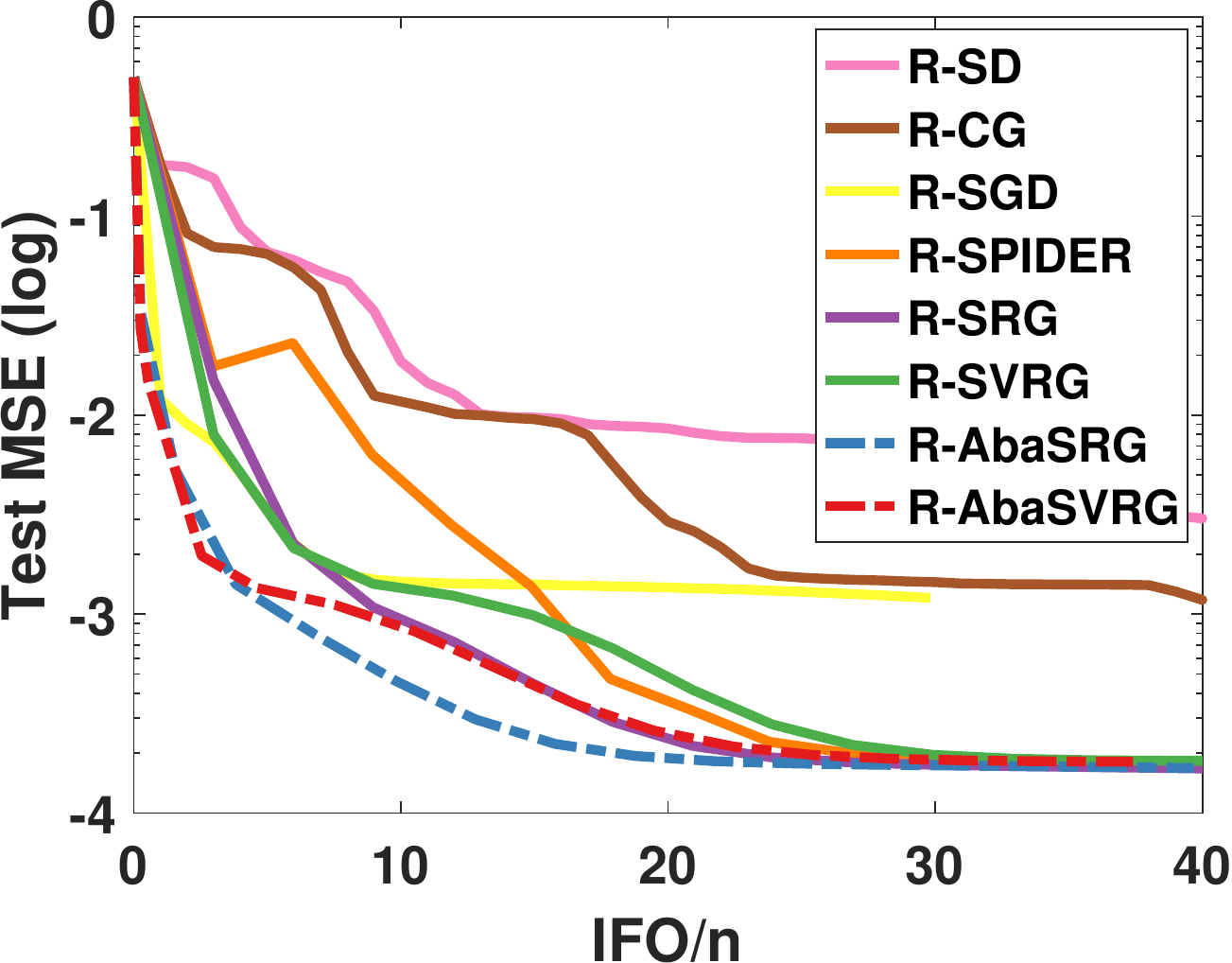}}
    \hspace{0.05in}
    \subfloat[Run 2]{\includegraphics[width = 0.28\textwidth, height = 0.21\textwidth]{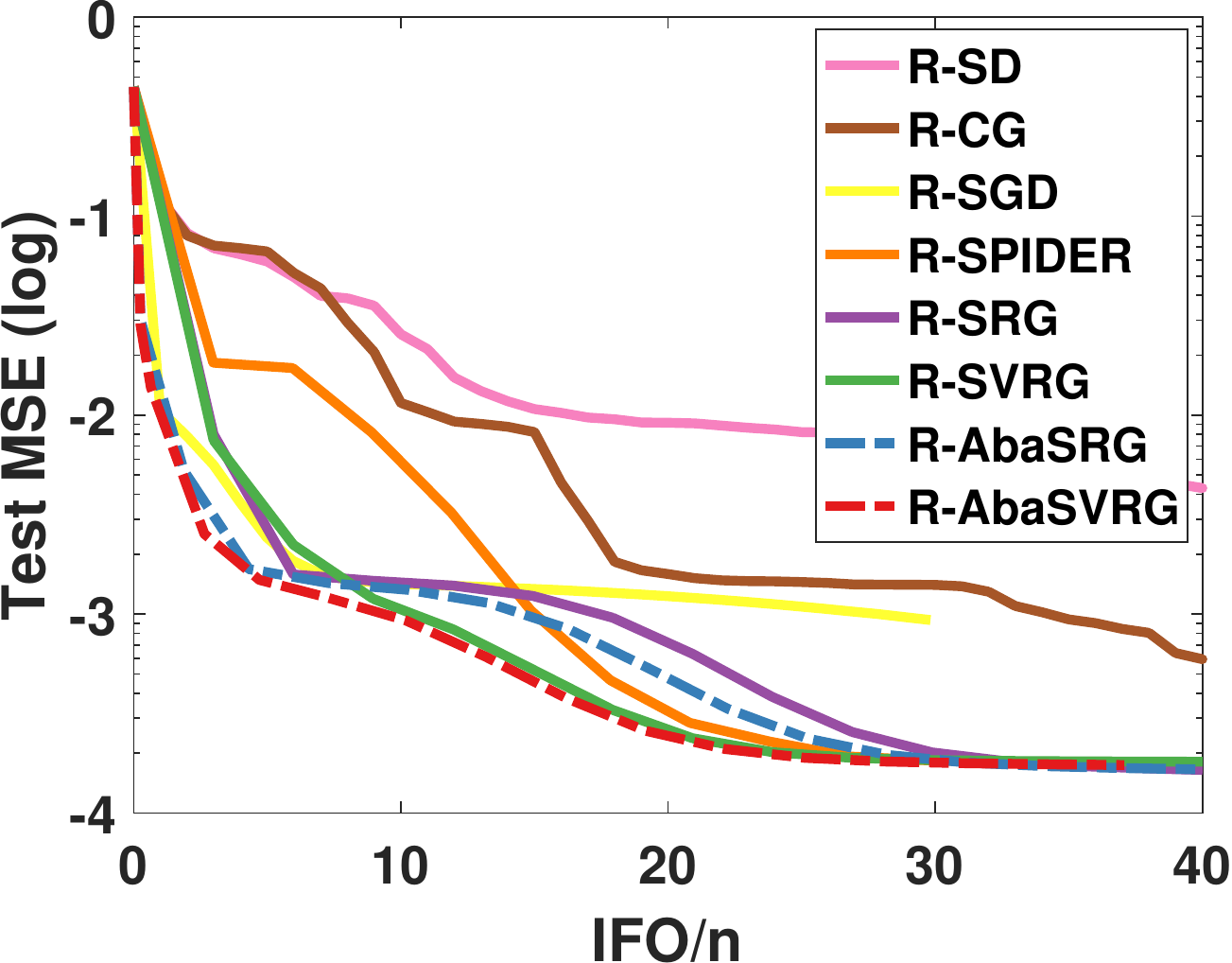}}
    \hspace{0.05in}
    \subfloat[Run 3]{\includegraphics[width = 0.28\textwidth, height = 0.21\textwidth]{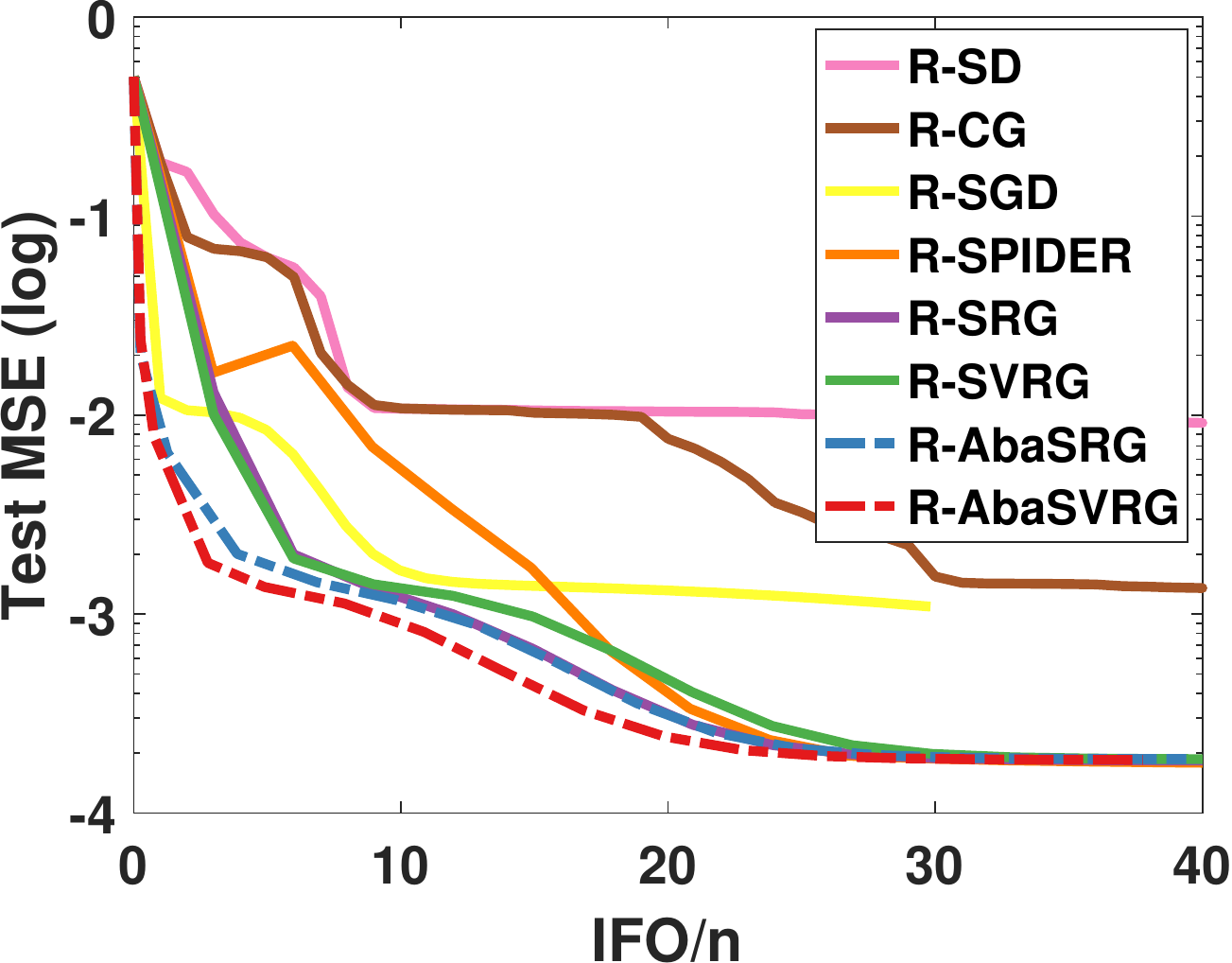}}
    \hspace{0.05in}
    \caption{LRMC Result sensitivity on baseline synthetic dataset}
    \label{add_LRMC_syn_run_appendix}
\end{figure}

\begin{figure}[H]
\captionsetup{justification=centering}
    \centering
    \subfloat[Large scale]{\includegraphics[width = 0.28\textwidth, height = 0.21\textwidth]{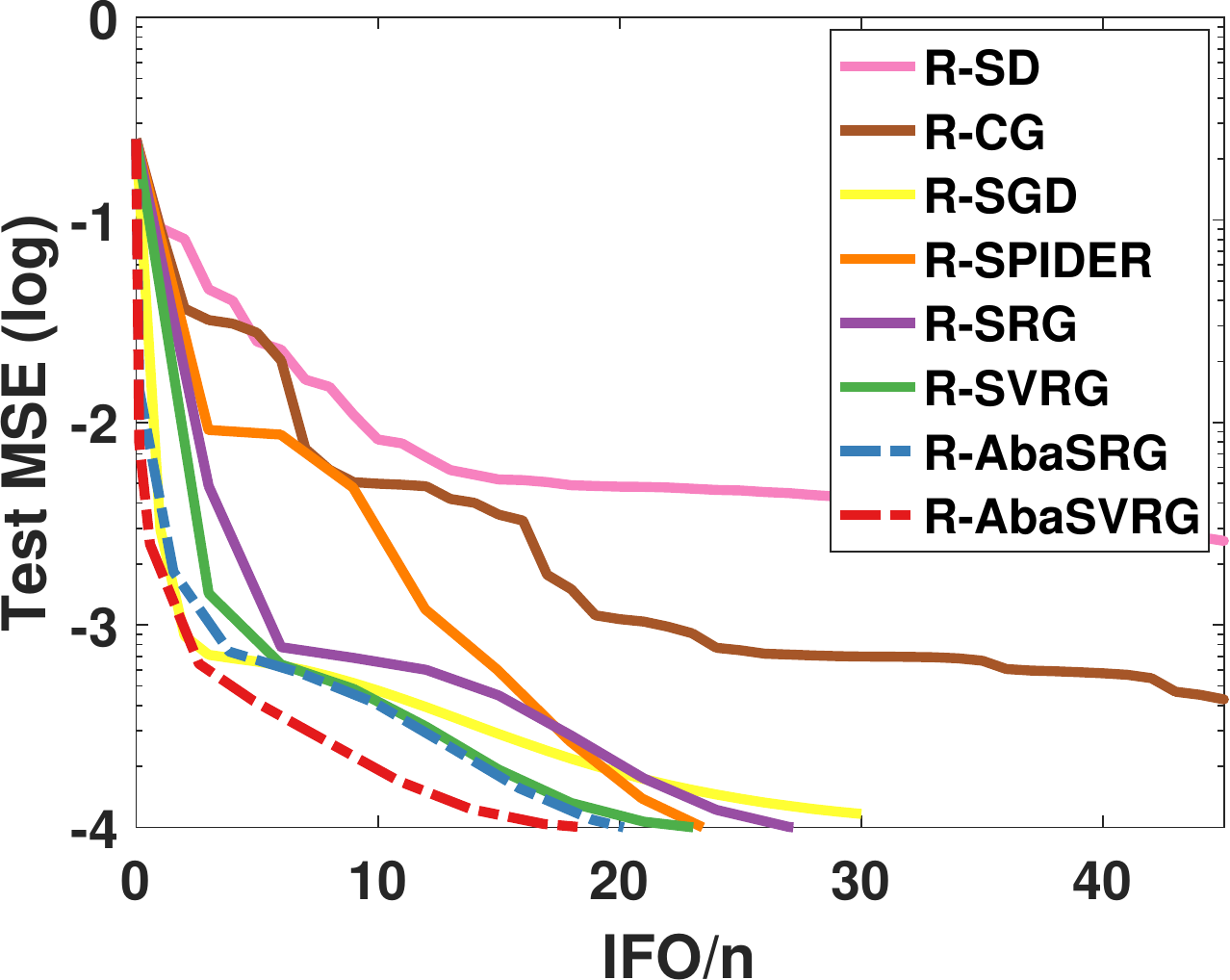}}
    \hspace{0.05in}
    \subfloat[High dimension]{\includegraphics[width = 0.28\textwidth, height = 0.21\textwidth]{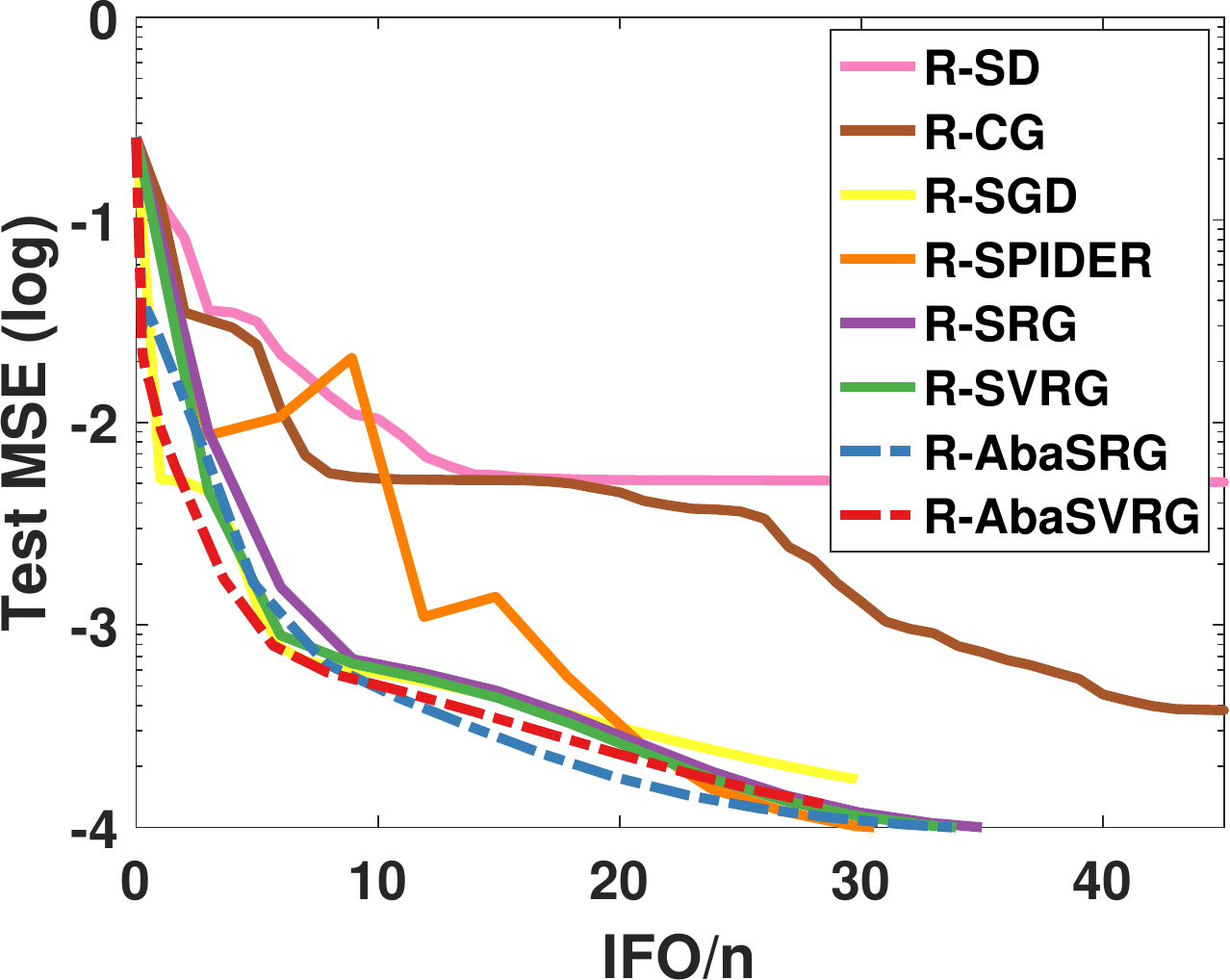}}
    \hspace{0.05in}
    \subfloat[High rank]{\includegraphics[width = 0.28\textwidth, height = 0.21\textwidth]{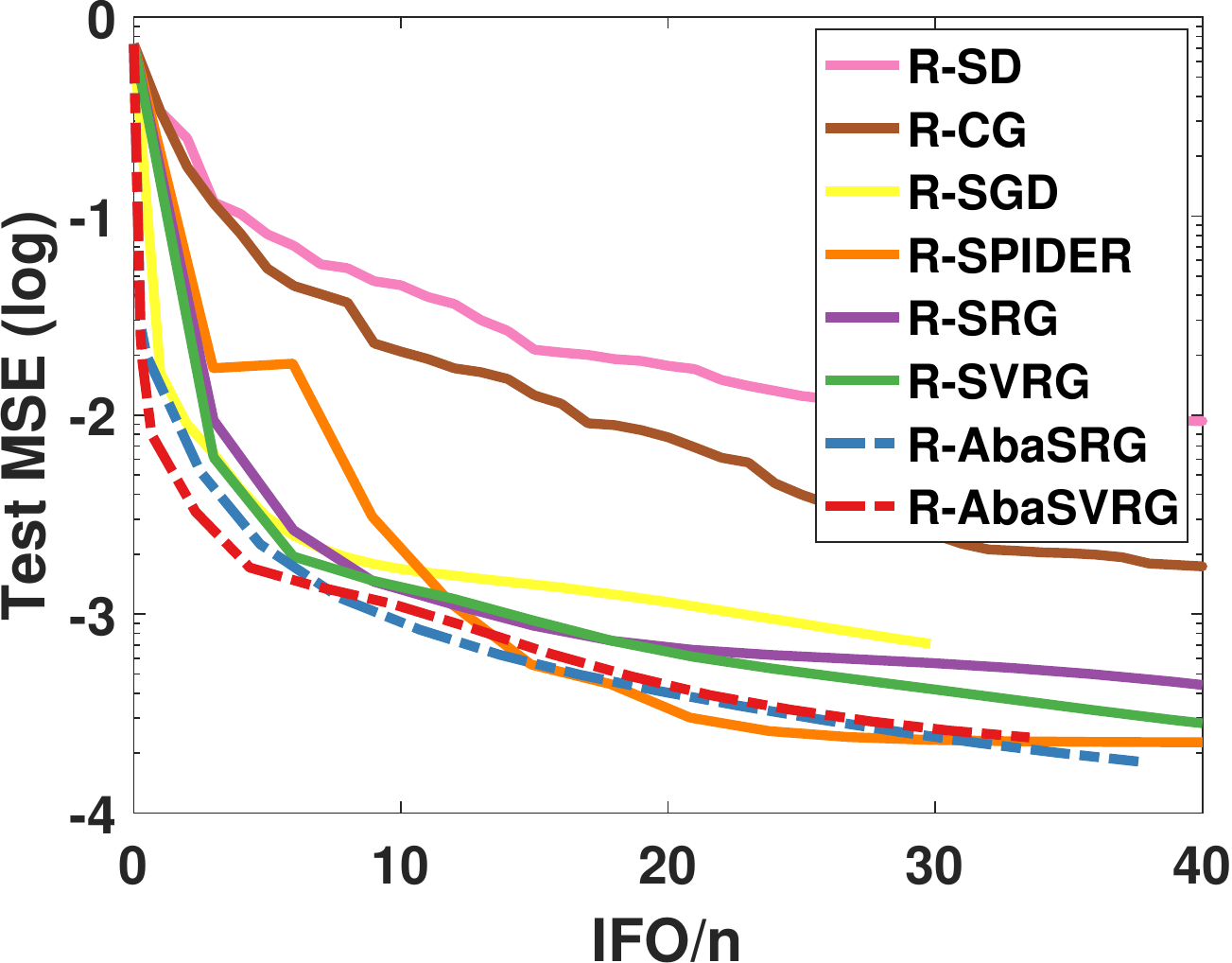}}\\    \subfloat[Ill condition]{\includegraphics[width = 0.28\textwidth, height = 0.21\textwidth]{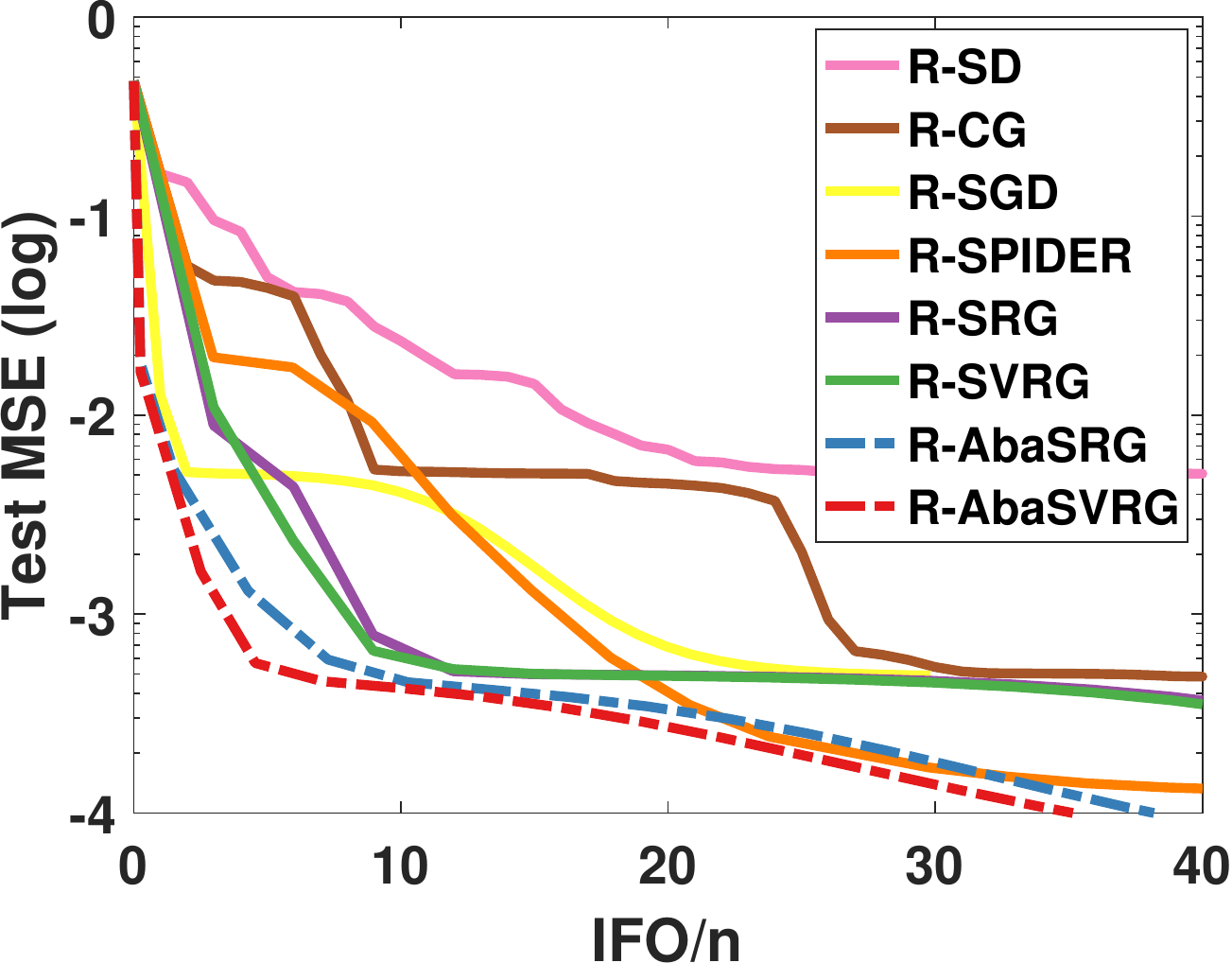}}
    \hspace{0.05in}
    \subfloat[Low sampling]{\includegraphics[width = 0.28\textwidth, height = 0.21\textwidth]{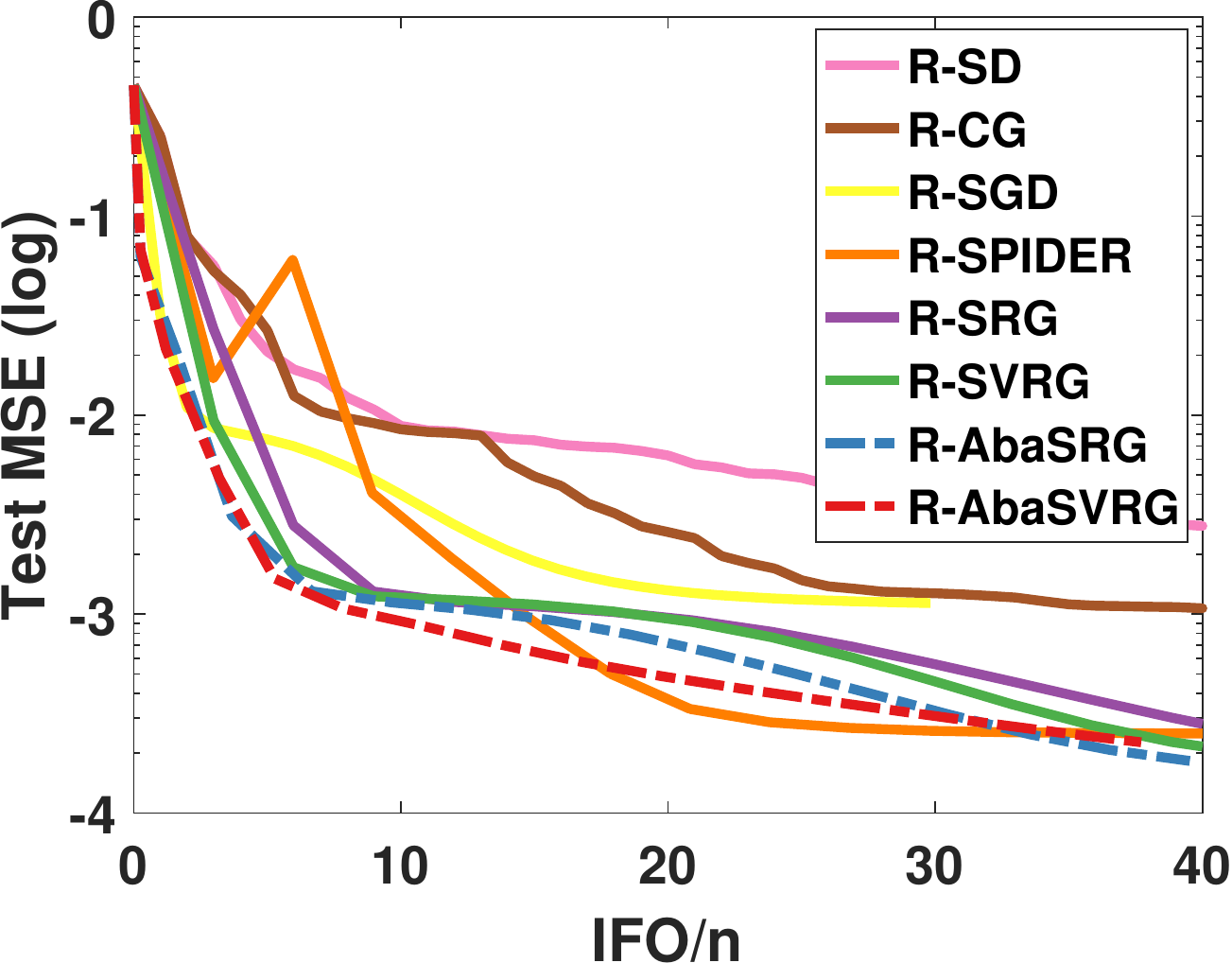}}
    \hspace{0.05in}
    \subfloat[High noise]{\includegraphics[width = 0.28\textwidth, height = 0.21\textwidth]{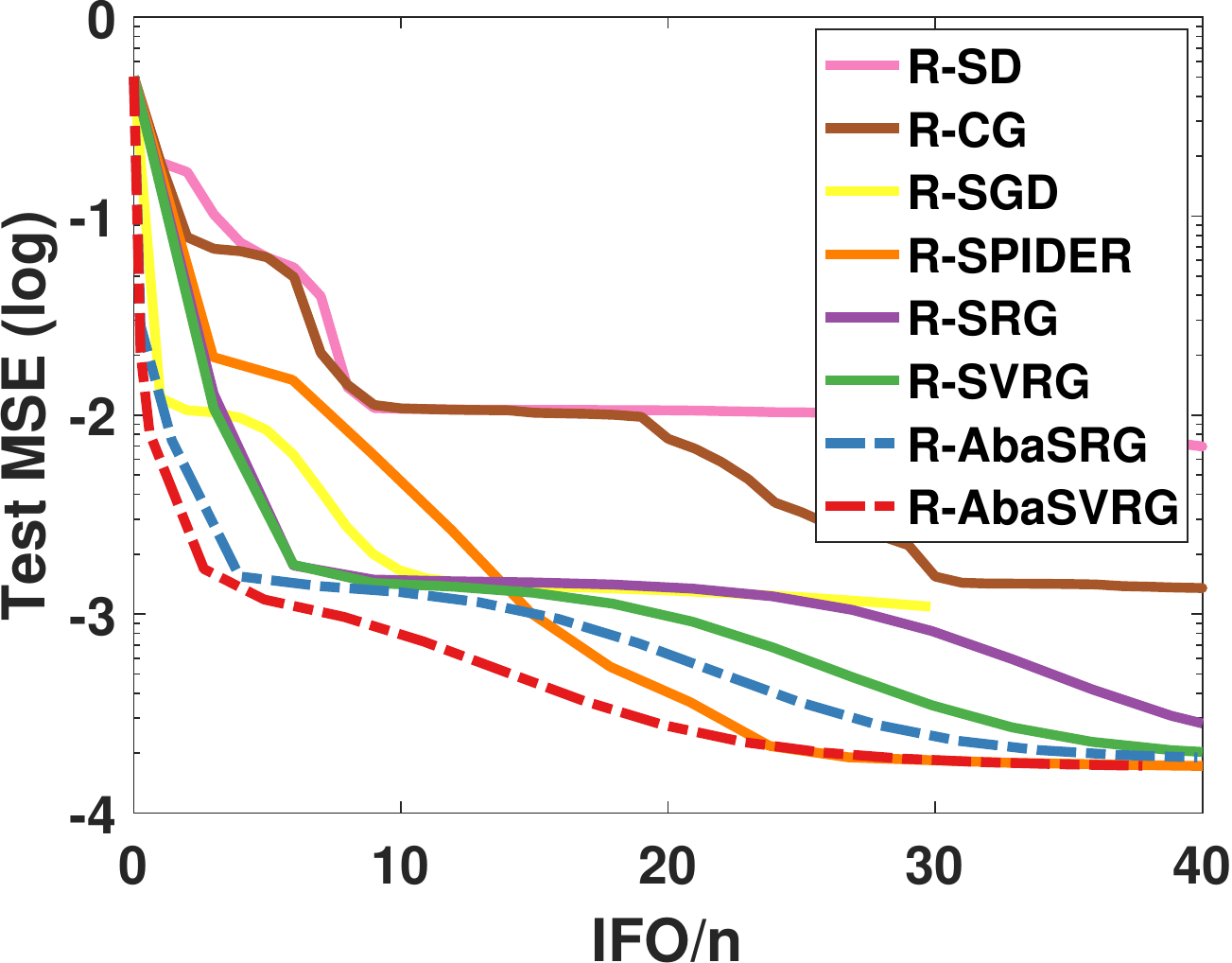}}
    \caption{LRMC results on datasets with different characteristics.}
    \label{add_LRMC_syn_charcter_appendix}
\end{figure}

\textbf{Additional results for Netflix and Movielens dataset.} We present training MSE results on Netflix and Movielens datasets accompanying test MSE results in the main text. Also, we examine sensitivity of R-AbaSVRG and R-AbaSRG to parameter $c_\beta$.

\begin{figure}[H]
\captionsetup{justification=centering}
    \centering
    \subfloat[Training MSE vs. IFO]{\includegraphics[width = 0.28\textwidth, height = 0.21\textwidth]{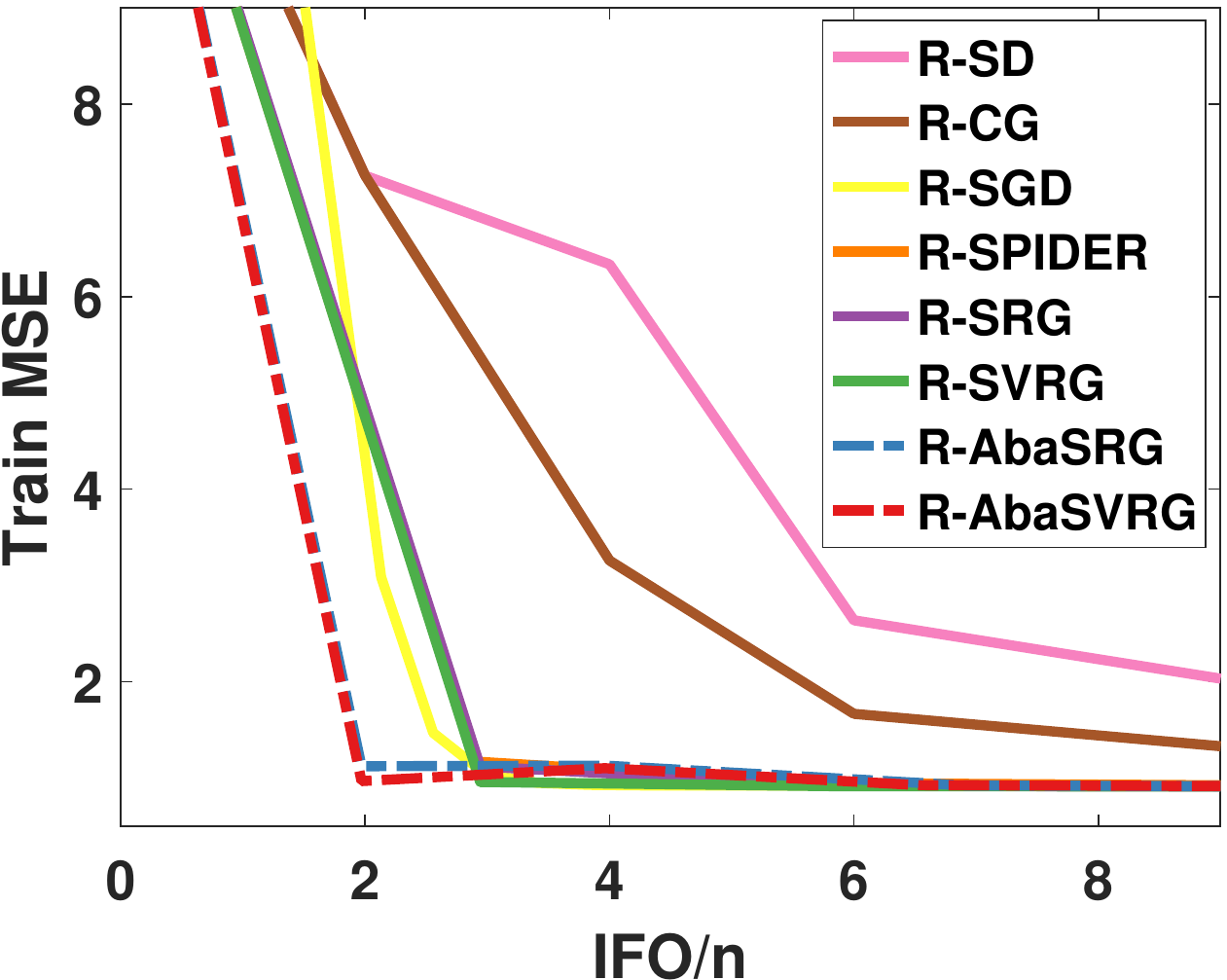}}
    \hspace{0.05in}
    \subfloat[Sensitivity of R-AbaSVRG to $c_\beta$]{\includegraphics[width = 0.28\textwidth, height = 0.21\textwidth]{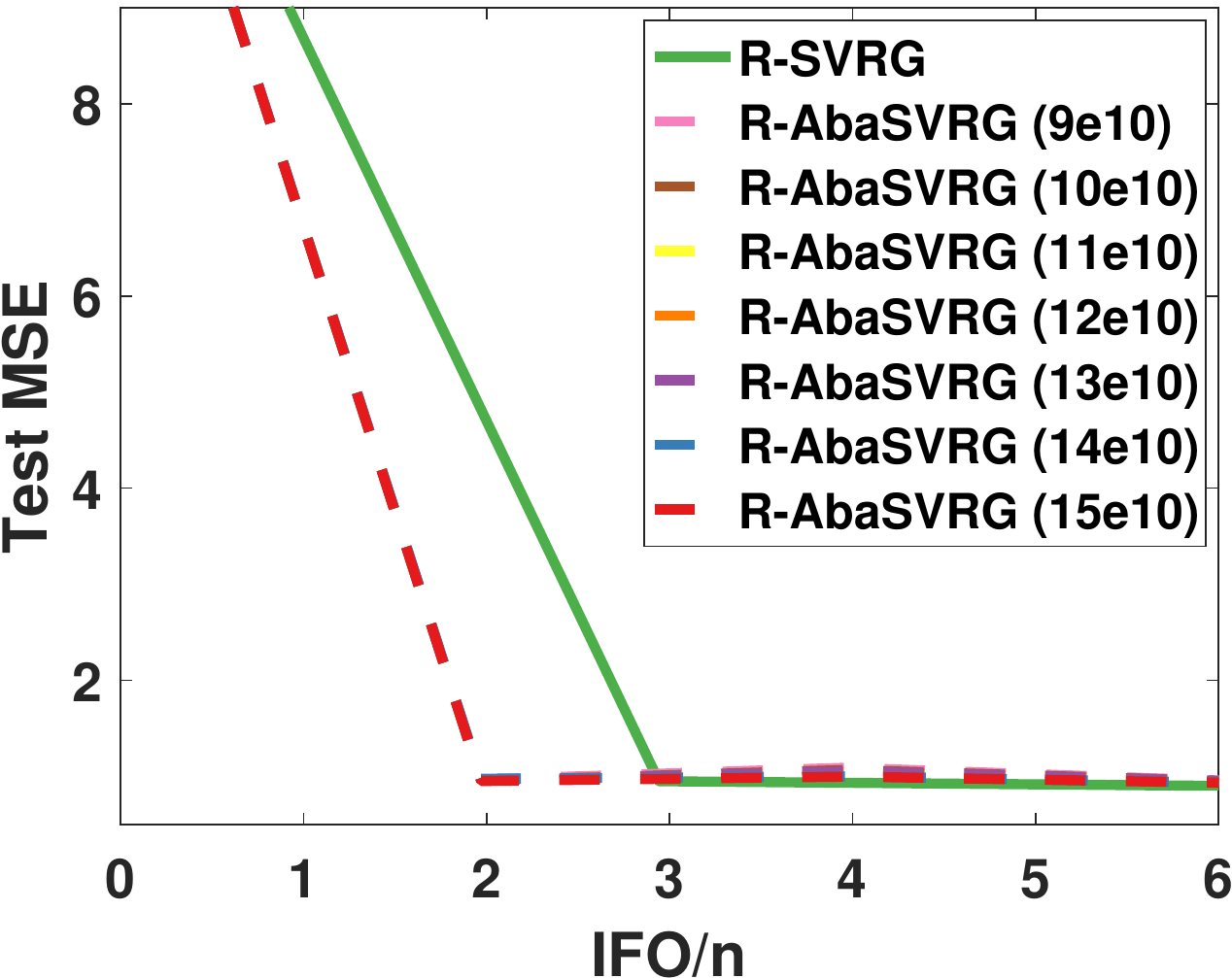}}
    \hspace{0.05in}
    \subfloat[Sensitivity of R-AbaSRG to $c_\beta$]{\includegraphics[width = 0.28\textwidth, height = 0.21\textwidth]{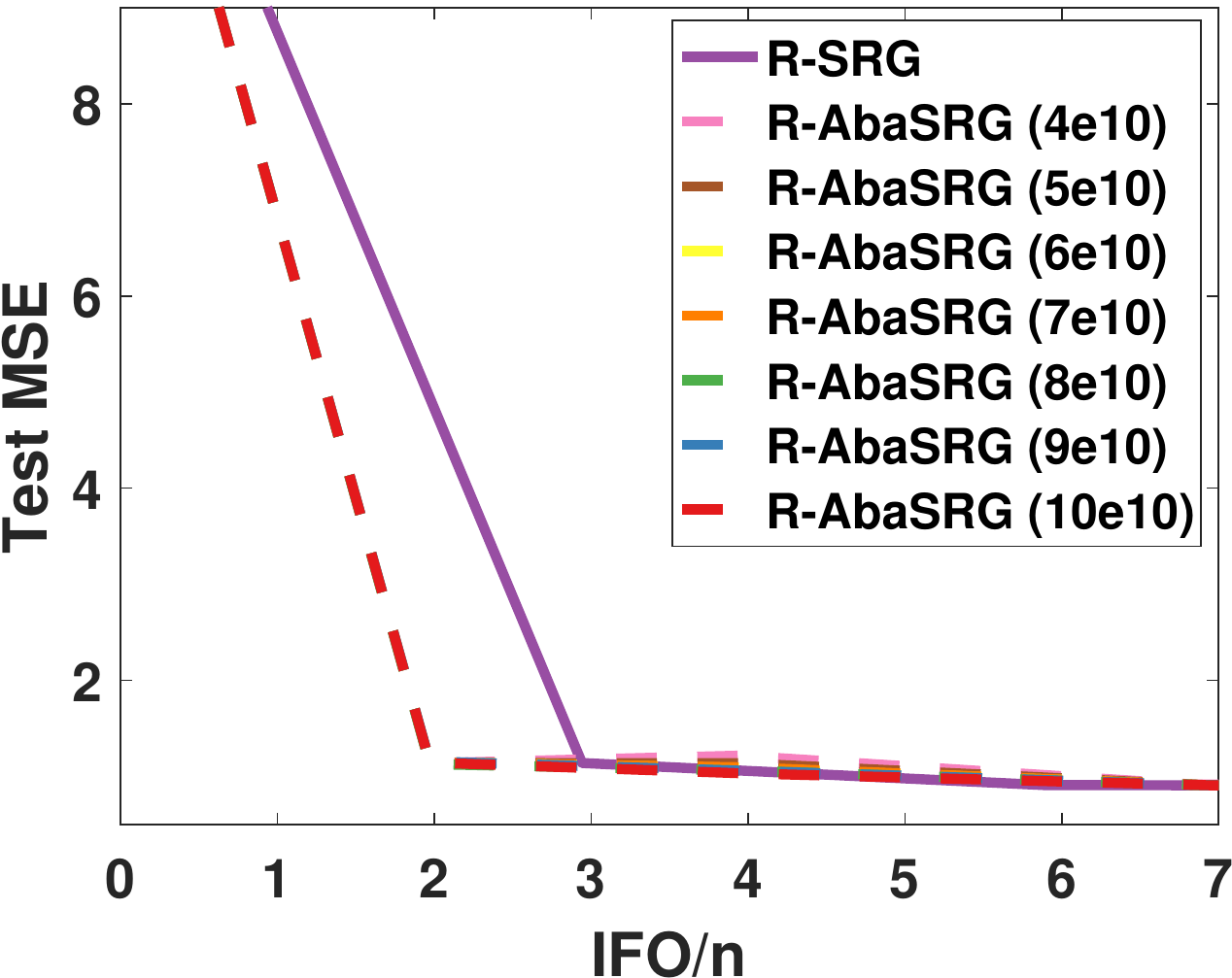}}
    \hspace{0.05in}
    \caption{Additional LRMC results on Netflix dataset.}
    \label{add_LRMC_Netflix_appendix}
\end{figure}

\begin{figure}[H]
\captionsetup{justification=centering}
    \centering
    \subfloat[Training MSE vs. IFO]{\includegraphics[width = 0.28\textwidth, height = 0.21\textwidth]{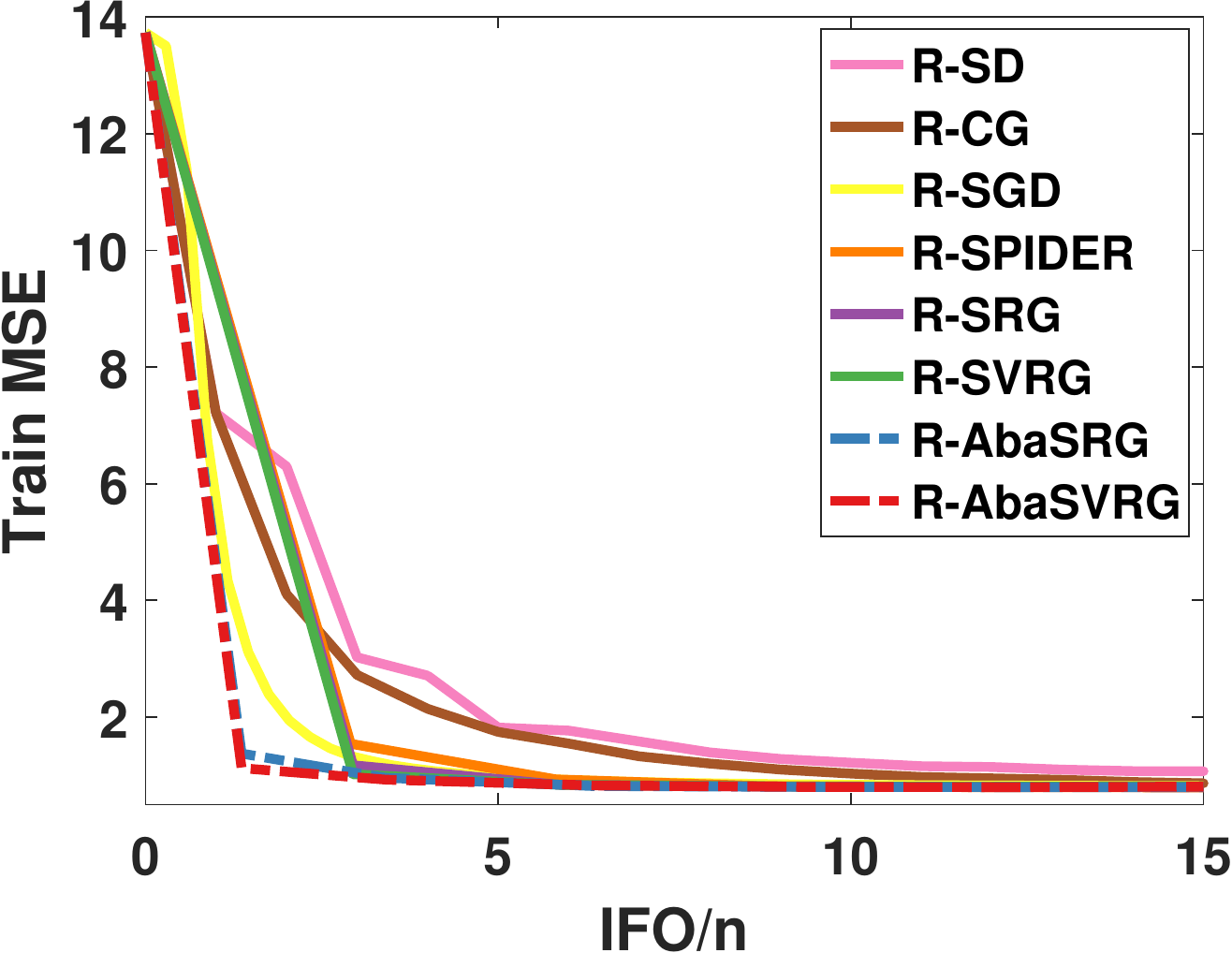}}
    \hspace{0.05in}
    \subfloat[Sensitivity of R-AbaSVRG to $c_\beta$]{\includegraphics[width = 0.28\textwidth, height = 0.21\textwidth]{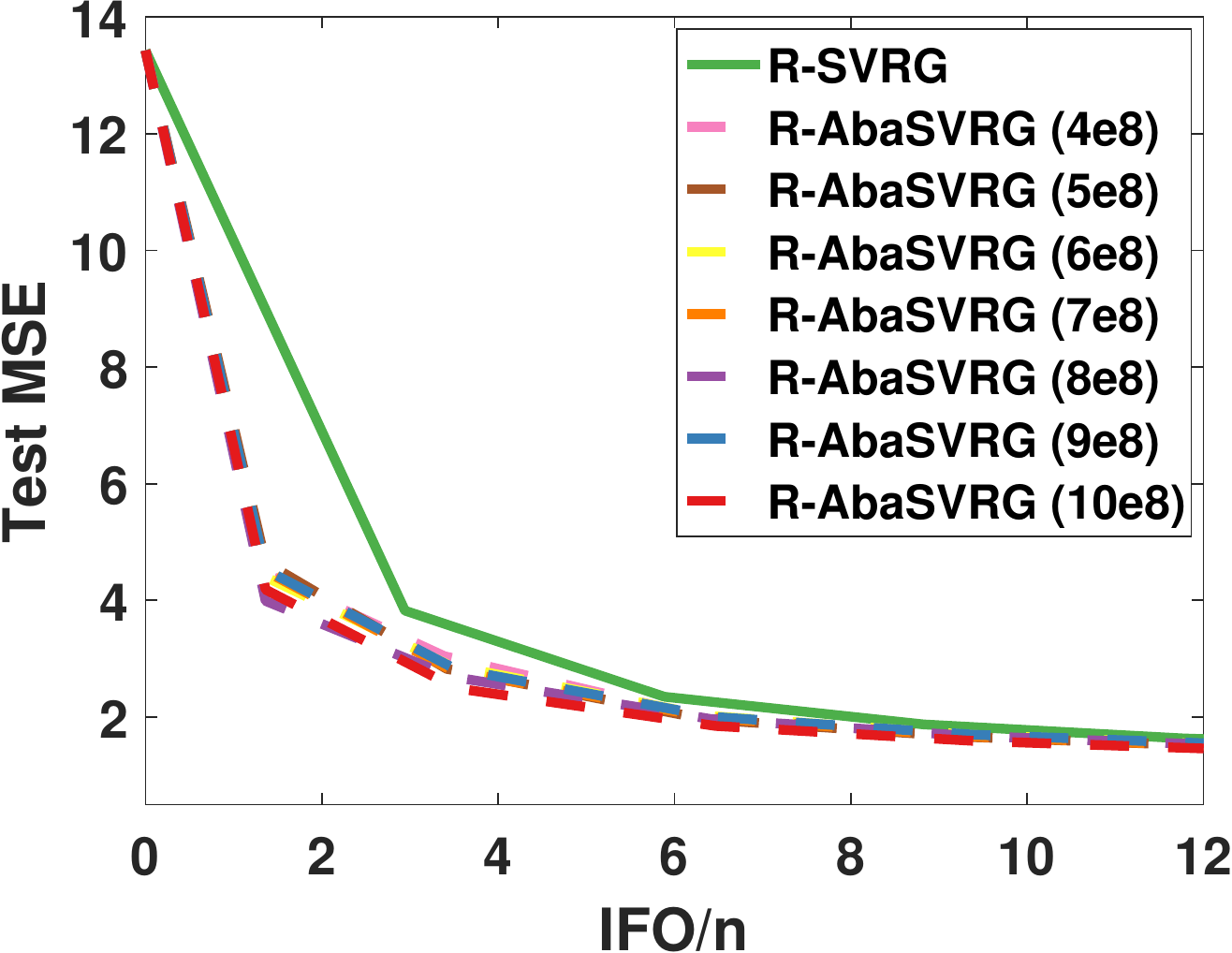}}
    \hspace{0.05in}
    \subfloat[Sensitivity of R-AbaSRG to $c_\beta$]{\includegraphics[width = 0.28\textwidth, height = 0.21\textwidth]{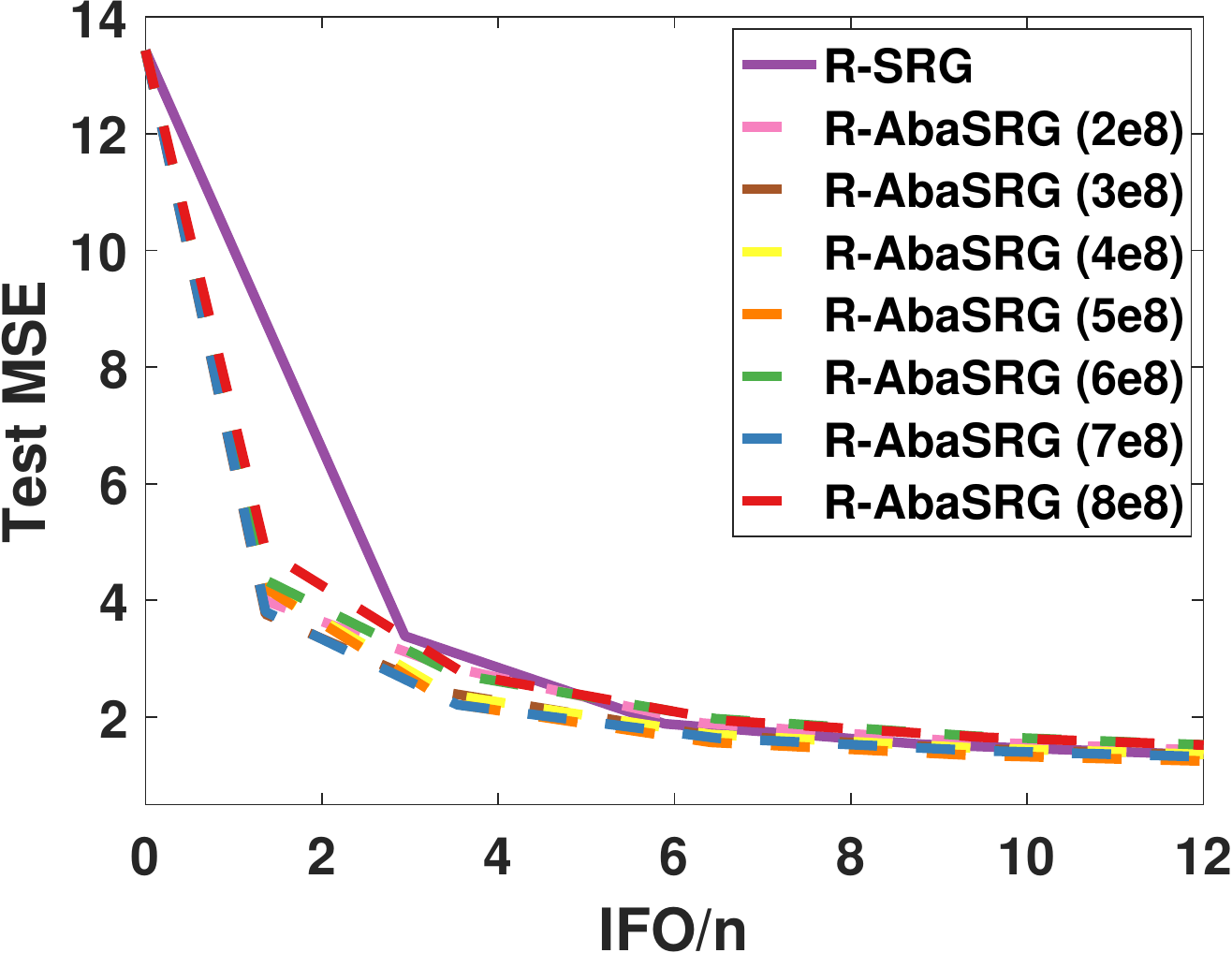}}
    \hspace{0.05in}
    \caption{Additional LRMC results on Movielens dataset.}
    \label{add_LRMC_movielens_appendix}
\end{figure}

\textbf{Additonal results on Jester dataset.} We also consider Jester dataset \cite{Goldbergjester2001} that contains continuous ratings in $[-10, 10]$ from $24983$ ($d$) users on $100$ jokes ($n$). We extract $10$ ratings per user as test set. We choose $q = -6, l = 10$.

\begin{figure}[H]
\captionsetup{justification=centering}
    \centering
    \subfloat[Test MSE vs. IFO]{\includegraphics[width = 0.28\textwidth, height = 0.21\textwidth]{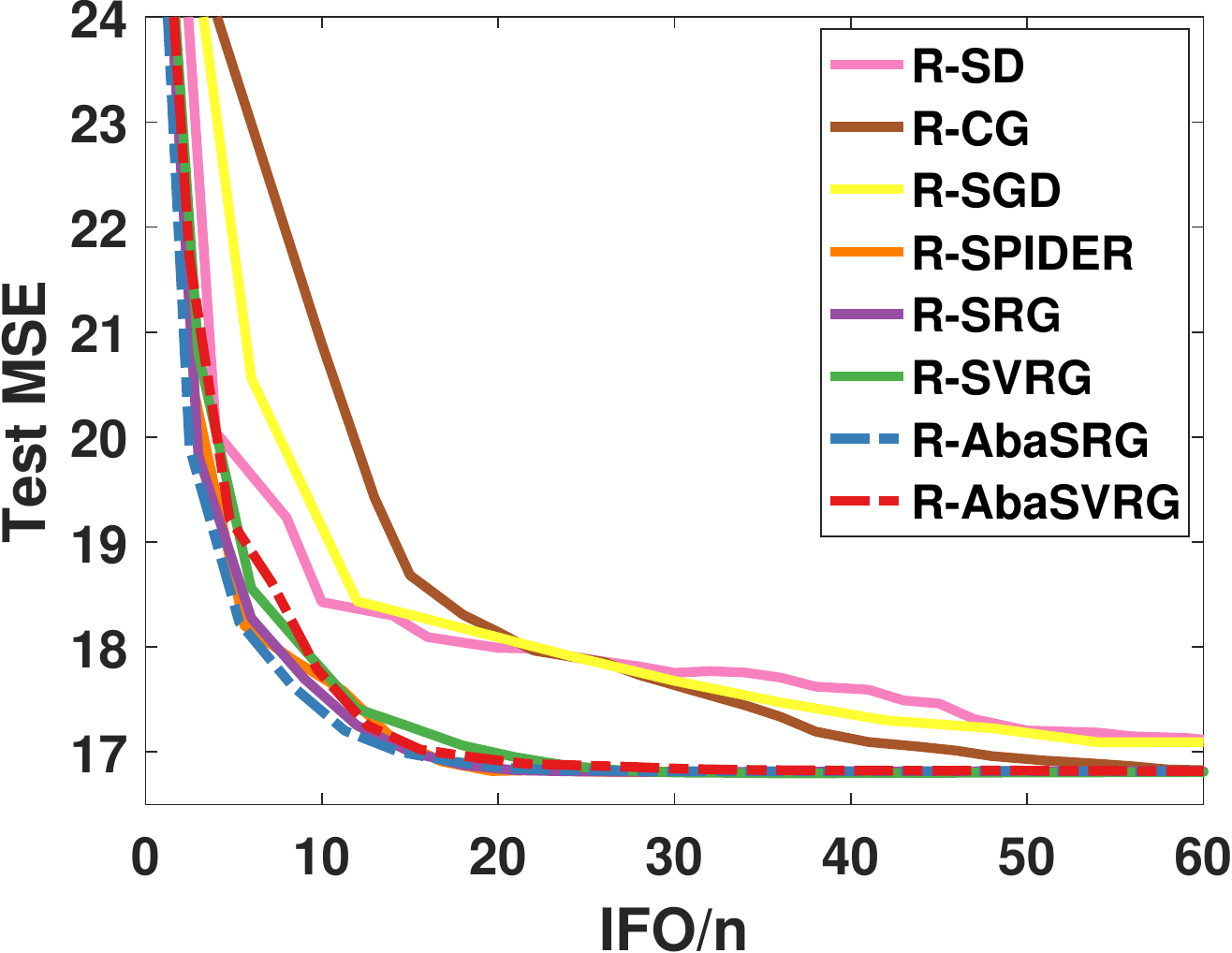}}
    \hspace{0.05in}
    \subfloat[Sensitivity of R-AbaSVRG to $c_\beta$]{\includegraphics[width = 0.28\textwidth, height = 0.21\textwidth]{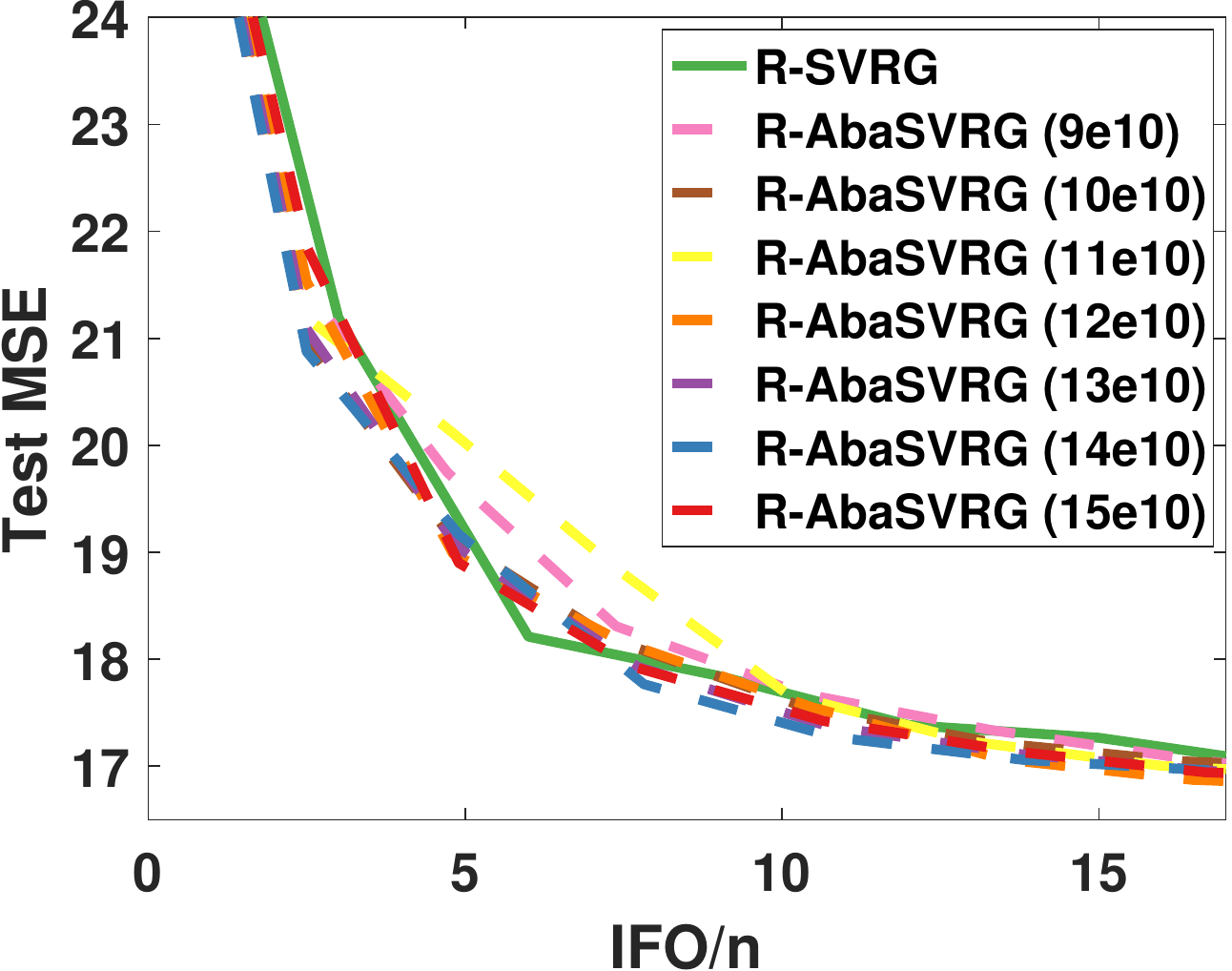}}
    \hspace{0.05in}
    \subfloat[Sensitivity of R-AbaSRG to $c_\beta$]{\includegraphics[width = 0.28\textwidth, height = 0.21\textwidth]{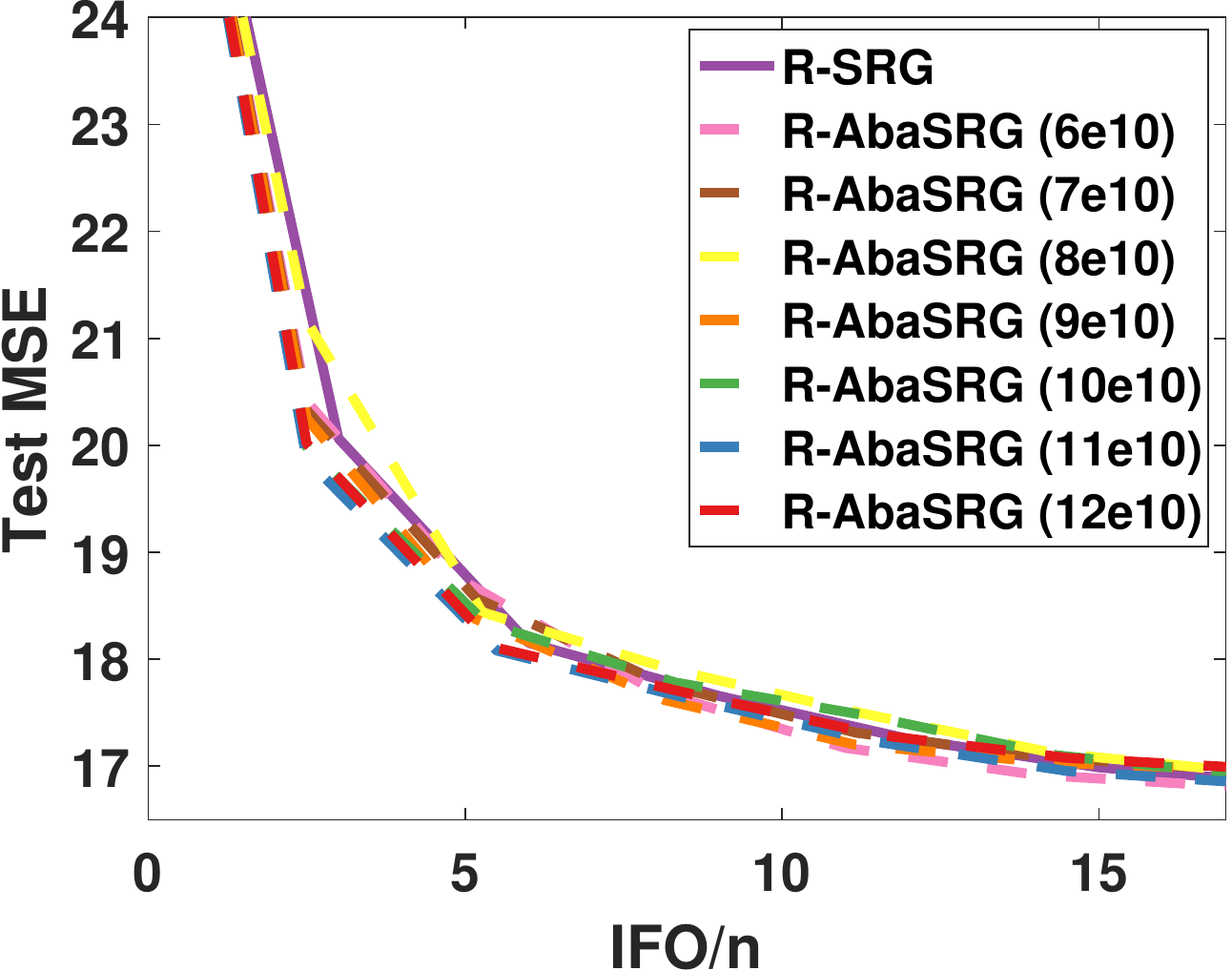}}
    \caption{LRMC results on Jester dataset.}
    \label{Jester_LRMC_appendix}
\end{figure}

\subsection{RKM on SPD manifold}
\textbf{Additional results on synthetic datasets.} Similar to PCA and LRMC, result sensitivity on baseline synthetic dataset with $(n, d, \text{cn}) = (5000, 10, 20)$ is evaluated by presenting three independent results in Fig. \ref{add_RKM_syn_run_appendix}. We also evaluate algorithms on datasets with large samples $n = 10000$, with high dimension $d = 30$ and with high condition number $\text{cn} = 50$. Optimality gap results are presented in Fig. \ref{add_RKM_syn_charcter_appendix}.

\begin{figure}[H]
\captionsetup{justification=centering}
    \centering
    \subfloat[Run 1]{\includegraphics[width = 0.28\textwidth, height = 0.21\textwidth]{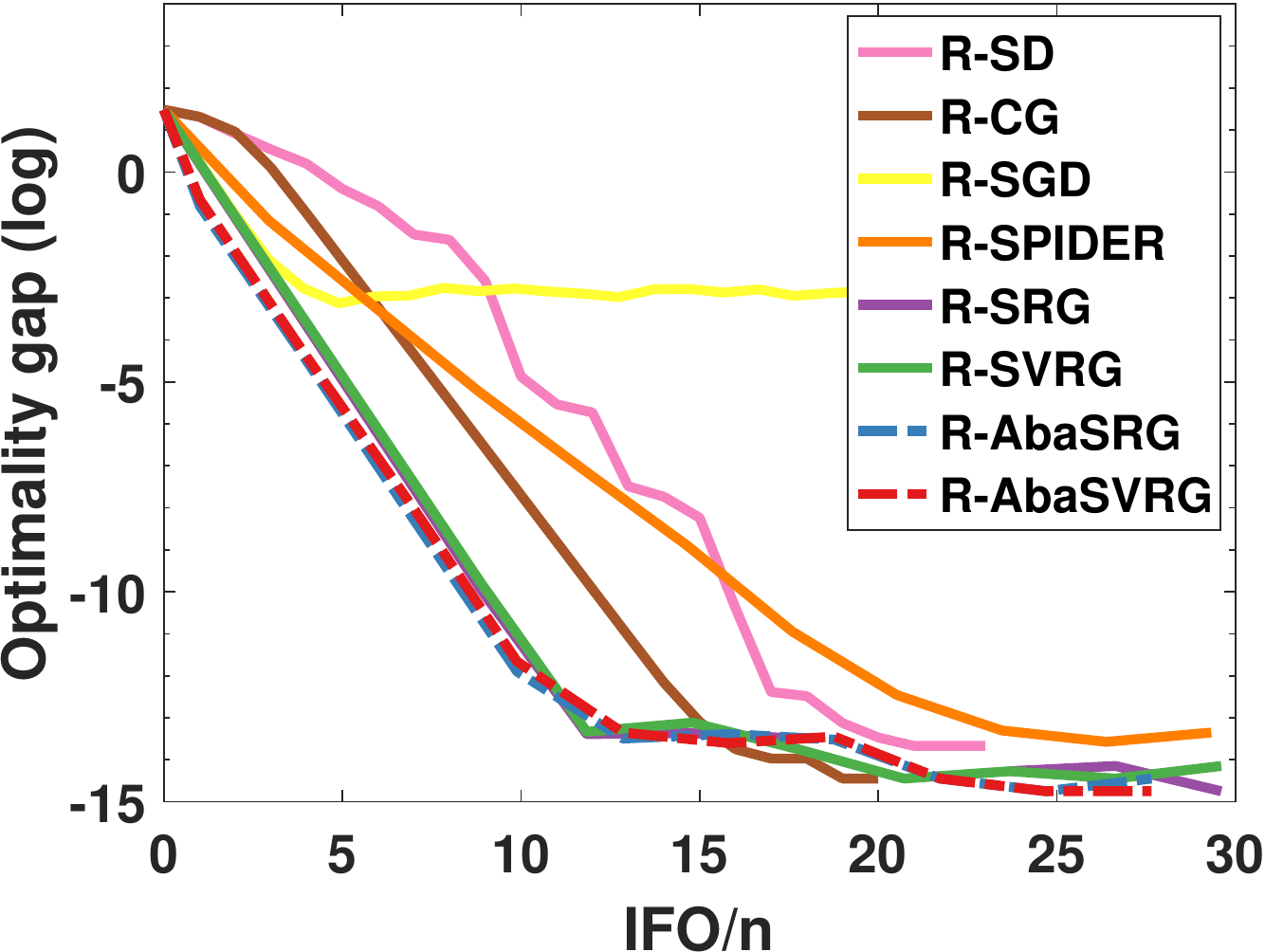}}
    \hspace{0.05in}
    \subfloat[Run 2]{\includegraphics[width = 0.28\textwidth, height = 0.21\textwidth]{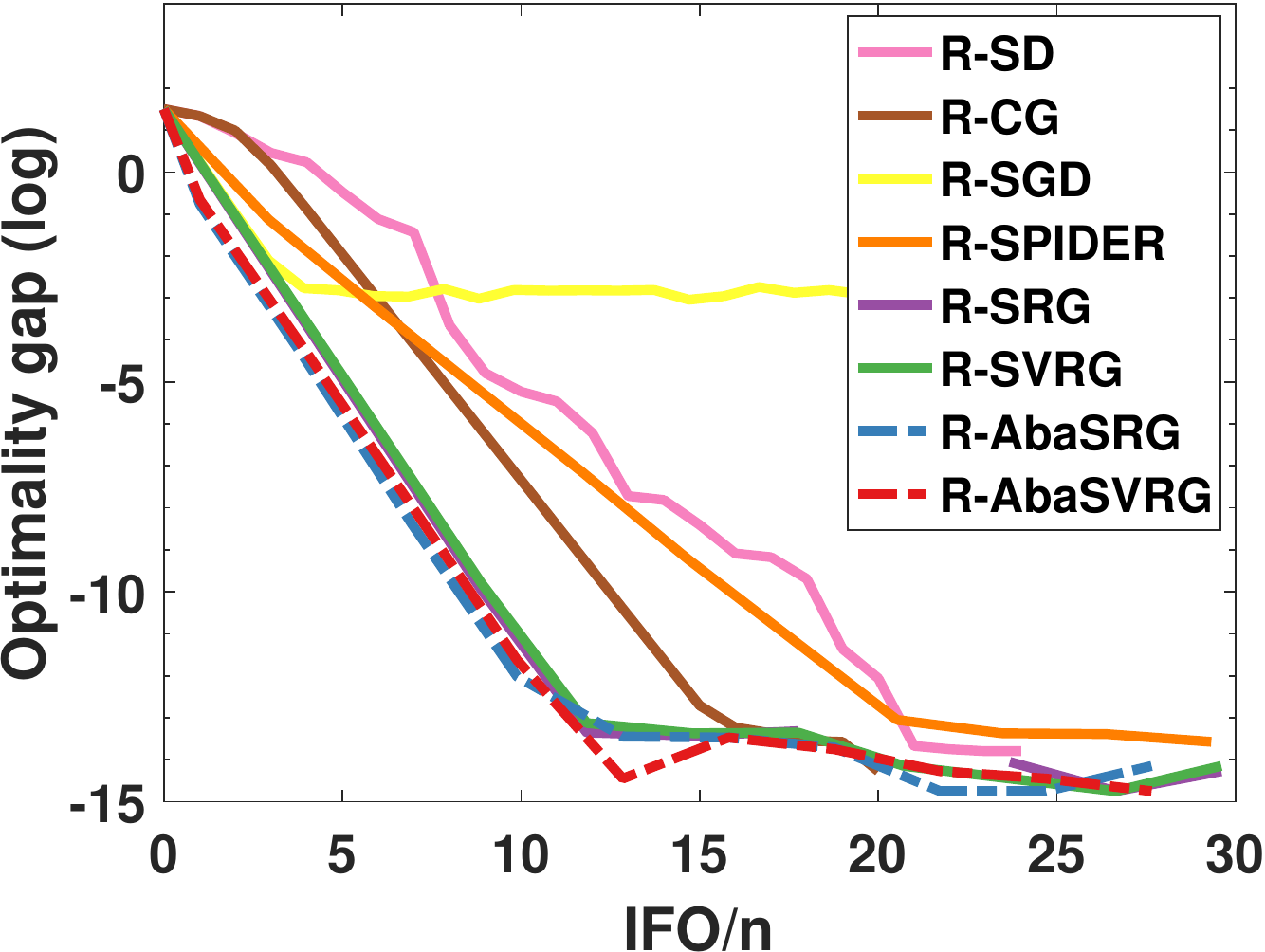}}
    \hspace{0.05in}
    \subfloat[Run 3]{\includegraphics[width = 0.28\textwidth, height = 0.21\textwidth]{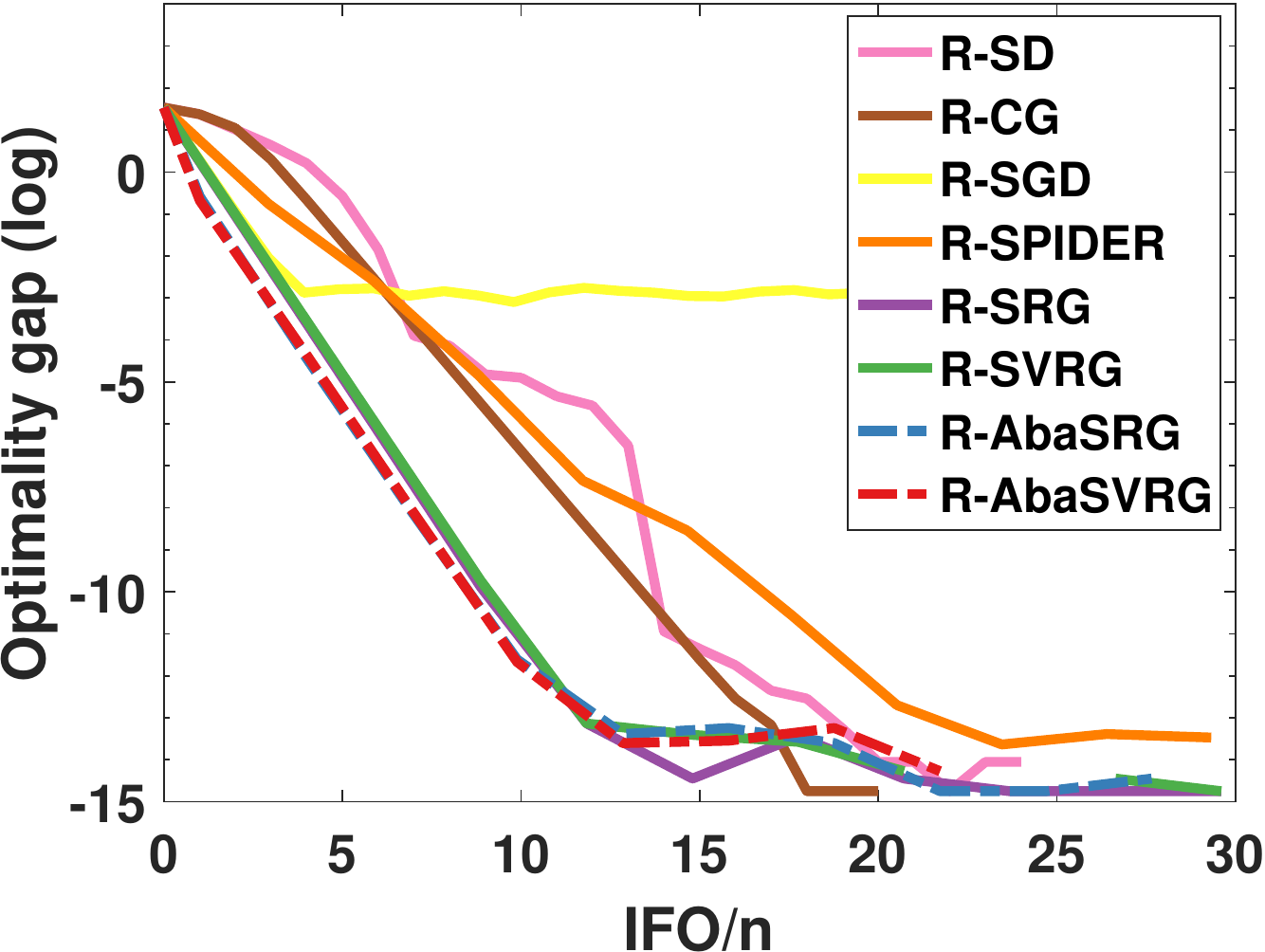}}
    \hspace{0.05in}
    \caption{RKM Result sensitivity on baseline synthetic dataset.}
    \label{add_RKM_syn_run_appendix}
\end{figure}

\begin{figure}[H]
\captionsetup{justification=centering}
    \centering
    \subfloat[Large scale]{\includegraphics[width = 0.28\textwidth, height = 0.21\textwidth]{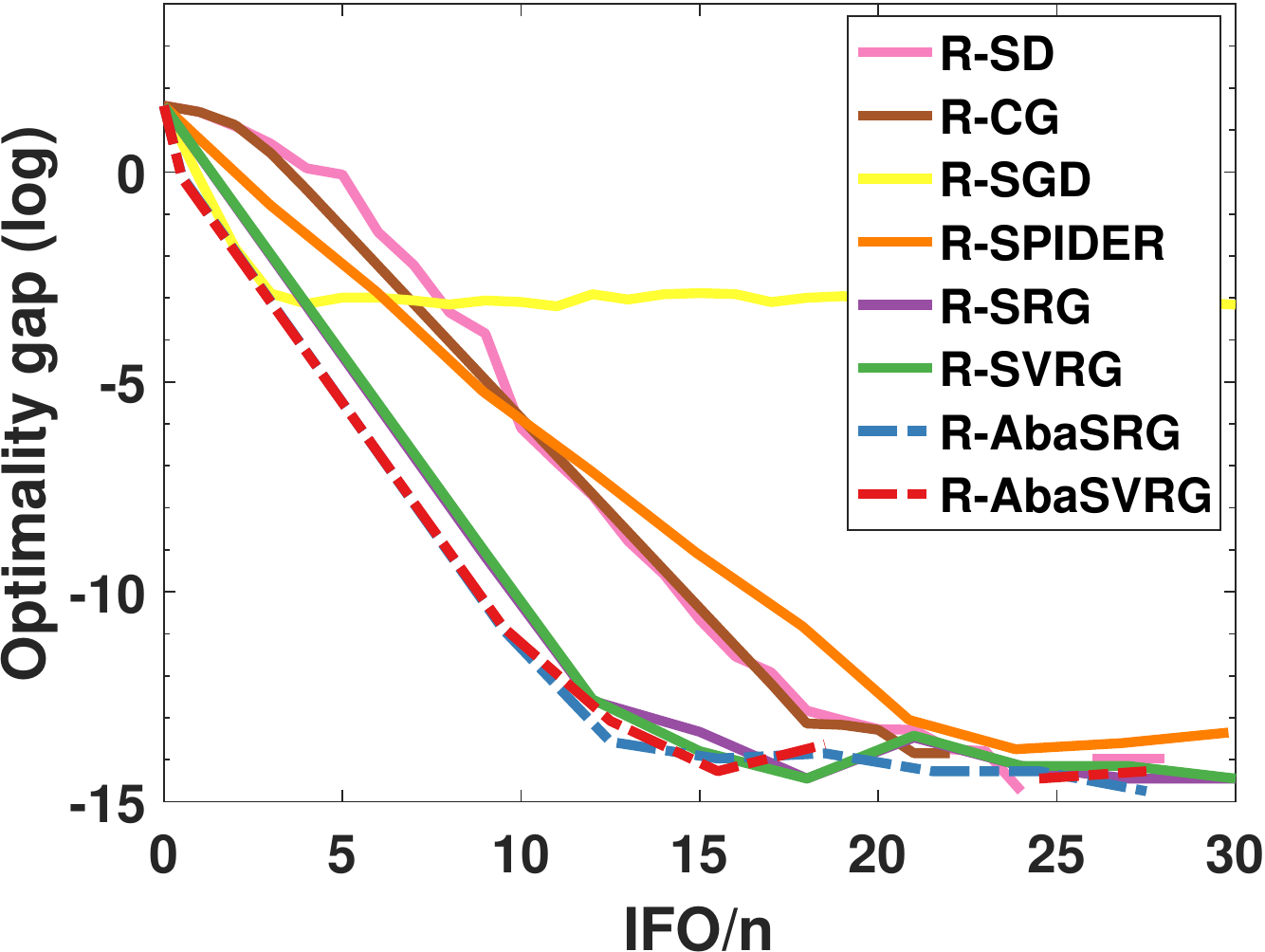}}
    \hspace{0.05in}
    \subfloat[High dimension]{\includegraphics[width = 0.28\textwidth, height = 0.21\textwidth]{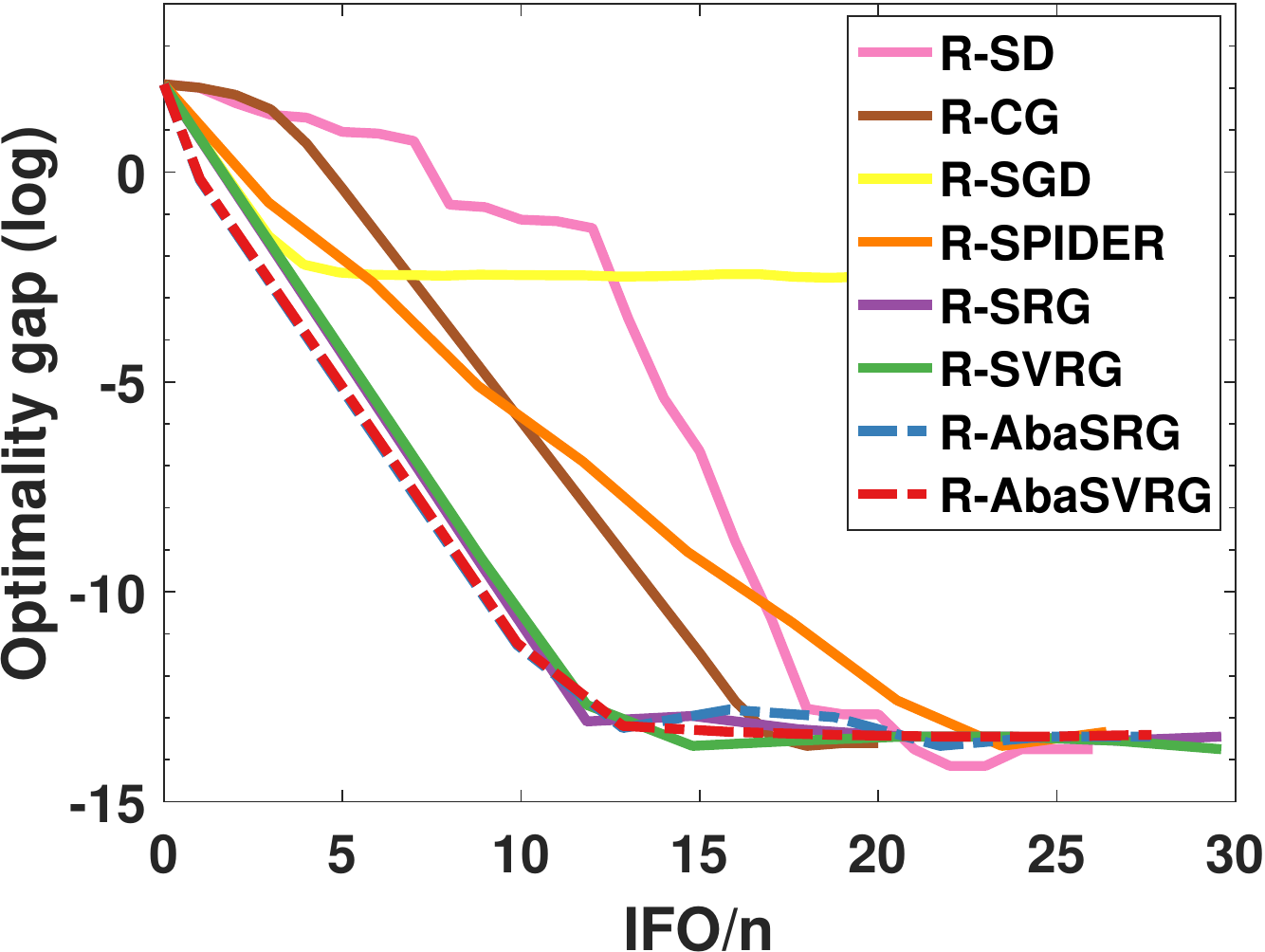}}
    \hspace{0.05in}
    \subfloat[Ill condition]{\includegraphics[width = 0.28\textwidth, height = 0.21\textwidth]{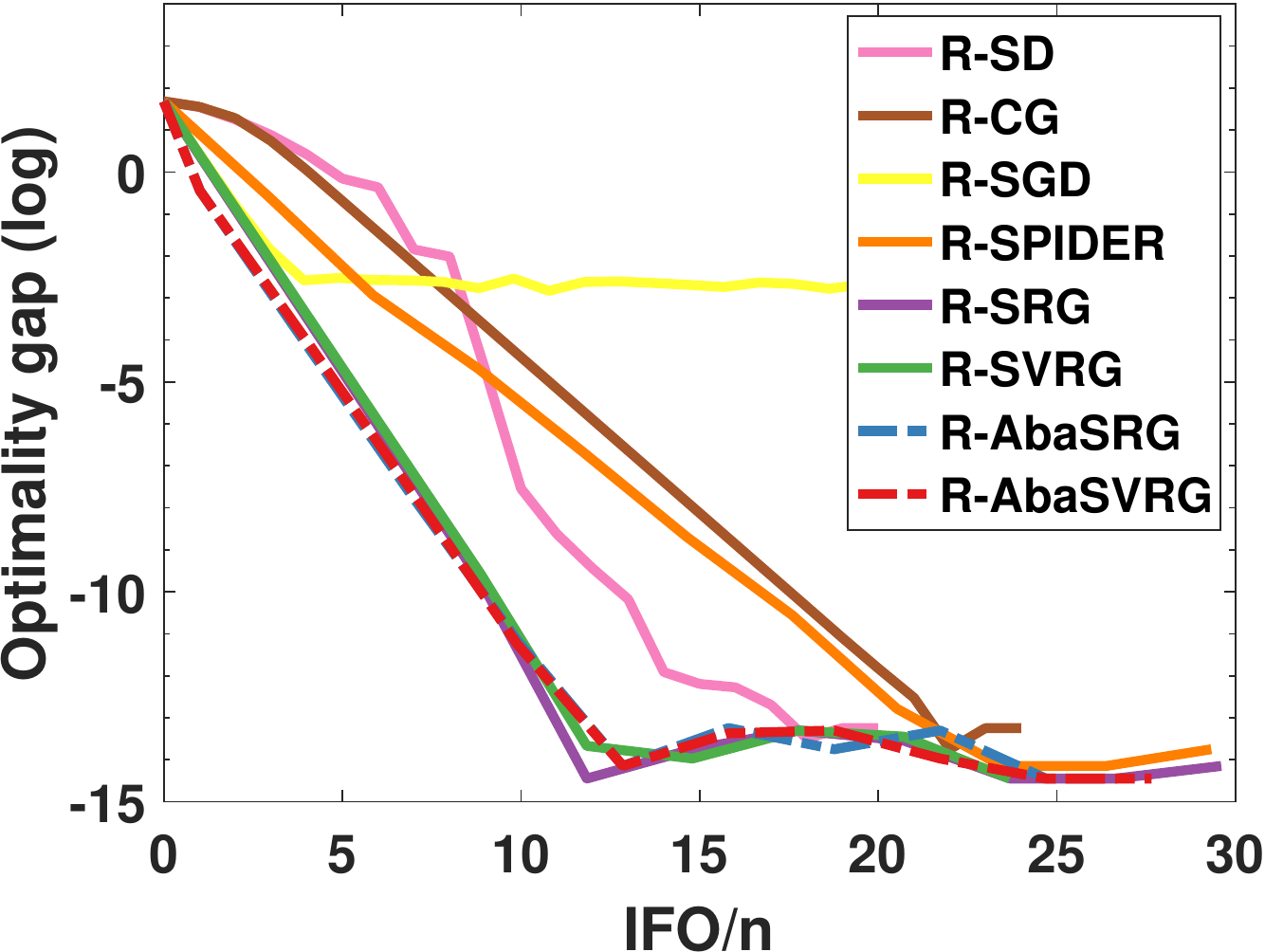}}
    \hspace{0.05in}
    \caption{RKM Result on datasets with different characteristics.}
    \label{add_RKM_syn_charcter_appendix}
\end{figure}

\end{document}